\documentclass[12pt,a4paper,reqno]{amsart}
\usepackage{amsmath,amsthm,amssymb,mathtools,enumitem,tikz,appendix,listings,mathrsfs,multirow,float,setspace,subcaption}
\usepackage[misc]{ifsym}
\usepackage{a4wide}
\usepackage{color}
\usepackage[foot]{amsaddr}
\usepackage[margin=1in]{geometry}

\usepackage[utf8]{inputenc}

\usepackage[pagebackref,colorlinks=true,citecolor=blue,linkcolor=red,urlcolor=blue]{hyperref}

\title[\MakeLowercase{$p$-groups related to exceptional} C\MakeLowercase{hevalley groups}]{O\MakeLowercase{n $p$-groups with automorphism groups}\\ \MakeLowercase{related to the exceptional }C\MakeLowercase{hevalley groups}}

\author{Saul D. Freedman}
\address{\parbox{\linewidth}{ %
Centre for the Mathematics of Symmetry and Computation,\\
The University of Western Australia, 35 Stirling Highway,\\
Crawley, WA 6009, Australia\\
\textit{Current affiliation}: School of Mathematics and Statistics, University of St Andrews, UK}}
\email{sdf8@st-andrews.ac.uk}

\date{\today}
\thanks{\textit{\keywordsname}. $p$-group, exterior square, Lie power, exceptional Chevalley group}

\newtheorem{thm}{Theorem}[section]
\newtheorem*{thm*}{Theorem}
\newtheorem*{conj*}{Conjecture}

\newtheorem{conj}[thm]{Conjecture}
\newtheorem{lem}[thm]{Lemma}
\newtheorem*{lem*}{Lemma}
\newtheorem{prop}[thm]{Proposition}

\theoremstyle{remark}

\theoremstyle{definition}
\newtheorem{defn}[thm]{Definition}
\newtheorem{rem}[thm]{Remark}
\newcounter{claim}[thm]

\newenvironment{nscenter}
 {\parskip=0pt\par\nopagebreak\centering}
 {\par\noindent\ignorespacesafterend}

\makeatletter
\newcommand{\slteq}{%
  \mathrel{\mathpalette\sl@unlhd\relax}%
}

\newcommand{\sl@unlhd}[2]{%
  \sbox\z@{$#1\lhd$}%
  \sbox\tw@{$#1\leqslant$}%
  \dimen@=\ht\tw@
  \advance\dimen@-\ht\z@
  \ifx#1\displaystyle
    \advance\dimen@ .2pt
  \else
    \ifx#1\textstyle
      \advance\dimen@ .2pt
    \fi
  \fi
  \ooalign{\raisebox{\dimen@}{$\m@th#1\lhd$}\cr$\m@th#1\leqslant$\cr}%
}
\makeatother

\newcommand{\nonsplit}[2]{#1\raisebox{0.6ex}{$\cdot$} #2}

\renewcommand{\le}{\leqslant}
\renewcommand{\ge}{\geqslant}
\renewcommand{\trianglelefteq}{\slteq}

\newcommand{\hinfbar}{\overline{\vphantom{H^\infty}\smash{H^\infty}}}
\newcommand{\zgl}{Z_{\mathrm{GL}}}

\begin{document}

\doublespacing

\begin{abstract}
Let $\hat G$ be the finite simply connected version of an exceptional Chevalley group, and let $V$ be a nontrivial irreducible module, of minimal dimension, for $\hat G$ over its field of definition. We explore the overgroup structure of $\hat G$ in $\mathrm{GL}(V)$, and the submodule structure of the exterior square (and sometimes the third Lie power) of $V$. When $\hat G$ is defined over a field of odd prime order $p$, this allows us to construct the smallest (with respect to certain properties) $p$-groups $P$ such that the group induced by $\mathrm{Aut}(P)$ on $P/\Phi(P)$ is either $\hat G$ or its normaliser in $\mathrm{GL}(V)$.
\end{abstract}

\maketitle

\section{Introduction}
\label{sec:intro}

Let $P$ be a finite $p$-group, and let $A(P)$ denote the group induced by $\mathrm{Aut}(P)$ on the \emph{Frattini quotient} $P/\Phi(P)$ of $P$. More explicitly, $A(P)$ is the image of the natural homomorphism from $\mathrm{Aut}(P)$ to $\mathrm{Aut}(P/\Phi(P))$. By Burnside's Basis Theorem, we can identify this Frattini quotient with the vector space $\mathbb{F}_p^d$, where $d$ is the \emph{rank} of $P$, i.e., the minimum size of a generating set for $P$. Thus we can identify $A(P)$ naturally with a subgroup of the general linear group $\mathrm{GL}(d,p)$.

Observe that if $P$ is the elementary abelian $p$-group of rank $d$, then $A(P) = \mathrm{GL}(d,p)$. Thus in order to induce a proper subgroup of $\mathrm{GL}(d,p)$ on a $p$-group, this $p$-group must either have nilpotency class at least $2$ or exponent at least $p^2$, i.e., exponent-$p$ class at least $2$. Recall that the nilpotency class of a $p$-group $P$ is the length of the lower central series of $P$, defined by $\gamma_1(P):=P$ and $\gamma_{i+1}(P):=[\gamma_i(P),P]$ for each $i \ge 1$. In particular, $\gamma_2(P)$ is the derived subgroup $P'$ of $P$. Similarly, the exponent-$p$ class of $P$ is the length of the lower exponent-$p$ central series of $P$, defined by $\mathcal{P}_1(P) := P$ and $\mathcal{P}_{i+1}(P) := [\mathcal{P}_i(P),P]\mathcal{P}_i(P)^p$ for each $i \ge 1$. Here, if $G$ is a group, then we write $G^p$ to denote the characteristic subgroup $\langle x^p \mid x \in G \rangle$ of $G$.

We now summarise some known results about the group $A(P)$. First and foremost, Bryant and Kov\'acs \cite{bryant} showed that any given subgroup $H$ of $\mathrm{GL}(d,p)$, with $d > 1$, can be induced on the Frattini quotient of some $p$-group. However, the $p$-group in their proof of this result has an exponent-$p$ class comparable to $|\mathrm{GL}(d,p)|$ (depending on both $d$ and $p$), and hence a huge order. It is therefore natural to ask when these properties of the $p$-group can in fact be relatively small. With this question in mind, Bamberg, Glasby, Morgan and Niemeyer \cite{bamberg} investigated the case where $H$ is maximal in $\mathrm{GL}(d,p)$. They showed that if $p > 3$, then $H$ can be induced on a $p$-group of nilpotency class (and exponent-$p$ class) $2$, $3$ or $4$, exponent $p$, and order at most $p^{d^4/2}$, as long as $H$ does not contain $\mathrm{SL}(d,p)$ and lies in a certain subset of the \emph{Aschbacher classes} of $\mathrm{GL}(d,p)$. The Aschbacher classes of classical groups describe their maximal subgroups, via Aschbacher's Theorem \cite{aschbacherthm}; we will define these classes later in this paper. 

Among the maximal subgroups considered by Bamberg et al.~are the full groups of similarities in $\mathrm{GL}(d,p)$ of bilinear forms on $\mathbb{F}_p^d$. Together with $\mathrm{GL}(d,p)$ itself, these groups are the normalisers in $\mathrm{GL}(d,p)$ of the universal covers (or particular quotients of these universal covers in finitely many cases) of the classical Chevalley groups of dimension $d$ defined over $\mathbb{F}_p$. In \cite{sdfg2}, Bamberg, Freedman and Morgan began a programme to explore the exceptional groups of Lie type in a similar way. In particular, for each odd prime $p$, they constructed a $p$-group $P$ of order $p^{14}$, nilpotency class $2$ and exponent $p$ (and hence exponent-$p$ class $2$) such that $A(P)$ is the normaliser of the exceptional Chevalley group $G_2(p)$ in $\mathrm{GL}(7,p)$. Note that $7$ is the minimal dimension of a nontrivial irreducible $\mathbb{F}_p[G_2(p)]$-module.

In this paper, we extend the aforementioned programme of study by constructing a ``small'' $p$-group related to each exceptional Chevalley group defined over $\mathbb{F}_p$ (with $p$ odd, and with $p > 3$ in some cases). In order to make this more precise, we require the following definition.

\begin{defn}
\label{def:optimalpgp}
Let $H$ be a subgroup of $\mathrm{GL}(d,p)$, and let $\mathscr{P}_H$ be the set of $p$-groups $P$ of rank $d$ such that $A(P) = H$. We say that a $p$-group $P \in \mathscr{P}_H$ is \emph{optimal}\index{optimal $p$-group} with respect to $H$ if:
\begin{enumerate}[label={(\roman*)}]
\item \label{optimal1} no group in $\mathscr{P}_H$ has a smaller exponent-$p$ class than $P$;
\item \label{optimal2} no group in $\mathscr{P}_H$ with the same exponent-$p$ class as $P$ has a smaller exponent than $P$;
\item \label{optimal3} no group in $\mathscr{P}_H$ with the same exponent-$p$ class and exponent as $P$ has a smaller nilpotency class than $P$; and
\item \label{optimal4} no group in $\mathscr{P}_H$ with the same exponent-$p$ class, nilpotency class and exponent as $P$ has a smaller order than $P$.
\end{enumerate}
\end{defn}

We can now summarise our main theorem, which we prove in \S\ref{sec:mainproof}, and which incorporates the aforementioned main result of \cite{sdfg2}.

\begin{thm}
\label{thm:mainthmsummary}
Let $\hat G \in \{G_2(p),F_4(p),E_8(p)\}$ and let $r:=2$, or let $\hat G$ be the universal cover of $E_6(p)$ or of $E_7(p)$ and let $r:=3$, where $p > r$ in either case. Additionally, let $d$ be the minimal dimension of a nontrivial irreducible $\mathbb{F}_p[\hat G]$-module. Then each optimal $p$-group with respect to $N_{\mathrm{GL}(d,p)}(\hat G)$ has exponent-$p$ class $r$, nilpotency class $r$ and exponent $p$. Furthermore, if $\hat G \in \{G_2(p),E_8(p)\}$, then each optimal $p$-group with respect to the subgroup $\hat G$ of $\mathrm{GL}(d,p)$ has exponent-$p$ class $2$, nilpotency class $2$ and exponent $p^2$.
\end{thm}

We will see in \S\ref{sec:highestweights} that the group $\hat G$ is the (finite) \emph{simply connected version} of the corresponding exceptional Chevalley group. In fact, when $r = 2$, we construct an optimal $p$-group of the normaliser of $\hat G$ in $\mathrm{GL}(d,p)$ (and of $\hat G$ itself when $\hat G \ne F_4(p)$). When $\hat G$ is the universal cover of $E_6(p)$ or of $E_7(p)$, we construct a $p$-group that satisfies the first three conditions of Definition \ref{def:optimalpgp} and a variation of condition \ref{optimal4}. In each case, we construct the $p$-group as a particular quotient of a suitable \emph{universal $p$-group}, i.e., the largest $p$-group that satisfies a certain set of properties. Although the optimal $p$-group with respect to $N_{\mathrm{GL}(7,p)}(G_2(p))$ was previously constructed by Bamberg, Freedman and Morgan, they did not induce $G_2(p)$ itself on the Frattini quotient of a $p$-group. We will show later in the paper that the construction of each $p$-group related to $\hat G$ relies on knowledge of the stabilisers in $\mathrm{GL}(d,p)$ of the submodules of a certain $\mathbb{F}_p[\hat G]$-module.

Note that no version of Theorem \ref{thm:mainthmsummary} is known in the cases where $p \le r$. Exploring these cases will likely require further developments in the theory of $p$-group generation, in particular so that we may explicitly construct the relevant universal $p$-groups.

Let $V$ be a nonzero vector space over a field $\mathbb{F}$. For elements $u$ and $v$ of the tensor algebra $T(V)$ of $V$, we can define a bracket operation $[\cdot,\cdot]:T(V) \times T(V) \to T(V)$ by $[u,v]:=u \otimes v - v \otimes u$. We use left-normed notation to denote brackets of brackets, for example, $[u,v,w]:=[[u,v],w]$ for $u,v,w \in T(V)$. The \emph{free Lie algebra} $L(V)$ on $V$ is the smallest subspace of $T(V)$ that contains $V$ and that is closed under $[\cdot,\cdot]$, while the $i$-th \emph{Lie power} $L^iV$ of $V$ is the intersection of $L(V)$ and the $i$-th tensor power of $V$ (see \cite[\S2]{bamberg}, \cite[\S1]{johnson}). It is easy to show that $L(V)$ is indeed a Lie algebra with Lie bracket $[\cdot,\cdot]$; that $L^kV$ is spanned by the set  $\{[v_1,v_2,\ldots,v_k] \mid v_1,v_2,\ldots,v_k \in V \}$; and that $\mathrm{GL}(V)$ acts linearly on $L^kV$, with $[v_1,\ldots,v_k]^\alpha := [v_1^\alpha,\ldots,v_k^\alpha]$ for all $v_1, \ldots, v_k \in V$ and all $\alpha \in \mathrm{GL}(V)$. It follows immediately that if $\mathrm{char}(\mathbb{F}) \ne 2$, then $L^2V$ is isomorphic as an $\mathbb{F}[\mathrm{GL}(V)]$-module to the exterior square $A^2V$ of $V$. Additionally, Bamberg et al.~\cite[Lemma 3.1]{bamberg} showed that if $\mathrm{char}(\mathbb{F}) \notin \{2,3\}$, then $L^3V \cong (A^2V \otimes V)/A^3V$.

The aforementioned $\mathbb{F}_p[\hat G]$-module of importance is the Lie power $L^rV$, where $V$ is a \emph{minimal} $\mathbb{F}_p[\hat G]$-module, i.e., a nontrivial irreducible $\mathbb{F}_p[\hat G]$-module of minimal dimension $d$. In fact, the structure of Lie powers of modules is of general interest; see \cite[\S5]{johnson} for a discussion of cases where this structure is known, and examples of applications. We therefore consider the more general case where $\hat G$ and $V$ are defined over any field $\mathbb{F}_q$ of odd characteristic $p$, and determine, for each  $\hat G$-submodule $U$ of $L^rV$, the stabiliser $\mathrm{GL}(d,q)_U$ of $U$ in $\mathrm{GL}(d,q) \cong \mathrm{GL}(V)$. To do this, we prove results about the groups $\hat G$ and the modules $L^rV$ that are interesting in their own right. The following two theorems summarise some of our most significant findings.

\begin{thm} \leavevmode
\label{thm:structsummary}
\begin{enumerate}[label={(\roman*)}]
\item The submodule structure of $L^2V \cong A^2V$ is given in Table \ref{table:a2vstructfinite}.
\item Suppose that $p > 3$, and that $\hat G$ is of type $E_6$ (with $q = p$ if $p = 5$) or $E_7$ (with $q = p$ if $p \in \{7,11,19\}$). Then the submodule structure of $L^3V$ is given in Figures \ref{fig:e6l3vsub} and \ref{fig:e7l3vsub}.
\end{enumerate}
\end{thm}

We now introduce notation similar to that used in \cite[\S2.1]{kleidman}.

\begin{defn}
\label{def:formisom}
Let $\beta$ be a form on a vector space $W$. We write $I(\beta)$ (respectively, $\Delta(\beta)$) to denote the full group of isometries (respectively, similarities) of $\beta$ in $\mathrm{GL}(W)$. We also write $S(\beta):=I(\beta) \cap \mathrm{SL}(W)$ and $\Sigma(\beta):=\Delta(\beta) \cap \mathrm{SL}(W)$.
\end{defn}

\begin{thm} \leavevmode
\label{thm:overgroupsummary}
Let $\beta$ be a non-degenerate reflexive bilinear form on $V$ with $\hat G \le I(\beta)$, if such a form exists, and otherwise let $\beta$ be the zero form. Additionally, let $R$ be the last group in the derived series of $I(\beta)$.
\begin{enumerate}[label={(\roman*)}]
\item The normaliser of $\hat G$ in $R$ is the unique maximal subgroup of $R$ that contains $\hat G$.
\item If $\beta \ne 0$, then $\Sigma(\beta)$ is the unique maximal subgroup of $\mathrm{SL}(d,q)$ that contains $\hat G$.
\end{enumerate}
\end{thm}


The proofs of Theorems \ref{thm:structsummary} and \ref{thm:overgroupsummary} are given in \S\ref{sec:liepowerschev} and \S\ref{sec:overgroupschev}, respectively. We also show that the submodule structure of $A^2V$ is equivalent (in terms of containments and dimensions) to that of the exterior square of the $\overline{\mathbb{F}_p}[\hat G]$-module constructed from $V$ by extending the scalars, and that of the exterior square of a related minimal $\overline{\mathbb{F}_p}[G]$-module, where $G$ is the simply connected simple linear algebraic group associated with $\hat G$. This equivalence also holds between the corresponding third Lie powers, as long as $p$ is not an exceptional prime listed in Theorem \ref{thm:structsummary}. Note that the submodule structure of $A^2V$ in the case $\hat G = G_2(q)$ was previously explored, in less detail, in \cite{sdfg2} and \cite[Ch.~9.3.2]{schroder}.

We now specify the structure of our paper. First, we discuss in \S\ref{sec:univpgps} how we can use the knowledge of the submodule structures of $V$, $L^2V$ and $L^3V$ (each defined over $\mathbb{F}_p$) and the stabilisers in $\mathrm{GL}(d,p)$ of their submodules to construct a $p$-group $P$ as a quotient of an appropriate universal $p$-group such that $A(P)$ is as required. In \S\ref{sec:highestweights}, we use highest weight theory to determine the composition factors of the second and third Lie powers of the minimal modules over $\overline{\mathbb{F}_p}$ for the simply connected simple linear algebraic groups associated with exceptional Chevalley groups. We then use this information in \S\ref{sec:liepowerschev} to determine the submodule structures of the second and third Lie powers of these modules, and of the modules for the groups $\hat G$, as detailed above. Next, in \S\ref{sec:overgroupschev}, we determine part of the overgroup structure of $\hat G$ in $\mathrm{GL}(d,q)$. This overgroup structure, and the submodule structures of the second and third Lie powers of $V$ (defined over $\mathbb{F}_q$) are then used in \S\ref{sec:liepowersubmodstab} to determine the stabiliser in $\mathrm{GL}(d,q)$ of each submodule of these Lie powers. Finally, we state and prove the full version of our main theorem in \S\ref{sec:mainproof}.

\section{Universal $p$-groups}
\label{sec:univpgps}

Let $V:=\mathbb{F}_p^d$, with $d > 1$ an integer and $p$ a prime. In this section, we discuss the relation between the submodule structures of $V$, $L^2V$ and $L^3V$ and the group induced by $\mathrm{Aut}(P)$ on $P/\Phi(P)$, where $P$ is a suitable quotient of an appropriate \emph{universal $p$-group}.

First, let $B$ be the free Burnside group $B(d,p)$ of rank $d$ and exponent $p$, i.e., the largest group of rank $d$ and exponent $p$. For each positive integer $r$, we write $\Gamma(d,p,r):=B/\gamma_{r+1}(B)$. This is a finite group of rank $d$, exponent $p$ and nilpotency class at most $r$. In fact, if there exists a $p$-group of rank $d$, exponent $p$ and nilpotency class $r$, then every such group is a quotient of $\Gamma(d,p,r)$. In this case, we call $\Gamma(d,p,r)$ the \emph{universal $p$-group} of rank $d$, exponent $p$ and nilpotency class $r$. The following theorem of Bamberg et al.~\cite[\S2]{bamberg} describes how $\Gamma(d,p,r)$ can be constructed using Lie powers of $V$ when $r \in \{2,3\}$ and $p > r$. Here, $[\cdot,\cdot]$ is the Lie bracket associated with the free Lie algebra $L(V)$.

\begin{thm} \leavevmode
\label{thm:gamma23}
\begin{enumerate}[label={(\roman*)}]
\item Let $\Gamma_2(V)$ be the set $V \times L^2V$ equipped with the multiplication defined by
\begin{equation}
\label{eq:gamma2}
(a,b)(f,g) := (a+f,b+g+[a,f])
\end{equation}
for $(a,b),(f,g) \in \Gamma_2(V)$. If $p > 2$, then $\Gamma_2(V)$ is a group of nilpotency class $2$, and $\Gamma_2(V) \cong \Gamma(d,p,2)$.
\item Let $\Gamma_3(V)$ be the set $V \times L^2V \times L^3V$ equipped with the multiplication defined by
\begin{equation}
\label{eq:gamma3}
(a,b,c)(f,g,h) := (a+f,b+g+[a,f],c+h+3([b,f]-[g,a])+[a,f,f-a])
\end{equation}
for $(a,b,c),(f,g,h) \in \Gamma_3(V)$. If $p > 3$, then $\Gamma_3(V)$ is a group of nilpotency class $3$, and $\Gamma_3(V) \cong \Gamma(d,p,3)$.
\end{enumerate}
\end{thm}

Note that for $r \in \{2,3\}$ with $p > r$, the identity of $\Gamma_r(V)$ is the element with each coordinate equal to $0$, and the inverse of an element of $\Gamma_r(V)$ is obtained by multiplying each coordinate by $-1$. Additionally, $\mathrm{GL}(d,p) \cong \mathrm{GL}(V)$ acts on $\Gamma_r(V)$, with $(a,b)^\alpha := (a^\alpha,b^\alpha)$ and $(a,b,c)^\alpha := (a^\alpha,b^\alpha,c^\alpha)$ for each $\alpha \in \mathrm{GL}(d,p)$, $a \in V$, $b \in L^2V$ and $c \in L^3V$ \cite[Theorem 2.5]{bamberg}. Bamberg et al.~also showed that if $p > 3$, then $\Gamma(d,p,4)$ is isomorphic to a group of nilpotency class $4$ with underlying set $V \times L^2V \times L^3V \times L^4V$. However, we will not consider this group in this paper.

It is clear that if $p > 2$, then $R:=\{(0,b) \mid b \in L^2V\}$ is equal to $Z(\Gamma_2(V))$ and isomorphic to the elementary abelian group $L^2V$. Hence we can identify the subgroups of $R$ with the subspaces of $L^2V$. If $U$ is a proper subgroup of $R$, then we can identify $P_U:=\Gamma_2(V)/U$ with the set
\begin{equation}
\label{eq:setpu}
V \times (L^2V)/U,
\end{equation}
equipped with the multiplication given by
\begin{equation}
\label{eq:gamma2quo}
(a,b+U)(f,g+U) := (a+f,b+g+[a,f]+U)
\end{equation}
for $a,f \in V$ and $b,g \in L^2V$. Similarly, if $p > 3$ and if $W$ is a proper subgroup of $S:=\{(0,0,c) \mid c \in L^3V\} = Z(\Gamma_3(V))$, then $W$ is a proper subspace of $L^3V$. Here, we can identify $Q_W:=\Gamma_3(V)/W$ with the set
\begin{equation}
\label{eq:setqw}
V \times L^2V \times (L^3V)/W,
\end{equation}
equipped with the multiplication given by
\begin{equation}
\label{eq:gamma3quo}
(a,b,c+W)(f,g,h+W) := (a+f,b+g+[a,f],c+h+3([b,f]-[g,a])+[a,f,f-a]+W)
\end{equation}
for $a,f \in V$, $b,g \in L^2V$ and $c,h \in L^3W$. Note that $P_{\{0\}} = \Gamma_2(V)$ and $Q_{\{0\}} = \Gamma_3(V)$.

The following propositions describe some important properties of $P_U$ and $Q_W$. In the cases where $P_U = \Gamma_2(V)$ or $Q_W = \Gamma_3(V)$, these results follow easily from \cite[\S2]{bamberg}.

\begin{prop}
\label{prop:gamma2derphi}
Suppose that $p > 2$, and let $U$ be a proper subspace of $L^2V$. Then ${P_U}' = \Phi(P_U) = \{(0,b+U) \mid b \in L^2V\}$, which is isomorphic to $(L^2V)/U$.
\end{prop}

\begin{proof}
It is clear that $R:=\{(0,b+U) \mid b \in L^2V\}$ is a subgroup of $P_U$ isomorphic to $(L^2V)/U$. Observe that the commutator of two elements $(a,b+U),(f,g+U) \in P_U$ is $(0,2[a,f]+U)$. Hence ${P_U}' \le R$. Let $\{e_1,\ldots,e_d\}$ be a basis for $V$, let $i, j \in \{1,\ldots,d\}$, and let $m:=(p+1)/2$. Then $[(e_i,U),(e_j,U)]^m = (0,[e_i,e_j]+U)$. The bilinearity of $[\cdot,\cdot]$ implies that $L^2V$ is spanned by the Lie brackets $[e_i,e_j]$, and so $R$ is generated by the commutators $[(e_i,U),(e_j,U)]$. Therefore, ${P_U}' = R$. Finally, $\Gamma_2(V)$ has exponent $p$, and hence its quotient $P_U \ne 1$ also has exponent $p$. Thus $\Phi(P_U) = {P_U}'$.
\end{proof}

\begin{prop}
\label{prop:gamma3derphi}
Suppose that $p > 3$, and let $W$ be a proper subspace of $L^3V$. Then:
\begin{enumerate}[label={(\roman*)}]
\item $\gamma_3(Q_W) = \{(0,0,c+W) \mid c \in L^3V\}$, which is isomorphic to $(L^3V)/W$; and
\item ${Q_W}' = \Phi(Q_W) = \{(0,b,c+W) \mid b \in L^2V, c \in L^3V\}$.
\end{enumerate}
\end{prop}

\begin{proof}
Let $a_i \in V$, $b_i \in L^2V$ and $c_i \in L^3V$ for each $i \in \{1,2,3\}$. Using \eqref{eq:gamma3}, \eqref{eq:gamma3quo}, and the formula for a commutator of three elements of $\Gamma_3(V)$ from \cite[p.~2936]{bamberg}, we see that $$[(a_1,b_1,c_1+W),(a_2,b_2,c_2+W),(a_3,b_3,c_3+W)] = (0,0,12[a_1,a_2,a_3]+W).$$ The group $\gamma_3(Q_W) = [Q_W,Q_W,Q_W]$ is generated by commutators of three elements of $Q_W$ (see \cite[Proposition 1.6.5(ii)]{bourbaki}), and so it lies in $S:={\{(0,0,c+W) \mid c \in L^3V\}}$, which is a subgroup of $Q_W$ isomorphic to the elementary abelian group $(L^3V)/W$. Now, let $\{e_1,\ldots,e_d\}$ be a basis for $V$. As $p > 3$, there is a positive integer $m < p$ such that $12m \equiv 1 \pmod p$, and hence ${(0,0,12[e_i,e_j,e_k]+W)^m} = {(0,0,[e_i,e_j,e_k]+W)}$. The bilinearity of $[\cdot,\cdot]$ implies that $L^3V$ is spanned by the Lie brackets $[e_i,e_j,e_k]$, and hence $S$ is generated by the commutators $[(e_i,0,W),(e_j,0,W),(e_k,0,W)]$. Thus $S = \gamma_3(Q_W)$.

Next, the first two coordinates of $[(a_1,b_1,c_1+W),(a_2,b_2,c_2+W)]$ are equal to the two coordinates of $[(a_1,b_1),(a_2,b_2)] \in \Gamma_2(V)$. Thus by Proposition \ref{prop:gamma2derphi}, ${Q_W}'$ lies in the subgroup $T:={\{(0,b,c+W) \mid b \in L^2V, c \in L^3V\}}$ of $Q_W$. It also follows from this proposition that for each $b \in L^2V$, we can multiply commutators of elements of $\Gamma_3(V)$ to obtain an element $(0,b,h+W)$ of $Q_W$, for some $h \in L^3V$. Since ${Q_W}'$ contains $\gamma_3(Q_W)$, we have shown above that, for each $c \in L^3V$, the element $(0,0,c-h+W)$ is a product of commutators of elements of $\Gamma_3(V)$. As $(0,b,h+W)(0,0,c-h+W) = (0,b,c+W)$, the subgroup $T$ is generated by commutators of elements of $\Gamma_3(V)$. Therefore, ${Q_W}' = T$.
\end{proof}


\begin{lem}
\label{lem:gamma23quotients}
Let $P$ be a $p$-group.
\begin{enumerate}[label={(\roman*)}]
\item Suppose that $p > 2$. Then $P$ has rank $d$, exponent $p$ and nilpotency class $2$ if and only if $P \cong P_U$ for some proper subspace $U$ of $L^2V$.
\item Suppose that $p > 3$. Then $P$ has rank $d$, exponent $p$ and nilpotency class $3$ if and only if $P$ is a quotient of $\Gamma_3(V)$ by a normal subgroup that lies in $\gamma_2(\Gamma_3(V))$, and that does not contain $\gamma_3(\Gamma_3(V))$. In particular, if $P \cong Q_W$ for some proper subspace $W$ of $L^3V$, then $P$ has rank $d$, exponent $p$ and nilpotency class $3$. \label{gamma3quos}
\end{enumerate}
\end{lem}

\begin{proof}
If $N$ is a normal subgroup of a group $G$ and if $i$ is a positive integer, then $G/N$ is nilpotent of nilpotency class $i$ if and only if $N$ contains $\gamma_{i+1}(G)$ but not $\gamma_i(G)$. In addition, $G$ and $G/N$ have the same rank if and only if $N \le \Phi(G)$. The result therefore follows from Propositions \ref{prop:gamma2derphi} and \ref{prop:gamma3derphi}, and from the fact that, for each $r \in \{2,3\}$, the group $\Gamma_r(V)$ is the universal group of rank $d$, exponent $p$ and nilpotency class $r$.
%
%
%
\end{proof}

Note that if $N$ is a normal subgroup of $\Gamma_3(V)$ that lies in $\gamma_2(\Gamma_3(V))$ and neither contains nor lies in $\gamma_3(\Gamma_3(V))$, then $\Gamma_3(V)/N$ is a $p$-group of rank $d$, exponent $p$ and nilpotency class $3$ that cannot be expressed as $Q_W$ for any subspace $W$ of $L^3V$.

Now, if $B=B(d,p)$, and if $r$ is any positive integer such that $\gamma_r(B) \ne 1$, then $\gamma_r(B)/\gamma_{r+1}(B)$ is an $\mathbb{F}_p[\mathrm{GL}(d,p)]$-module isomorphic to $L^rV$. We can therefore identify the subspaces of $L^rV$ with the quotients $M/\gamma_{r+1}(B)$, for the subgroups $M$ of $\gamma_r(B)$ that contain $\gamma_{r+1}(B)$. In each case, $M \trianglelefteq B$, and we have $$\Gamma(d,p,r)/(M/\gamma_{r+1}(B)) = (B/\gamma_{r+1}(B))/(M/\gamma_{r+1}(B)) \cong B/M.$$

\begin{thm}[{\cite[Theorem 2.2]{bamberg}}]
\label{thm:burnsideap}
Let $r$ be a positive integer, and let $M$ be a proper subgroup of $\gamma_r(B)$ that contains $\gamma_{r+1}(B)$. Then the group $A(B/M)$ induced by $\mathrm{Aut}(B/M)$ on the Frattini quotient of $B/M$ is the stabiliser of $M/\gamma_{r+1}(B)$ in $\mathrm{GL}(d,p)$.
\end{thm}

The following theorem describes $A(P)$ when $P$ is a $p$-group isomorphic to $P_U$ or to $Q_W$, for some proper subspace $U$ of $L^2V$ or $W$ of $L^3V$. This result was used in \cite{sdfg2} and \cite{bamberg} to construct $p$-groups of exponent $p$ and nilpotency class $2$ or $3$.

\begin{thm}
\label{thm:univap}
Suppose that $p > r$, with $r \in \{2,3\}$, and that $P$ is a $p$-group isomorphic to $P_X$ or $Q_X$ for some proper subspace $X$ of $L^rV$. Then $A(P) = \mathrm{GL}(d,p)_X$.
\end{thm}

\begin{proof}
There exists a proper subgroup $M$ of $\gamma_{r}(B)$ with $\gamma_{r+1}(B) \le M$ and $X = M/\gamma_{r+1}(B)$. We have $P \cong B/M$, and it follows from Theorem \ref{thm:burnsideap} that $A(P) = A(B/M) = \mathrm{GL}(d,p)_X$, up to conjugacy in $\mathrm{GL}(d,p)$.
\end{proof}

This theorem shows that in order to construct a $p$-group $P$ as $P_U$ (respectively, as $Q_W$) such that $A(P)$ is a particular subgroup $H$ of $\mathrm{GL}(d,p)$, then $H$ must be the stabiliser in $\mathrm{GL}(d,p)$ of some proper subspace of $L^2V$ (respectively, of $L^3V$). We must therefore be able to distinguish between $H$ and any proper overgroup of $H$ in $\mathrm{GL}(d,p)$ by comparing the subspaces of $L^2V$ (respectively, of $L^3V$) stabilised by these linear groups. Observe that if $r \in \{2,3\}$, if $p > r$, and if $P = \Gamma_r(V)$, then $A(P) = \mathrm{GL}(d,p)_{\{0\}} = \mathrm{GL}(d,p)$. Hence if $X$ is a proper subspace of $L^rV$, then $A(\Gamma_r(V)/X) \le A(\Gamma_r(V))$.

Suppose now that $p > 3$, that $H$ is a subgroup of $\mathrm{GL}(d,p)$, and that there is no $p$-group $P$ of exponent-$p$ class $2$ such that $A(P) = H$. Additionally, suppose that $W$ is a proper subgroup of $L^3V$ such that $A(Q_W) = H$, with $W$ having the largest order of such a proper subgroup. Then $Q_W$ has exponent-$p$ class $3$, exponent $p$ and nilpotency class $3$, and minimal order among the groups that can be expressed as $Q_X$, with $X$ a proper subspace of $L^3V$ such that $A(Q_X) = H$. Moreover, if $Q$ is an optimal $p$-group with respect to $H$ as in Definition \ref{def:optimalpgp}, then $Q$ has the same exponent-$p$ class, exponent and nilpotency class as $Q_W$. However, we are not able to determine the order of $Q$. This is because, if $N$ is a normal subgroup of $\Gamma_3(V)$ that satisfies the hypotheses of Lemma \ref{lem:gamma23quotients}\ref{gamma3quos} but does not lie in $\gamma_3(\Gamma_3(V)) \cong L^3V$, then Theorem \ref{thm:univap} does not yield any information about $A(\Gamma_3(V)/N)$. We therefore introduce the following definition.

\begin{defn}
\label{def:quasioptimalpgp}
Let $H$ be a subgroup of $\mathrm{GL}(d,p)$, and suppose that each $p$-group that is optimal with respect to $H$ has exponent-$p$ class $3$, exponent $p$ and nilpotency class $3$. Additionally, let $\mathscr{Q}_H$ be the set of $p$-groups $Q_W$, for proper subspaces $W$ of $L^3V$, such that $A(Q_W) = H$. The groups in $\mathscr{Q}_H$ of minimal order are called \emph{quasi-optimal}\index{quasi-optimal $p$-group} with respect to $H$.
\end{defn}

Observe that a $p$-group that is quasi-optimal with respect to $H$ satisfies conditions \ref{optimal1}--\ref{optimal3} of Definition \ref{def:optimalpgp} and a variation of condition \ref{optimal4}. Although we are aware of $p$-groups that are quasi-optimal with respect to certain linear groups (such as those described in \S\ref{sec:mainproof} and \cite[Table 6.1]{bamberg}), we do not know whether or not any of these $p$-groups are optimal with respect to the corresponding linear groups.

Now, let $Q$ be a $p$-group of rank $n>1$ and exponent-$p$ class $c$. We summarise some definitions from a paper by O'Brien \cite[\S2]{obrien} related to the \emph{$p$-covering group} $Q^*$ of $Q$. Here, $Q^*$ is the largest group among $p$-groups $P$ of rank $n$ that contain an elementary abelian subgroup $M_P$, with $M_P \le Z(P) \cap \Phi(P)$ and $P/M_P \cong Q$, and every such $p$-group $P$ is a quotient of $Q^*$. The subgroup $M:=M_{Q^*}$ of $Q^*$ is called the \emph{$p$-multiplicator} of $Q$. A subgroup $X$ of $M$ is called \emph{allowable} if $R:=Q^*/X$ has rank $n$ and exponent-$p$ class $c+1$, and if $R/\mathcal{P}_{c+1}(R) \cong Q$, where $\mathcal{P}_{c+1}(R)$ is the $(c+1)$-th group in the lower exponent-$p$ central series of $R$. In fact, a subgroup $X$ of $M$ is allowable if and only if $X < M$ and $X \mathcal{P}_{c+1}(Q^*) = M$ \cite[Theorem 2.4]{obrien}.

Throughout the rest of this section, let $E$ denote the elementary abelian $p$-group of rank $d$. In the following proposition, we consider the $p$-covering group $E^*$ of $E$.

\begin{prop}
\label{prop:vstar}
A $p$-group $P$ has rank $d$ and exponent-$p$ class $2$ if and only if $P$ is a quotient of $E^*$ by a proper subgroup of $\Phi(E^*)$. In particular, $E^*$ has exponent-$p$ class $2$. Moreover, $M_{E^*} = \Phi(E^*)$.
\end{prop}

\begin{proof}
First assume that $P$ has rank $d$ and exponent-$p$ class $2$. Then $\Phi(P)$ is an elementary abelian subgroup of $Z(P)$, and $P/\Phi(P) \cong E$ by Burnside's Basis Theorem. Thus by the definition of $E^*$, there exists a normal subgroup $Y$ of $E^*$ such that $P \cong E^*/Y$. As $E^*$ and $P$ both have rank $d$, and as $P$ is not elementary abelian, we have $Y < \Phi(E^*)$, as required. Note that the direct product of $d$ copies of the cyclic group $C_{p^2}$ has rank $d$ and exponent $p$-class $2$, and so $\Phi(E^*) > 1$.

To prove the converse, let $M:=M_{E^*}$. Since $M \le \Phi(E^*)$ and $E^*/M \cong E$ is elementary abelian, we have $M = \Phi(E^*)$. Additionally, $E$ has exponent-$p$ class $1$, and $\mathcal{P}_{2}(E^*) = \Phi(E^*)$. Therefore, every proper subgroup $X$ of $M$ is allowable, i.e., $E^*/X$ has rank $d$ and exponent-$p$ class $2$. As $\Phi(E^*) > 1$, it follows that $E^*$ itself has exponent-$p$ class $2$.
\end{proof}

We can therefore consider $E^*$ as the universal $p$-group of rank $d$ and exponent-$p$ class $2$. Additionally, each automorphism $\alpha$ of $E$ lifts to an automorphism $\alpha^*$ of $E^*$, via the natural epimorphism from $E^*$ to $E$ with kernel $\Phi(E^*)$ \cite[p.~2275]{eick}. In fact, $\mathrm{Aut}(E)$ acts on $M_{E^*}=\Phi(E^*)$, with $x^\alpha:=x^{\alpha^*}$ for each $x \in \Phi(E^*)$ and each $\alpha \in \mathrm{Aut}(E)$. In particular, if $V$ is the vector space $\mathbb{F}_p^d$, then the natural linear action of $\mathrm{Aut}(E) \cong \mathrm{GL}(V)$ on $E \cong V$ induces an action of $\mathrm{GL}(V)$ on $\Phi(E^*)$.

Observe that the exponent-$p$ class of a $p$-group of exponent $p$ is equal to its nilpotency class. Supposing that $p > 2$, Theorem \ref{thm:gamma23} implies that the universal $p$-group $\Gamma(d,p,2)$ of rank $d$, exponent $p$ and nilpotency class $2$ is also the universal $p$-group of rank $d$, exponent $p$ and exponent-$p$ class $2$. Thus $\Gamma(d,p,2)$ is the largest quotient of $E^*$ of exponent $p$, i.e., $\Gamma(d,p,2) \cong E^*/(E^*)^p$. 
Moreover, $\Phi(E^*) = (E^*)^p \oplus (E^*)'$, with $(E^*)^p \cong V$ and ${(E^*)' \cong A^2V}$ as $\mathbb{F}_p[\mathrm{GL}(V)]$-modules \cite[\S1--2]{glasbyucs}. Since $A^2V \cong L^2V$, and since ${E^*/\Phi(E^*) \cong V}$ by Burnside's Basis Theorem, we deduce the following.

\begin{prop}
\label{prop:vstarext}
Suppose that $p > 2$. Then $E^*$ is an extension of $V \oplus L^2V$ by $V$.
\end{prop}

We can also use the above decomposition of $\Phi(E^*)$ into $\mathrm{GL}(V)$-submodules to describe $A(P)$ when $P$ is a $p$-group of exponent-$p$ class $2$. A version of the following result, which is a generalisation of the $r  = 2$ case of Theorem \ref{thm:univap}, previously appeared in \cite[\S2]{glasbyucs}.

\begin{thm}
\label{thm:expclass2stab}
Suppose that $p > 2$, and that $P$ is a $p$-group of exponent-$p$ class $2$, i.e., that $P \cong E^*/X$ for a proper subgroup $X$ of $\Phi(E^*) \cong V \oplus L^2V$. Then $A(P) = \mathrm{GL}(d,p)_X$. Furthermore:
\begin{enumerate}[label={(\roman*)}]
\item $P$ is abelian if and only if $X$ contains the direct summand $L^2V$. In this case, $P$ has exponent $p^2$; and $A(P) = \mathrm{GL}(d,p)_{X \cap V}$. \label{exppsqstab}
\item $P$ has exponent $p$ if and only if $X$ contains the direct summand $V$. In this case, $P$ has nilpotency class $2$; $A(P) = \mathrm{GL}(d,p)_{X \cap L^2V}$; and $P \cong P_{X \cap L^2V}$. \label{exppnilp2stab}
\end{enumerate}
\end{thm}

\begin{proof}
Let $Q: = E^*/X$. Since $X$ is an allowable subgroup of $E^*$, the image of the natural action $\theta: \mathrm{Aut}(Q) \to \mathrm{Aut}(Q/\mathcal{P}_2(Q))$ is the stabiliser of $X$ in $\mathrm{Aut}(E) \cong \mathrm{GL}(d,p)$ (see \cite[Theorem 3.2]{eick} and \cite[Theorem 2.10]{obrien}). We have $\mathcal{P}_2(Q) = \Phi(Q)$, and hence $\mathrm{GL}(d,p)_X = A(Q)$, which is equal to $A(P)$ up to conjugacy in $\mathrm{GL}(d,p)$.

Now, $P$ is abelian if and only if $X$ contains $(E^*)' \cong L^2V$. Here, as $P$ has exponent-$p$ class $2$, it must have exponent $p^2$. Additionally, $X = (X \cap V) \oplus L^2V$, and hence $A(P) = \mathrm{GL}(d,p)_{(X \cap V) \oplus L^2V}$. However, $\mathrm{GL}(d,p)$ stabilises $L^2V$, and hence $A(P) = \mathrm{GL}(d,p)_{X \cap V}$.

Finally, $P$ has exponent $p$ if and only if $X$ contains $(E^*)^p \cong V$. Here, $P$ has nilpotency class $2$, and $A(P) = \mathrm{GL}(d,p)_{V \oplus (X \cap L^2V)}$. Since $\mathrm{GL}(d,p)$ stabilises $V$, we have $A(P) = \mathrm{GL}(d,p)_{X \cap L^2V}$. We also have $P \cong E^*/X \cong (E^*/(E^*)^p)/(X/(E^*)^p)$, with $E^*/(E^*)^p \cong \Gamma(d,p,2)$ and ${X/(E^*)^p \cong (V \oplus (X \cap L^2V))/V \cong X \cap L^2V}$. Hence $P \cong P_{X \cap L^2V}$.
\end{proof}

Thus in order to construct a $p$-group $P$ of exponent-$p$ class $2$ such that $A(P)$ is a particular subgroup $H$ of $\mathrm{GL}(d,p)$, we must be able to distinguish between $H$ and any proper overgroup of $H$ in $\mathrm{GL}(d,p)$ by comparing the subspaces of $V \oplus L^2V$ stabilised by these linear groups. Observe that if $P$ is a $p$-group satisfying condition \ref{exppsqstab} of Theorem \ref{thm:expclass2stab}, with $X \cap V \ne \{0\}$, then $A(P)$ acts reducibly on $V$. In fact, $A(P)$ is a maximal subgroup of $\mathrm{GL}(d,p)$ that lies in the Aschbacher class denoted $\mathcal{C}_1$ \cite[Remark 6.3]{bamberg}. If instead $P = E^*$, then $A(P) = \mathrm{GL}(d,p)_{\{0\}} = \mathrm{GL}(d,p)$.

The final theorem in this section reveals more information about a given $p$-group of exponent-$p$ class $2$.

\begin{thm}[{\cite[Theorem 4]{glasbyucs}}]
\label{thm:ucsgroups}
Suppose that $p > 2$, and that $P$ and $X$ are as in Theorem \ref{thm:expclass2stab}. Then $P$ has a unique nontrivial proper characteristic subgroup if and only if $\mathrm{GL}(d,p)_X$ acts irreducibly on each of $V$ and $(V \oplus L^2V)/X$.
\end{thm}

A $p$-group with a unique (nontrivial proper) characteristic subgroup is called a \emph{UCS $p$-group} \cite{glasbyucs}. Since each group in the lower exponent-$p$ central series of a $p$-group is characteristic, each UCS $p$-group has exponent-$p$ class at most $2$.


\section{Minimal modules for exceptional algebraic groups}
\label{sec:highestweights}

In this section, we consider Lie powers of minimal modules of the linear algebraic groups associated with the exceptional Chevalley groups, from the perspective of highest weight theory. This will allow us, in the next section, to derive corresponding results about the Lie powers of minimal modules of the Chevalley groups themselves (or their universal covers). Malle and Testerman \cite{malle} and L\"ubeck \cite{lubeck} give excellent introductions to linear algebraic groups and highest weight theory, and we assume that the reader is relatively familiar with these topics. In this paper, all linear algebraic groups are assumed to be simple of simply connected type.

Throughout this section and the next, we use the following notation:
\begin{itemize}
\item $q$ is a power of a prime $p$;
\item $\tilde G = {}^t Y_\ell(q)$ is a finite simple group of Lie type defined over $\mathbb{F}_q$;
\item $J$ is the universal cover of $\tilde G$, with $Z(J)$ the Schur multiplier of $\tilde G$;
\item $\hat G$ is the quotient of $J$ by the unique Sylow $p$-subgroup of $Z(J)$;
\item $K:= \overline {\mathbb{F}_p}$;
\item $G$ is the (simple, simply connected) linear algebraic group of type $Y_\ell$ over $K$, so that $\hat G$ is the group of fixed points of a particular Steinberg endomorphism of $G$;
\item $T$ is a fixed maximal torus of $G$;
\item $X(T) \cong \mathbb{Z}^\ell$ is the character group of $T$;
\item $W$ is the Weyl group of $G$;
\item $\Delta = \{\alpha_1,\ldots,\alpha_\ell\}$ is a fixed base of the root system of $G$; and
\item $\{\lambda_1,\ldots,\lambda_\ell\}$ is the set of fundamental dominant weights of $T$ (with respect to $\Delta$).
\end{itemize}

Note that we do not consider the Tits group ${{}^2 F_4(2)}'$ as a group of Lie type. We follow the convention of \cite[Ch.~5.1.1]{BHRD} and refer to $\hat G$ as the (finite) \emph{simply connected version} of $\tilde G$. In fact, for all but a finite number of simple groups of Lie type $\tilde G$, the Sylow $p$-subgroup of $Z(J)$ is trivial, and hence $\hat G = J$. Even when this is not the case, $\tilde G \cong \hat G/Z(\hat G)$. Moreover, if $Z(J)$ is a $p$-group, then $\hat G = J/Z(J) \cong \tilde G$. In particular, if $\tilde G$ is an exceptional Chevalley group, then $\hat G \ne J$ if and only if $\tilde G \in \{G_2(3),G_2(4),F_4(2)\}$, and $\hat G \not\cong \tilde G$ if and only if $\tilde G = E_6(q)$ with $q \equiv {1 \pmod 3}$ or $\tilde G = E_7(q)$ with $q$ odd \cite[Theorem 5.1.4]{kleidman}. In the former $\hat G \not\cong \tilde G$ case, $\hat G = J = \nonsplit{3}{E_6(q)}$ \cite[Ch.~4.10.6]{wilson}, and in the latter, $\hat G = J = \nonsplit{2}{E_7(q)}$ \cite[Ch.~4.12]{wilson}. Here, and throughout the paper, we denote group extensions using \textsc{Atlas} \cite{ATLAS} notation.

We recall that a group $H$ is \emph{quasisimple} if $H$ is perfect and $H/Z(H)$ is non-abelian and simple. Observe that since $\hat G$ is a quotient of the perfect group $J$, and since $\hat G/Z(\hat G)$ is isomorphic to the non-abelian simple group $\tilde G$, the group $\hat G$ is quasisimple.

As the above notation suggests, if $r$ is a power of $p$, and if ${}^s Y_\ell(r)$ is a simple group of Lie type, then $G$ is the linear algebraic group associated with both $\tilde G$ and ${}^s Y_\ell(r)$. Moreover, $G$ can be considered as a group of matrices over $K$. We can also identify $X(T)$ with a subset of the Euclidean space $\mathbb{R}^\ell$, with $\Delta \subset X(T)$ a basis for $\mathbb{R}^\ell$. The Weyl group $W$ is the group generated by the reflections in each hyperplane orthogonal to a root in $\Delta$, and the action of $W$ on $\mathbb{R}^\ell$ induces an action on $X(T)$. 
%
Each character $\alpha \in X(T)$ is a $\mathbb{Z}$-linear combination of the fundamental dominant weights $\lambda_i$, and if the coefficient of each $\lambda_i$ is nonnegative, then we say that $\alpha$ is \emph{dominant}.

In this paper, whenever we refer to a $K[G]$-module, we mean a finite-dimensional rational $K[G]$-module. We recall the following standard definition.

\begin{defn}
\label{def:weights}
Let $V$ be a $K[G]$-module. For $\lambda \in X(T)$, we define $$V_\lambda := {\{v \in V \mid v^t = (t)\lambda v \text{ for all } t \in T\}}.$$ If $V_\lambda \ne \{0\}$, then $\lambda$ is a \emph{weight} of $V$ with (respect to $T$) with \emph{multiplicity} $\dim(V_\lambda)$, and $V_\lambda$ is a \emph{weight space} of $V$.
\end{defn}

We will write $\Lambda(V)$ to denote the weight multiset for $V$, with each weight represented as many times as its multiplicity. Each weight space of $V$ is clearly a $T$-submodule of $V$. In fact, $V$ decomposes as the direct sum of its weight spaces, and thus $|\Lambda(V)| = \dim(V)$. It is easy to show that if $U$ is a $K[G]$-module isomorphic to $V$, then $\Lambda(U) = \Lambda(V)$. We will shortly summarise methods of calculating the weight multisets for $K[G]$-modules constructed from other $K[G]$-modules. First, we define one such construction that applies to a module for any group over any field.

Let $V$ be a module for a group $H$ over a field $\mathbb{F}$, and let $\alpha \in \mathrm{Aut}(H)$. Then $V^\alpha$ denotes the module obtained by \emph{twisting} $V$ by $\alpha$. Specifically, if $V$ affords the representation $\rho$, then $V^{\alpha}$ affords the representation $\rho_\alpha$, where $(h)\rho_\alpha := (h^\alpha)\rho$ for each $h \in H$. For reasons that will become clear, we are particularly interested in the modules $V^{\phi^n}$, where $n$ is a positive integer, $V$ is a $K[G]$-module, and $\phi$ is the field automorphism of $G$ that maps each matrix entry in an element of $G$ to its $p$-th power. Note that $\phi$ is an abstract group automorphism, but not a linear algebraic group automorphism; in fact, it is a Frobenius endomorphism of $G$.

%

Hall \cite[Lemma 10.37]{bhall} states the first result of the following lemma; Samelson \cite[p.~106--108]{samelson} states the second and third; and the fourth follows from \cite[p.~135]{malle}.

\begin{lem}
\label{lem:tensorprodweights}
Let $V$ be a $K[G]$-module.
\begin{enumerate}[label={(\roman*)}]
\item Suppose that $U$ is a submodule of $V$. Then $\Lambda(V/U) = \Lambda(V)\setminus\Lambda(U)$.
\item Suppose that $U$ is a $K[G]$-module. Then $\Lambda(U \otimes V)$ is the multiset $$[\lambda + \mu \mid \lambda \in \Lambda(U), \mu \in \Lambda(V)],$$ with $\lambda_1 + \mu_1$ and $\lambda_2 + \mu_2$ considered distinct if and only if $\lambda_1 \ne \lambda_2$ and $\mu_1 \ne \mu_2$.
\item Suppose that $n$ is an integer at least $2$ and at most $m := \dim(V)$, and that $\Lambda(V) = [\lambda_1, \ldots, \lambda_m]$.
Then $$\Lambda(A^nV) = [\lambda_{i_1} + \cdots + \lambda_{i_n} \mid 1 \le i_1 < \cdots < i_n \le m].$$
\item Suppose that $n$ is a nonnegative integer. Then $\Lambda(V^{\phi^n}) = [p^n \lambda \mid \lambda \in \Lambda(V)]$. \label{tensorweights4}
\end{enumerate}
\end{lem}

Throughout this paper, when we refer to the composition factors of $V$, we mean the multiset of composition factors, with possible isomorphisms between these factors.

\begin{prop}
\label{prop:chainweights}
Let $V$ be a $K[G]$-module. Then $\Lambda(V)$ is the disjoint union of the weight multisets for the composition factors of $V$.
\end{prop}

\begin{proof}
Let $\{0\} = V_0 \subset V_1 \subset \cdots \subset V_n = V$ be a composition series for $V$. If $0 \le k \le n-1$, then $V/V_{k+1} \cong (V/V_k)/(V_{k+1}/V_k)$, and so $$\Lambda(V/V_{k+1}) = \Lambda((V/V_k)/(V_{k+1}/V_k)) = {\Lambda(V/V_k)\setminus\Lambda(V_{k+1}/V_k)}.$$ This means that $\Lambda(V/V_k)$ is the disjoint union of $\Lambda(V/V_{k+1})$ and $\Lambda(V_{k+1}/V_k)$. It follows by induction on $k$ that $\Lambda(V) = \Lambda(V/V_0)$ is the disjoint union of $\Lambda(V_{k+1}/V_k)$ for all $i \in \{1,2,\ldots,n-1\}$, as required.
%
\end{proof}

Let $\le$ be the relation on $X(T)$ such that $\alpha \le \beta$ for $\alpha, \beta \in X(T)$ if and only if $\beta - \alpha$ is a linear combination of roots in $\Delta$, with the coefficient of each root nonnegative. Then $\le$ is a partial order. The following well-known theorem (see \cite[Ch.~15--16]{chevalley}) shows that there is a 1-1 correspondence between isomorphism classes of irreducible $K[G]$-modules and dominant characters of $T$.

\begin{thm}[Chevalley]\leavevmode
\label{thm:highestweightirred}
\begin{enumerate}[label={(\roman*)}]
\item Let $V$ be an irreducible $K[G]$-module. Then there is a unique weight $\lambda \in \Lambda(V)$, called the \emph{highest weight} of $V$, such that $\mu \le \lambda$ for all $\mu \in \Lambda(V)$. Moreover, $\lambda$ is dominant and has multiplicity $1$.
\item If $\lambda \in X(T)$ is dominant, then there exists an irreducible $K[G]$-module $V$ with highest weight $\lambda$.
\item Two $K[G]$-modules are isomorphic if and only if they have the same highest weight.
\end{enumerate}
\end{thm}

For a dominant character $\lambda \in X(T)$, we write $L(\lambda)$ to denote the unique (up to isomorphism) irreducible $K[G]$-module with highest weight $\lambda$. Lemma \ref{lem:tensorprodweights}\ref{tensorweights4} implies that if $n$ is a positive integer, then $L(\lambda)^{\phi^n}$ is isomorphic to the irreducible module $L(p^n \lambda)$. The only weight of the trivial irreducible module is the trivial character $0$, and hence this module is equal to $L(0)$. For an arbitrary $K[G]$-module $V$, we say that a weight $\lambda$ in a subset $D$ of $\Lambda(V)$ is a highest weight of $D$ if $\mu \le \lambda$ for all $\mu \in D$.

\begin{lem}
\label{lem:highestweightfactor}
Let $C$ be a (possibly empty) disjoint union of weight multisets for composition factors of the $K[G]$-module $V$. Additionally, let $\lambda \in D:= \Lambda(V) \setminus C$ be such that $\lambda \not < \mu$ for all $\mu \in D$. Then $L(\lambda)$ is a composition factor of $V$. In particular, this is the case if $D$ has a highest weight $\lambda$, or if $V$ itself has a highest weight $\lambda$.
\end{lem}

\begin{proof}
By Proposition \ref{prop:chainweights}, $\Lambda(V)$ is the disjoint union of the weight multisets for the composition factors of $V$. The subset $D$ of $\Lambda(V)$ is also a disjoint union of weight multisets for composition factors of $V$, and hence $\lambda \in \Lambda(U)$ for some composition factor $U$ of $V$. Theorem \ref{thm:highestweightirred} implies that $U = L(\mu)$, for some $\mu \in \Lambda(U) \subseteq D$ with $\lambda \le \mu$. By the definition of $\lambda$, we therefore have $\mu = \lambda$ and $U = L(\lambda)$.
\end{proof}

We can apply this lemma iteratively to a $K[G]$-module to increase the multiset of known composition factors of $V$, corresponding to $C$, and decrease the multiset of unknown composition factors, corresponding to $D$.

Now, for each $i \in \{1,2\}$, let $\mathbb{F}_i$ be the algebraic closure of a finite field, and let $H_i$ be a linear algebraic group of type $Y_\ell$ defined over $\mathbb{F}_i$, with maximal torus $T_i$. Additionally, let $\{\lambda_1,\ldots,\lambda_\ell\}$ and $\{\mu_1,\ldots,\mu_\ell\}$ be the fundamental dominant weights of $X(T_1)$ and $X(T_2)$, respectively, each ordered according to a fixed labelling of the Dynkin diagram of type $Y_\ell$. Through the rest of this section, we will identify two characters $\alpha \in X(T_1)$ and $\beta \in X(T_2)$ if the $\mathbb{Z}$-linear combinations $\alpha = \sum_{j=1}^\ell a_j \lambda_j$ and $\beta = \sum_{j=1}^\ell b_j \mu_j$ satisfy $a_j = b_j$ for all $j$. The following definition describes a related identification of modules for linear algebraic groups defined over different fields.

\begin{defn}
\label{def:fieldchangeequiv}
For each $i \in \{1,2\}$, let $V_i$ be an $\mathbb{F}_i[H_i]$-module with composition factor multiset $\mathcal{U}_i$. We say that $V_1$ and $V_2$ are \emph{compositionally equivalent} if there exists a bijection $\theta: \mathcal{U}_1 \to \mathcal{U}_2$ such that $\Lambda((U)\theta) = \Lambda(U)$ for all $U \in \mathcal{U}_1$.
\end{defn}

The following result is based on the discussion in \cite[\S3]{lubeck}. We will soon use this and the subsequent result to derive an important structural property of certain $K[G]$-modules.

\begin{lem}
\label{lem:weylmodule}
Let $\lambda$ be a dominant character of $T_1$. There exists an $\mathbb{F}_1[H_1]$-module $V(\lambda)$, called the \emph{Weyl module}\index{Weyl module} corresponding to $\lambda$, such that:
\begin{enumerate}[label={(\roman*)}]
\item $\dim(V(\lambda))$ and $\Lambda(V(\lambda))$ depend only on $Y$, $\ell$ and $\lambda$, and not on $\mathbb{F}_1$;
\item the irreducible $\mathbb{F}_1[H_1]$-module $L(\lambda)$ is a quotient of $V(\lambda)$; and
\item there exists a positive integer $r$ such that, if $\mathrm{char}(\mathbb{F}_1) \ge r$, then $V(\lambda) \cong L(\lambda)$.
\end{enumerate}
\end{lem}

We now return our focus to the linear algebraic group $G$ defined over $K = \overline{\mathbb{F}_p}$. In the following proposition, $w_0$ is the \emph{longest element} of $W$, i.e., the unique element of $W$ such that $\Delta^{w_0} = -\Delta$. Note that if $\lambda \in X(T)$ is dominant, then so is $-\lambda^{w_0}$ \cite[p.~125, p.~132]{malle}. In fact, $-\lambda^{w_0} = \lambda$ unless $G$ is equal to $A_\ell$ (with $\ell > 1$), $D_\ell$ (with $\ell$ odd) or $E_6$ \cite[p.~133]{malle}.

\begin{prop}[{\cite[p.~24, p.~118]{humphreysmod}}]
\label{prop:nonsplitkgmod}
Let $\lambda, \mu \in X(T)$ be dominant, and let $M$ be a $K[G]$-module whose composition factors are exactly $L(\lambda)$ and $L(\mu)$, with $L(\mu)$ isomorphic to a submodule $U$ of $M$. If $L(\lambda)$ is not isomorphic to a submodule $W$ of $M$, with $M = U \oplus W$, then:
\begin{enumerate}[label={(\roman*)}]
\item $\mu < \lambda$, and $M$ is isomorphic to a quotient of the Weyl module $V(\lambda)$; or
\item $\lambda < \mu$, and $M$ is isomorphic to a submodule of the dual module $(V(-\mu^{w_0}))^*$.
\end{enumerate}
\end{prop}

Let $S$ be a finite subset of $X(T)$, with each $\lambda \in S$ dominant. Lemma \ref{lem:weylmodule} implies that there exists a positive integer $r(S)$ such that, if we vary $p$ with the restriction that $p \ge r(S)$, then for each $\lambda \in S$, each of $\dim(L(\lambda))$ and $\Lambda(L(\lambda))$ do not depend on $p$. We can also construct a new $K[G]$-module $M$ using irreducible modules $L(\lambda)$ with $\lambda \in S$, via tensor products, exterior powers and quotients (of submodules constructed similarly). The weight multiset for $M$ can be expressed in terms of the weights of the modules $L(\lambda)$, as described in Lemma \ref{lem:tensorprodweights}, and so if we vary $p$ (again, with $p \ge r(S)$), then $\dim(M)$ and $\Lambda(M)$ remain fixed. On the other hand, the way in which the weights of $M$ are distributed among its composition factors (per Proposition \ref{prop:chainweights}) may vary with $p$. To rectify this, we define a new set $S_M$ as the union of $S$ and the dominant weights of $M$. We deduce from Lemma \ref{lem:highestweightfactor} that if we vary $p$ with $p \ge r(S_M)$, then $M$ remains fixed up to compositional equivalence. In the following theorem, we set $r:=r(S_M)$.

Recall also that a module $V$ (for any group, over any field) is \emph{multiplicity free} if it is a semisimple module with pairwise non-isomorphic irreducible submodules. Moreover, $V$ is multiplicity free if and only if its composition factors are pairwise non-isomorphic, with each isomorphic to a submodule of $V$.


\begin{thm}
\label{thm:semisimplekgmod}
Let $M$ be a $K[G]$-module constructed from irreducible $K[G]$-modules via tensor products, exterior powers and quotients. If $M$ is compositionally equivalent to the corresponding module defined over $\overline{\mathbb{F}_r}$, and if the composition factors of $M$ are pairwise non-isomorphic, then $M$ is multiplicity free.
\end{thm}

\begin{proof}
We may assume that $M$ is reducible. Let $\{0\} = M_0 \subset M_1 \subset \cdots \subset M_n = M$ be a composition series for $M$. Then for each $i \in \{1,2,\ldots,n\}$, Theorem \ref{thm:highestweightirred} implies that $M_i/M_{i-1} \cong L_i := L(\mu_i)$ for some dominant $\mu_i \in X(T)$. Since $M$ is compositionally equivalent to the corresponding module defined over $\overline{\mathbb{F}_r}$, Lemma \ref{lem:weylmodule} implies that $L_i \cong V(\mu_i)$ for each $i$. Thus $V(\mu_i)$ is irreducible. Furthermore, $L(-\mu_i^{w_0}) \cong L_i^*$ \cite[Proposition 16.1]{malle}, and so $\dim(L(-\mu_i^{w_0})) = \dim(L_i^*) = \dim(L_i)$ is equal to the dimension of the corresponding module defined over $\overline{\mathbb{F}_s}$ for any $s \ge r$. Hence $V(-\mu_i^{w_0}) \cong L(-\mu_i^{w_0})$ by Lemma \ref{lem:weylmodule}. Since the duals of isomorphic modules are isomorphic to each other, we have $(V(-\mu_i^{w_0}))^* \cong (L(-\mu_i^{w_0}))^*$, which is isomorphic to $(L_i^*)^*$. This is isomorphic to $L_i$ since $\dim(L_i)$ is finite, and thus $(V(-\mu_i^{w_0}))^*$ is irreducible.

Now, for a given $j \in \{2,3,\ldots,n\}$, the composition factors of $M_j/M_{j-2}$ are $M_j/M_{j-1} \cong L_j$ and $M_{j-1}/M_{j-2} \cong L_{j-1}$, with the latter isomorphic to a maximal submodule of $M_j/M_{j-2}$. Since $V(\mu_j)$ and $(V(-\mu_{j-1}^{w_0}))^*$ are irreducible, Proposition \ref{prop:nonsplitkgmod} implies that $L_j$ is also isomorphic to a maximal submodule of $M_j/M_{j-2}$. It follows that $M$ has a composition series $\{0\} = M_0 \subset M_1 \subset \cdots \subset M_{j-2} \subset U \subset M_j \subset \cdots \subset M_n = M,$ with $U/M_{j-2} \cong L_j$. This process of finding a new composition series can be performed a total of $(j-1)$ times in order to obtain a composition series whose smallest nonzero module is isomorphic to $L_j$. Hence each composition factor of $V$ is isomorphic to a submodule of $V$. As these composition factors are pairwise non-isomorphic, $M$ is multiplicity free.
\end{proof}

McNinch \cite[Corollary 1.1.1]{mcninch} proved a more uniform result about (not necessarily multiplicity free) semisimple modules: if $G$ is of type $Y_\ell$, then each $K[G]$-module of dimension at most $\ell p$ is semisimple. However, we will use Theorem \ref{thm:semisimplekgmod} to prove the semisimplicity of $K[G]$-modules whose dimensions are higher than this upper bound.

For the remainder of this section, we assume that $G$ is a linear algebraic group associated with an exceptional Chevalley group. We will shortly calculate the composition factors of the exterior square $A^2V$ of each minimal $K[G]$-module $V$, and in some cases, the composition factors of $(A^2V \otimes V)/A^3V$. Recall that $A^2V \cong L^2V$ if $p > 2$, and that $(A^2V \otimes V)/A^3V \cong L^3V$ if $p > 3$. Although we are mainly interested in the cases where these isomorphisms hold, we also include the small prime cases where they do not hold. Our calculations use data from a paper of L\"ubeck \cite{lubeck}, which, for each $G$, lists the dimension and highest weight of each irreducible $K[G]$-module whose dimension is below a certain value. We also use L\"ubeck's supplementary data \cite{lubeckws}, which, for each of these modules, lists the multiplicity of each of the module's dominant weights. Since each orbit of the action of the Weyl group $W$ on $X(T)$ contains exactly one dominant weight, and since weight multiplicity for a given module is constant on Weyl orbits, we can use this data to compute the multiplicity of every weight of the module. Note that L\"ubeck also considers classical linear algebraic groups of relatively small rank, and that in some cases, his supplementary data covers higher ranks and dimensions than his paper.

Table \ref{table:minmodweights} gives the highest weight and dimension of the minimal $K[G]$-modules for each $G$, up to isomorphism and twisting by a field automorphism of $G$, based on the data from \cite{lubeck}. Observe that the highest weight in each case is a fundamental dominant weight of $T$. Our ordering of the fundamental dominant weights of $T$, which corresponds to a certain labelling of the Dynkin diagram of $G$, follows the convention of Malle and Testerman \cite[Table 9.1]{malle} and the Magma \cite{magma} computer algebra system. In some cases, this is different from the ordering used by L\"ubeck \cite[Appendix A.1]{lubeck}. If $p \ne 3$, then the $K[G_2]$-module $L(\lambda_2)$ has dimension $14$, and if $p \ne 2$, then the $K[F_4]$-module $L(\lambda_1)$ has dimension $52$.

\begin{table}[ht]
\centering
\renewcommand{\arraystretch}{1.1}
\caption{The highest weights and dimensions of minimal $K[G]$-modules, up to isomorphism and twisting by a field automorphism of $G$.}
\label{table:minmodweights}
\begin{tabular}{ |c|c|c|c| }
\hline
$G$ & Highest weight & Condition on $p$ & Dimension\\
\hline
\hline
\multirow{2}{*}{$G_2$} & \multirow{2}{*}{$\lambda_1$} & $p = 2$ & $6$\\
 & & $p > 2$ & $7$\\
\hline
$G_2$ & $\lambda_2$ & $p=3$ & $7$\\
\hline
$F_4$ & $\lambda_1$ & $p=2$ & $26$\\
\hline
\multirow{2}{*}{$F_4$} & \multirow{2}{*}{$\lambda_4$} & $p = 3$ & $25$\\
 & & $p \ne 3$ & $26$\\
\hline
$E_6$ & $\lambda_1$ & All $p$ & $27$\\
\hline
$E_6$ & $\lambda_6$ & All $p$ & $27$\\
\hline
$E_7$ & $\lambda_7$ & All $p$ & $56$\\
\hline
$E_8$ & $\lambda_8$ & All $p$ & $248$\\
\hline
\end{tabular}
\end{table}

Each irreducible $K[G]$-module is self-dual, unless $G = E_6$ \cite[Appendix A.3]{lubeck}. If $G = E_6$ and $\lambda = \sum_{i=1}^6 c_i \lambda_i$, then the highest weight of $(L(\lambda))^*$ is $\sum_{i=1}^6 c_{(i)\tau} \lambda_i$, where $\tau$ is the permutation $(16)(35)$ associated with the nontrivial automorphism of the Dynkin diagram of $G$ \cite[Appendix A.3]{lubeck}. It follows that $(L(\lambda))^*$ is isomorphic to the module obtained by twisting $L(\lambda)$ by the graph automorphism of $E_6$. In particular, this automorphism interchanges $L(\lambda_1)$ and $L(\lambda_6) \cong (L(\lambda_1))^*$. Similarly, if $G = G_2$ and $p = 3$ (respectively, if $G = F_4$ and $p = 2$), then the graph automorphism of $G$ interchanges $L(\lambda_1)$ and $L(\lambda_2)$ (respectively, $L(\lambda_1)$ and $L(\lambda_4)$).

We use Magma to calculate the composition factors of $W \in \{A^2V,(A^2V \otimes V)/A^3V\}$, via the following process.
\begin{enumerate}[label={(\roman*)}]
\item Determine the weight multiset $\Lambda(V)$, e.g., by calculating the appropriate Weyl orbits from L\"ubeck's data.
\item Calculate the weight multiset $D:=\Lambda(W)$, using Lemma \ref{lem:tensorprodweights}.
\item Find a weight $\lambda \in D$ such that $\lambda \not < \mu$ for all $\mu \in D$.
\item Add $L(\lambda)$ to the list of composition factors of $W$ (per Lemma \ref{lem:highestweightfactor}).
\item Determine the weight multiset $\Lambda(L(\lambda))$ using L\"ubeck's data and Weyl orbit calculations.
\item Exclude each element of $\Lambda(L(\lambda))$ from $D$.
\item Repeat steps (iii)--(vi) until $D$ is empty.
\end{enumerate}

The following theorem summarises our results for $A^2V$ in each case. For completeness, we consider the cases where $(G,V) \in \{(G_2,L(\lambda_2)),(F_4,L(\lambda_1))\}$, even when $p$ is such that $V$ is not a minimal module. Here, and in the subsequent theorem related to ${(A^2V \otimes V)/A^3V}$, the statement that certain modules are multiplicity free is a consequence of Theorem \ref{thm:semisimplekgmod}.

\begin{thm}
\label{thm:a2vkgmod}
Suppose that $(G,V) \in \{(G_2,L(\lambda_2)),(F_4,L(\lambda_1))\}$, or that $V$ is a minimal $K[G]$-module whose highest weight is a fundamental dominant weight of $T$. Then the composition factors of $A^2V$, for each prime $p$, are listed in the tables in Appendix \ref{sec:compkga2v}. Furthermore, for the values of $p$ given in the final row of each table, $A^2V$ is multiplicity free. In particular, if $G = E_6$ and $p > 2$, then $A^2V$ is irreducible.
\end{thm}

Note that the first isomorphism in Table \ref{table:g2a2v2} is a consequence of Steinberg's Tensor Product Theorem (see \cite[Theorem 2.2]{lubeck}).

We will see later that the module $V$ is closely associated with an irreducible $\mathbb{F}_q[\hat G]$-module, with the Lie powers of the former associated with the Lie powers of the latter. In particular, this correspondence preserves irreducibility. In the case $q=p$, we would like to use Theorem \ref{thm:univap} to construct a $p$-group $P$ such that $\hat G$ is a subgroup of relatively small index of the group $A(P)$ induced by $\mathrm{Aut}(P)$ on $P/\Phi(P)$. As $L^2V$ is irreducible when $G = E_6$ and $p > 2$, we must therefore consider the structure of $L^3V$. Although $L^2V$ is not irreducible when $G = E_7$, we will see later in this paper that considering the structure of $L^3V$ is again necessary in this case. We therefore determine the composition factors of $(A^2V \otimes V)/A^3V$ when $G \in \{E_6,E_7\}$.

\begin{thm}
\label{thm:l3vkgmod}
Suppose that $G \in \{E_6,E_7\}$, and let $V$ be a minimal $K[G]$-module whose highest weight is a fundamental dominant weight of $T$. Then the composition factors of $(A^2V \otimes V)/A^3V$, for each prime $p$, are listed in the tables in Appendix \ref{sec:compkgl3v}. Furthermore, for the values of $p$ given in the final row of each table, $(A^2V \otimes V)/A^3V$ is multiplicity~free.
\end{thm}

\section{Minimal modules for exceptional Chevalley groups}
\label{sec:liepowerschev}

Here, we use the results of the previous section to study the minimal modules, and their Lie powers, for the exceptional Chevalley groups (or their universal covers). We retain the notation outlined at the start of the previous section, and we initially allow $\tilde G = {}^t Y_\ell(q)$ to be any finite simple group of Lie type. Additionally, let $u:=1$ if $\tilde G$ is a Suzuki or Ree group, i.e., if $t = 2$ and $Y \in \{B,G,F\}$, and otherwise let $u:=t$.

%

\begin{prop}
\label{prop:hatgquotients}
Suppose that $\tilde G$ is an exceptional group of Lie type, and that $\hat G \not\cong \tilde G$. Then $Z(\hat G)$ is the unique nontrivial proper normal subgroup of $\hat G$.
\end{prop}

\begin{proof}
As $\hat G$ is quasisimple, each proper normal subgroup of $\hat G$ lies in $Z(\hat G)$. Using the fact that $\hat G$ is not isomorphic to $\tilde G \cong \hat G/Z(\hat G)$, we see from \cite[Corollary 24.13]{malle} that $|Z(\hat G)|$ is prime, and the result follows.
\end{proof}

The above proof also applies to certain classical groups $\tilde G$; see \cite[Corollary 24.13]{malle}.

Our next proposition is from \cite[\S13]{steinbergyale} (see also \cite[Proposition 5.4.4]{kleidman}).

\begin{prop}
\label{prop:untwistedsplitting}
The field $\mathbb{F}_{q^u}$ is a splitting field for $\hat G$.
\end{prop}

Now, if $\tilde G$ is not a Suzuki or Ree group, then let $\mathcal{L}$ be the set of irreducible $K[G]$-modules $L(\lambda)$ such that $\lambda$ is \emph{$q$-restricted}, i.e., $\lambda = \sum_{i=1}^\ell c_i \lambda_i$ with $0 \le c_i < q$ for each $i$. Steinberg \cite[\S11, \S13]{steinbergend} gives the following theorem, as well as the slightly more complicated definition of $\mathcal{L}$ that is required when $\tilde G$ is a Suzuki or Ree group (see also \cite[p.~191]{kleidman}).

\begin{thm}
\label{thm:steinbqrest}
The irreducible $K[\hat G]$-modules are the restrictions to $\hat G$ of the irreducible $K[G]$-modules in $\mathcal{L}$. Furthermore, if $L(\lambda), L(\mu) \in \mathcal{L}$, then the restrictions of $L(\lambda)$ and $L(\mu)$ to $\hat G$ are isomorphic if and only if $\lambda = \mu$.
\end{thm}

We now consider the relationship between $K[G]$-modules and $\mathbb{F}_{q^u}[\hat G]$-modules. Here, if $V$ is an $\mathbb{F}_{q^u}[\hat G]$-module, then $V^K$ denotes the $K[\hat G]$-module constructed from $V$ by extending the scalars.

\begin{prop}
\label{prop:extrestrictiso}
For each $i \in \{1,2\}$, let $V_i$ be an $\mathbb{F}_{q^u}[\hat G]$-module, and let $W_i$ be a $K[G]$-module such that $V_i^K \cong W_i|_{\hat G}$. Then the following statements hold:
\begin{enumerate}[label={(\roman*)}]
\item \label{extrestr1} If $V_1 \le V_2$ and $W_1 \le W_2$ (and the embeddings are compatible with the above isomorphisms), then $(V_2/V_1)^K \cong (W_2/W_1)|_{\hat G}$.
\item \label{extrestr2} $(V_1 \otimes V_2)^K \cong (W_1 \otimes W_2)|_{\hat G}$.
\item \label{extrestr3} $(A^r(V_1))^K \cong (A^r(W_1))|_{\hat G}$, for each positive integer $r$.
\end{enumerate}
\end{prop}

\begin{proof}
Each quotient, tensor product and exterior power operation in \ref{extrestr1}--\ref{extrestr3} commutes with the extension of scalars of $\mathbb{F}_{q^u}[\hat G]$-modules to $K$, or with the restriction of $K[G]$-modules to $\hat G$, as appropriate.
\end{proof}

\begin{lem}
\label{lem:untwistedfinitemods}
Up to isomorphism, there is a 1-1 correspondence between irreducible $K[G]$-modules in $\mathcal{L}$ and irreducible $\mathbb{F}_{q^u}[\hat G]$-modules. Specifically, if $W \in \mathcal{L}$, and if $V$ is the corresponding irreducible $\mathbb{F}_{q^u}[\hat G]$-module, then $V^K \cong W|_{\hat G}$ as $K[\hat G]$-modules.
\end{lem}

\begin{proof}
Since $\mathbb{F}_{q^u}$ is a splitting field for $\hat G$ by Proposition \ref{prop:untwistedsplitting}, there is a 1-1 correspondence $V_i \longleftrightarrow V_i^K$ between distinct irreducible $\mathbb{F}_{q^u}[\hat G]$-modules and distinct irreducible $K[\hat G]$-modules \cite[Corollary 9.8]{ichar}. Theorem \ref{thm:steinbqrest} therefore implies that the restrictions to $\hat G$ of distinct elements of $\mathcal{L}$ are obtained from the distinct irreducible $\mathbb{F}_{q^u}[\hat G]$-modules by extending the scalars.
\end{proof}

\begin{lem}
\label{lem:finitemodscompfactors}
Let $W$ be a $K[G]$-module whose composition factors all lie in $\mathcal{L}$, and let $V$ be an $\mathbb{F}_{q^u}[\hat G]$-module such that $V^K \cong W|_{\hat G}$. Then the following statements hold:
\begin{enumerate}[label={(\roman*)}]
\item \label{finitemods1} The correspondence of Lemma \ref{lem:untwistedfinitemods} applies between the composition factors of $W$ and the composition factors of $V$, with multiplicities preserved.
\item \label{finitemods2} If $W$ is multiplicity free, then so are $W|_{\hat G}$ and $V$.
\end{enumerate}
\end{lem}


\begin{proof}\let\qed\relax \leavevmode
\begin{enumerate}[label={(\roman*)}]
\item Let $X$ and $Y$ be irreducible sections of $W$ and $V$, respectively. Then $X|_{\hat G}$ is a section of $W|_{\hat G}$, and $Y^K$ is a section of $V^K$. As $X \in \mathcal{L}$, Theorem \ref{thm:steinbqrest} implies that $X|_{\hat G}$ is irreducible, as is $Y^K$ since $\mathbb{F}_{q^u}$ is a splitting field for $\hat G$ by Proposition \ref{prop:untwistedsplitting}. Hence the composition factors of $W|_{\hat G}$ are the restrictions to $\hat G$ of the composition factors of $W$, and the composition factors of $V^K$ are obtained from the composition factors of $V$ by extending the scalars. Since $V^K \cong W|_{\hat G}$, the result follows.
\item As $W$ is multiplicity free, its composition factors appear as non-isomorphic submodules. These submodules restrict to non-isomorphic irreducible submodules of $W|_{\hat G}$ by Theorem \ref{thm:steinbqrest}, and hence $W|_{\hat G} \cong V^K$ is also multiplicity free. This implies that $V$ is semisimple \cite[Proposition 2.1.5]{bahturin}, i.e., $V$ is the direct sum of a set $\mathcal{U}$ of irreducible submodules of $V$. If $A,B \in \mathcal{U}$, then $A^K$ and $B^K$ are irreducible, and $A^K \not\cong B^K$ since $V^K$ is multiplicity free. It follows that $A \not\cong B$ \cite[Corollary 9.8]{ichar}, and therefore $V$ is multiplicity free. \hfill $\square$
\end{enumerate}
\end{proof}

We say that a module for an arbitrary group $H$ over a field $\mathbb{F}$ can be \emph{written over} a subfield $E$ of $\mathbb{F}$ if the module affords a representation that maps each element of $H$ to a matrix with entries in $E$. The following result describes the smallest field over which we can write an absolutely irreducible $\hat G$-module. We assume here that $\tilde G$ is not a Suzuki or Ree group; see \cite[Remark 5.4.7(b)]{kleidman} for information about these excluded cases.

\begin{prop}[{\cite[p.~193--194]{kleidman}}]
\label{prop:minfieldmod}
Suppose that $q = p^e$ for some positive integer $e$, and that $\tilde G$ is not a Suzuki or Ree group. Let $f$ be a positive integer, let $U$ be an absolutely irreducible $\mathbb{F}_{p^f}[\hat G]$-module that cannot be written over a proper subfield of $\mathbb{F}_{p^f}$, and let $W \in \mathcal{L}$ be the $K[G]$-module such that $W|_{\hat G} \cong U^K$. Additionally, if $t > 1$, then let $\gamma$ be the graph automorphism of $G$ of order $t$. Then either:
\begin{enumerate}[label={(\roman*)}]
\item $f \mid e$; $W \cong W^\gamma$ if $t > 1$; and there exists an irreducible $K[\hat G]$-module $V$ with $\dim(U) = \dim(V)^{e/f}$; or \label{minfield:case1}
\item $t > 1$; $f \mid te$; $f \nmid e$; $W \not\cong W^\gamma$; and there exists an irreducible $K[\hat G]$-module $V$ with $\dim(U) = \dim(V)^{te/f}$. \label{minfield:case2}
\end{enumerate}
\end{prop}


L\"ubeck's \cite{lubeck} lists of irreducible modules show that $\gamma$ does not fix the minimal $K[G]$-modules when $\tilde G \in \{{}^3 D_4(q),{}^2 E_6(q)\}$. Hence the corresponding minimal $\mathbb{F}_{q^u}[\hat G]$-modules (which are absolutely irreducibly by Proposition \ref{prop:untwistedsplitting}) cannot be written over $\mathbb{F}_q$ or any of its subfields. Furthermore, when $\tilde G$ is a Suzuki or Ree group, in which case $q > p$, the minimal $\mathbb{F}_q[\hat G]$-modules cannot be written over a proper subfield of $\mathbb{F}_q$ \cite[Remark 5.4.7(b)]{kleidman}. Thus no minimal $\hat G$-module can be written over $\mathbb{F}_p$ for any twisted exceptional group of Lie type $\tilde G$. This means that we cannot apply Theorem \ref{thm:univap} or Theorem \ref{thm:expclass2stab} to the images of the representations afforded by these modules. Of course, in the case $\tilde G \in \{{}^2 B_2(q),{}^2 F_4(q)\}$, the fact that $q$ is even also prohibits us from applying these theorems.


The following definition applies to modules for an arbitrary group.

\begin{defn}[{\cite[Ch.~1.8.2]{BHRD}}]
\label{def:quaseq}
Let $V$ and $W$ be modules for a group $H$ over the same field. We say that $V$ and $W$ are \emph{quasi-equivalent} if there exists an automorphism $\alpha$ of $H$ such that $W \cong V^\alpha$.
\end{defn}

Throughout the remainder of this section, we assume that:
\begin{nscenter}
\underline{$\tilde G$ is an exceptional Chevalley group}.
\end{nscenter}

Recall that if $\hat G \not\cong \tilde G$, then $\hat G$ is the universal cover of $\tilde G$. In this case, each automorphism of $\tilde G$ lifts to a unique automorphism of $\hat G$, and $\hat G$ has no other automorphisms \cite[Corollary 5.1.4]{gorensteincfsg3}. If $H \in \{\hat G, G\}$, and if $V$ and $W$ are minimal $\mathbb{F}_q[H]$-modules such that $V \cong W^{\beta}$ for some (possibly trivial) field automorphism $\beta$ of $H$, then we will say that $V$ and $W$ are \emph{$\mathcal{F}$-equivalent} and lie in the same \emph{$\mathcal{F}$-class} of minimal $\mathbb{F}_q[H]$-modules.

\begin{lem}
\label{lem:lieconjimages}
Let $d$ be the dimension of a minimal $K[G]$-module, as in Table \ref{table:minmodweights}. Then:
\begin{enumerate}[label={(\roman*)}]
\item \label{lieconj1} $d$ is the dimension of the minimal $\mathbb{F}_q[\hat G]$-modules, and there is a unique quasi-equivalence class $\mathcal{Q}$ of these modules;
\item \label{lieconj2} each module in $\mathcal{Q}$ is absolutely irreducible and faithful, and cannot be written over any proper subfield of $\mathbb{F}_q$;
\item \label{lieconj3} the images of the $\mathbb{F}_q$-representations afforded by modules in $\mathcal{Q}$ form a single conjugacy class of subgroups of $\mathrm{GL}(d,q)$;
\item \label{lieconj4} if $V \in \mathcal{Q}$, and if $\alpha$ is a nontrivial field automorphism of $\hat G$, then $V \not\cong V^\alpha$;
\item \label{lieconj5} if $\tilde G \ne E_6(q)$ and $(\tilde G,p) \notin \{(G_2(q),3),(F_4(q),2)\}$, then $\mathcal{Q}$ is also the unique $\mathcal{F}$-class of minimal $\mathbb{F}_q[\hat G]$-modules;
\item \label{lieconj6} if $\tilde G = E_6(q)$ or if $(\tilde G,p) \in \{(G_2(q),3),(F_4(q),2)\}$, then there are two $\mathcal{F}$-classes of minimal $\mathbb{F}_q[\hat G]$-modules, and they are interchanged by the graph automorphism of $\hat G$; and
\item \label{lieconj7} if $\tilde G = E_6(q)$, then a given minimal $\mathbb{F}_q[\hat G]$-module and its dual lie in different $\mathcal{F}$-classes of minimal $\mathbb{F}_q[\hat G]$-modules.
\end{enumerate}
\end{lem}

\begin{proof}
By Lemma \ref{lem:untwistedfinitemods}, the minimal $K[G]$-modules and the minimal $\mathbb{F}_q[\hat G]$-modules have the same dimension. Table \ref{table:minmodweights} lists all minimal $K[G]$-modules up to $\mathcal{F}$-equivalence, and the highest weight of each listed module is a fundamental dominant weight. It follows from Lemma \ref{lem:tensorprodweights} that each minimal $K[G]$-module that lies in $\mathcal{L}$ can be written as $L(\lambda)^{\phi^i} \cong L(p^i \lambda)$, where $\lambda$ is a fundamental dominant weight and $i \in \{0,1,2,\ldots,e-1\}$, with $q = p^e$. Moreover, $L(p^i \lambda) \not\cong L(p^j \lambda)$ if $i \ne j$ by Theorem \ref{thm:highestweightirred}. Observe that if $\alpha$ is an automorphism of $G$ that induces an automorphism $\beta$ of $\hat G$, then $\beta$ fixes an $\mathbb{F}_q[\hat G]$-module if and only if $\alpha$ fixes the corresponding $K[G]$-module. Since the field automorphisms of $G$ induce those of $\hat G$, it follows from Lemma \ref{lem:untwistedfinitemods} that the distinct (up to isomorphism, or up to $\mathcal{F}$-equivalence) minimal $\mathbb{F}_q[\hat G]$-modules are those that correspond to the distinct (in the same way) minimal $K[G]$-modules in $\mathcal{L}$, which gives \ref{lieconj4}. In particular, if $\tilde G = E_6(q)$ or if $(\tilde G,p) \in \{(G_2(q),3),(F_4(q),2)\}$, then there are two $\mathcal{F}$-classes of these $\hat G$-modules. Otherwise, there is a unique $\mathcal{F}$-class, and hence \ref{lieconj5} follows.

Let $V$ be a minimal $K[G]$-module. If $\tilde G = E_6(q)$, then the highest weight of $V$ is not a scalar multiple of the highest weight of $V^*$. Hence $V$ and $V^*$ are not $\mathcal{F}$-equivalent, and \ref{lieconj7} easily follows. We also saw previously that if $\tilde G = E_6(q)$ or $(\tilde G,p) \in \{(G_2(q),3),(F_4(q),2)\}$, then the graph automorphism of $G$ does not fix $V$ up to $\mathcal{F}$-equivalence. As this graph automorphism induces the graph automorphism of $\hat G$, we obtain \ref{lieconj6}. The aforementioned $\mathcal{F}$-classes of minimal $\mathbb{F}_q[\hat G]$-modules therefore form a single quasi-equivalence class, proving \ref{lieconj1}. In all cases, the quasi-equivalence of modules implies \ref{lieconj3} \cite[p.~39--40]{BHRD}.

If $\hat G$ is isomorphic to the simple group $\tilde G$, then it is immediate that each minimal $\mathbb{F}_q[\hat G]$-module is faithful. In fact, even when $\hat G \not\cong \tilde G$, which can occur only if $\tilde G \in \{E_6(q),E_7(q)\}$, the minimal $\mathbb{F}_q[\hat G]$-modules are faithful \cite[p.~202--203]{kleidman}. Moreover, each minimal $\mathbb{F}_q[\hat G]$-module is absolutely irreducible by Proposition \ref{prop:untwistedsplitting}. Finally, as there are no nontrivial irreducible $K[G]$-modules of dimension less than $d$, Proposition \ref{prop:minfieldmod} implies that no minimal $\mathbb{F}_q[\hat G]$-module can be written over a proper subfield of $\mathbb{F}_q$, yielding \ref{lieconj2}.
\end{proof}

In fact, Theorem \ref{thm:steinbqrest} implies that the dimension of a minimal $K[\hat G]$-module is equal to the dimension of a minimal $K[G]$-module.

\begin{prop}
\label{prop:faithfultildegmod}
Let $d$ be the dimension of a minimal $K[G]$-module, as in Table \ref{table:minmodweights}. Additionally, let $H \ne 1$ be a quotient of $\hat G$, and let $V$ be a faithful $\mathbb{F}_q[H]$-module of dimension at most $d$. Then $H = \hat G$, and $V$ is irreducible of dimension $d$.
\end{prop}

\begin{proof}
If $H$ is not equal to the quasisimple group $\hat G$, then $H = \hat G/Z(\hat G) \cong \tilde G$ by Proposition \ref{prop:hatgquotients}. Thus $H$ is perfect in each case, and so $V$ has a nontrivial composition factor $U$ (see \cite[Corollary 5.3.3]{gorenstein}). Since $\dim(V) \le d$, Lemma \ref{lem:lieconjimages} shows that we must have $H = \hat G$ and $\dim(U) = d$, i.e., $U = V$.
%
\end{proof}

Recall that if $V$ is a nonzero module over a field of characteristic not equal to $2$, then $L^2V \cong A^2V$, and if the characteristic of the field is also not equal to $3$, then $L^3V \cong (A^2V \otimes V)/A^3V$. We will now describe the submodule structure of $L^2V$, where $V$ is a minimal $\mathbb{F}_q[\hat G]$-module with $p > 2$, and the structure of $L^3V$ when $\tilde G \in \{E_6(q),E_7(q)\}$, with $p > 3$. In most cases, we also show that the submodule structure of each Lie power is equivalent (in terms of containments and dimensions) to the submodule structure of the corresponding $K[G]$-module and $K[\hat G]$-module. We will say that two modules with equivalent submodule structures are \emph{$\mathcal{S}$-equivalent}. As the submodule structure of the exterior square of a module is important in many applications, we use the notation $A^2V$ in the following theorem, instead of $L^2V$.

For the proof of the following theorem, recall that if $\mathcal{M}$ is the set of irreducible submodules of a multiplicity free module $M$, then isomorphism of modules gives a 1-1 correspondence between the composition factors of $M$ and the submodules in $\mathcal{M}$. Furthermore, the set of submodules of $M$ is exactly $\{\bigoplus_{N \in \mathcal{N}} N \mid \mathcal{N} \subseteq \mathcal{M}\}$.

\begin{thm}
\label{thm:finiteextsubmods}
Let $V$ be a minimal $\mathbb{F}_q[\hat G]$-module, with $p$ odd, and let $W \in \mathcal{L}$ be the irreducible $K[G]$-module corresponding to $V$, as in Lemma \ref{lem:untwistedfinitemods}.
\begin{enumerate}[label={(\roman*)}]
\item The submodule structure\footnote{The dimensions of these modules' composition factors are also given in the tables in Appendix \ref{sec:compkga2v}.} of $A^2V$ is given in Table \ref{table:a2vstructfinite}, and $A^2V$ is $\mathcal{S}$-equivalent to each of $A^2W$ and $(A^2W)|_{\hat G}$. If $p$ is an ``exceptional prime'' for $\tilde G$, i.e., if $(\tilde G, p)$ lies in the set $\{(G_2(q),3),(F_4(q),3),(E_7(q),7),(E_8(q),3),(E_8(q),5)\},$ then $A^2V$ is uniserial. Otherwise, $A^2V$ is multiplicity free. In particular, if $\tilde G = E_6(q)$, then $A^2V$ is irreducible. \label{finitestructa2v}
\item If $\tilde G \in \{G_2(q),E_8(q)\}$ and $B \in \{V,W,W|_{\hat G}\}$, then the quotient of $A^2B$ by its largest maximal submodule is isomorphic to $B$.
\item Suppose that $\tilde G \in \{E_6(q),E_7(q)\}$ and that $p > 3$, with $q = p$ if $\tilde G = E_6(q)$ and $p = 5$, or if $\tilde G = E_7(q)$ and $p \in \{7,11,19\}$. The submodule structure\footnote{The dimensions of composition factors here are also given in the tables in Appendix \ref{sec:compkgl3v}.} of $L^3V$ is given in Figures \ref{fig:e6l3vsub} and \ref{fig:e7l3vsub}. If $p$ is not a listed ``exceptional prime'' for $\tilde G$, then $L^3V$ is multiplicity free, and is $\mathcal{S}$-equivalent to each of of $L^3W$ and $(L^3W)|_{\hat G}$. \label{finitestructl3v}
\item For a fixed group $\tilde G$ and a fixed prime $p > 3$, the dimensions of the composition factors of $L^3V$ do not depend on $q$ or on the choice of $V$.
\end{enumerate}
\end{thm}

\begin{proof}
Lemma \ref{lem:lieconjimages} implies that there is a unique quasi-equivalence class $\mathcal{Q}$ of minimal $\mathbb{F}_q[\hat G]$-modules. If $T \in \mathcal{Q}$ and if $\alpha \in \mathrm{Aut}(\hat G)$, then $A^2(T^\alpha) = (A^2T)^\alpha$, and hence all modules in the set $\{A^2T \mid T \in \mathcal{Q}\}$ are quasi-equivalent. In particular, all modules in this set are $\mathcal{S}$-equivalent. Similarly, this holds for all modules in the set $\{L^3T \mid T \in \mathcal{Q}\}$. Hence this structure does not depend on the choice of $V$.

Observe from Proposition \ref{prop:extrestrictiso} that $(A^2W)|_{\hat G} \cong (A^2V)^K$ and $(L^3W)|_{\hat G} \cong (L^3V)^K$, when these modules are defined. Furthermore, if $(\tilde G,W) \ne (G_2(3),L(\lambda_2))$, then the tables in Appendices \ref{sec:compkga2v} and \ref{sec:compkgl3v} show that each composition factor of $A^2W$ or $L^3W$ lies in $\mathcal{L}$. If instead $(\tilde G,W) = (G_2(3),L(\lambda_2))$, then the only composition factor of $A^2W$ that does not lie in $\mathcal{L}$ is $L(3\lambda_1) \cong L(\lambda_1)^\phi$. The field automorphism $\phi$ in this case induces the trivial automorphism of $\hat G$, and it follows that $(L(3\lambda_1))|_{\hat G} \cong (L(\lambda_1))|_{\hat G}$. Thus Lemma \ref{lem:finitemodscompfactors}\ref{finitemods1}, with a very slightly modified argument in the case $(\tilde G,W) = (G_2(3),L(\lambda_2))$, implies that the composition factors of $A^2V$ or $L^3V$ correspond to those of $A^2W$ or $L^3W$, respectively. In particular, for a fixed group $\tilde G$ and a fixed prime $p$, the dimensions of the composition factors of $A^2V$ or of $L^3V$ do not depend on $q$ or on the choice of $V$, nor do the composition factors of $A^2W$, $L^3W$ or their restrictions to $\hat G$.

Suppose now that $(\tilde G, p) \notin \{(G_2(q), 3),(F_4(q), 3),(E_7(q),7),(E_8(q),3),(E_8(q),5)\}$. By Theorem \ref{thm:a2vkgmod}, the module $A^2W$ is multiplicity free, and it is also irreducible when $\tilde G = E_6(q)$. Furthermore, if $\tilde G = E_6(q)$ and $p > 5$, or if $\tilde G = E_7(q)$ and $p \notin \{3,7,11,19\}$, then $L^3W$ is multiplicity free by Theorem \ref{thm:l3vkgmod}. In each case, Lemma \ref{lem:finitemodscompfactors} implies that $(A^2W)|_{\hat G}$, $(L^3W)|_{\hat G}$, $A^2V$ and $L^3V$ are irreducible or multiplicity free as specified. The submodule structure of each multiplicity free module follows from the composition factors of the $K[G]$-modules given in the tables in Appendices \ref{sec:compkga2v} and \ref{sec:compkgl3v}.

Next, suppose that $\tilde G = Y_\ell(q)$, where $(Y_\ell, p) \in \{(G_2, 3),(F_4, 3),(E_7,7),(E_8,3),(E_8,5)\}$. Additionally, let $\tilde X:= Y_\ell(p)$, let $\hat X$ be the simply connected version of $\tilde X$, and let $U$ be a minimal $\mathbb{F}_p[\hat X]$-module. Then $G$ is the linear algebraic group associated with $\tilde X$. Recall that $(A^2V)^K \cong (A^2W)|_{\hat G}$, and similarly note that $(A^2U)^K \cong (A^2W)|_{\hat X}$. Since $\hat X \le \hat G$ \cite[Lemma 5.1.6]{BHRD}, it follows that $(A^2U)^K \cong ((A^2V)^K)|_{\hat X}$. Computations using Magma show that the submodule structure of $A^2U$ is as required, and in particular, that $A^2U$ is uniserial. As $K$ is algebraic over $\mathbb{F}_p$, which is a splitting field for $\hat X$, it follows that $(A^2U)^K$ is $\mathcal{S}$-equivalent to $A^2U$. Since the uniserial module $(A^2U)^K$ is the restriction to $\hat X$ of each of $A^2W$ and $(A^2V)^K$, and since these three modules have equivalent composition factors (in terms of dimensions), they are all $\mathcal{S}$-equivalent. Since $\mathbb{F}_q$ is a splitting field for $\hat G$, the module $(A^2V)^K$ is also $\mathcal{S}$-equivalent to $A^2V$.

Now, let $\tilde G \in \{G_2(q),E_8(q)\}$, and let $S$ be the largest maximal submodule of $A^2W$. Tables \ref{table:g2a2v1}, \ref{table:g2a2v2} and \ref{table:e8a2v} list the composition factors of $W$, with $W$ defined up to $\mathcal{F}$-equivalence. It follows from these tables and Lemma \ref{lem:tensorprodweights} that either $(A^2W)/S \cong W$, or $(\tilde G,p) = (G_2(q),3)$ and the two composition factors of $S$ have the same highest weight. However, we have shown that $S$ is uniserial, and hence Proposition \ref{prop:nonsplitkgmod} implies that its two composition factors have different highest weights. Hence $(A^2W)/S \cong W$. It follows easily that $W|_{\hat G} \cong (A^2W)|_{\hat G}/S|_{\hat G} \cong ((A^2V)/N)^K$, where $S|_{\hat G}$ is the largest maximal submodule of $(A^2W)|_{\hat G}$, and $N$ is the largest maximal submodule of $A^2V$. As both $V$ and $(A^2V)/N$ correspond to $W$ via the 1-1 correspondence of Lemma \ref{lem:untwistedfinitemods}, we have $A^2V/N \cong V$.

Additional Magma computations were used to directly determine the submodule structure of $L^3V$ in the case $\tilde G = E_6(5)$. We also used Magma in the case $\tilde G = E_7(7)$ to show that $L^3V$ contains submodules of dimension $56$, $51072$ and $7392$, with the third a uniserial module whose nonzero proper submodules have dimension $6480$ and $912$, respectively. Writing $U_k$ to denote a $k$-dimensional submodule of $L^3V$, it follows from the dimensions of the composition factors of $L^3V$ that $U_{56}$ and $U_{51072}$ are irreducible, and that $L^3V = U_{56} \oplus U_{51072} \oplus U_{7392}$. No two of these three direct summands have a common composition factor, and so the submodules of $L^3V$ are exactly the direct sums of the submodules of these direct summands. This fact also yields the containments between the submodules of $L^3V$. Similar computations were performed in the case $\tilde G = E_7(11)$, where $L^3V = U_{56} \oplus U_{912} \oplus U_{67552}$; and in the case $\tilde G = E_7(19)$, where $L^3V = U_{912} \oplus U_{51072} \oplus U_{6536}$. We provide additional details about these Magma computations in Appendix \ref{sec:magmacalc}. 
\end{proof}

\begin{table}[h]
\centering
\caption{The submodule structure of $A^2V$, where $V$ is a minimal $\mathbb{F}_q[\hat G]$-module, with $q$ a power of a prime $p > 2$.}
\label{table:a2vstructfinite}
\def\arraystretch{1.25}
\begin{tabular}{ |c|c|c|c|c| }
\hline
$\tilde G$ & $d$ & {\parbox{2.1cm}{\centering \vspace{.15cm}Exceptional\\primes\vspace{.15cm}}} & {\parbox{3.5cm}{\centering \vspace{.15cm}Standard structure\\of $A^2V$\vspace{.15cm}}} & {\parbox{3.8cm}{\centering \vspace{.15cm}Structure of $A^2V$ for\\exceptional primes\vspace{.15cm}}} \\
\hline
\hline
\raisebox{-.5\height}{$G_2(q)$} & \raisebox{-.5\height}{7} & \raisebox{-.5\height}{3} &
\raisebox{-.5\height}{\includegraphics{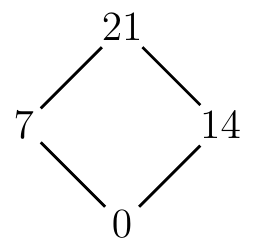}} &
\raisebox{-.5\height}{\includegraphics{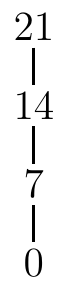}} \\
\hline
\raisebox{-.5\height}{$F_4(q)$} & {\parbox{2.2cm}{\centering \vspace{.35cm}$25$, if $p = 3$\\$26$, if $p > 3$\vspace{.15cm}}} & \raisebox{-.5\height}{3} &
\raisebox{-.5\height}{\includegraphics{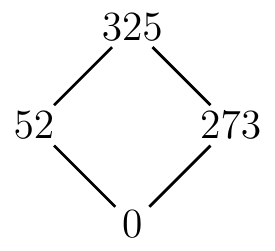}}&
\raisebox{-.5\height}{\includegraphics{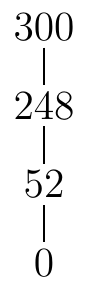}}\\
\hline
\raisebox{-.5\height}{$E_6(q)$} & \raisebox{-.5\height}{27} & \raisebox{-.5\height}{None} &
\raisebox{-.5\height}{\includegraphics{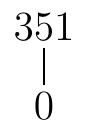}}&
\raisebox{-.5\height}{N/A} \\
\hline
\raisebox{-.5\height}{$E_7(q)$} & \raisebox{-.5\height}{56} & \raisebox{-.5\height}{7} &
\raisebox{-.5\height}{\includegraphics{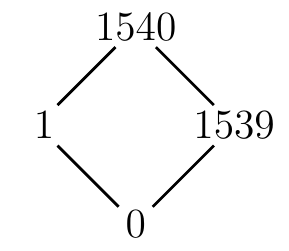}}&
\raisebox{-.5\height}{\includegraphics{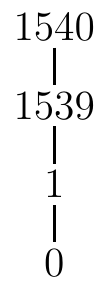}}\\
\hline
\raisebox{-.5\height}{$E_8(q)$} & \raisebox{-.5\height}{248} & \raisebox{-.5\height}{3, 5} &
\raisebox{-.5\height}{\includegraphics{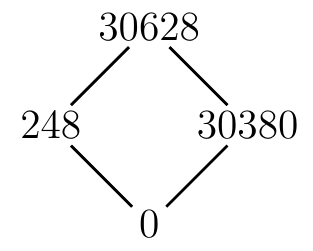}}&
\raisebox{-.5\height}{\includegraphics{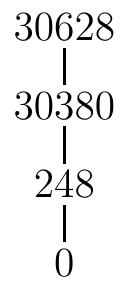}}\\
\hline
\end{tabular}
\end{table}

\begin{figure}[b]
\centering
\begin{subfigure}[b]{0.475\textwidth}
\centering
\includegraphics{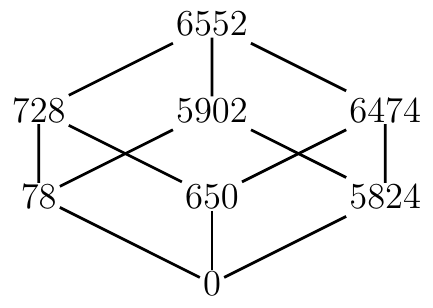}
\caption{$q$ a power of a prime $p > 5$}
\end{subfigure}
\begin{subfigure}[b]{0.475\textwidth}
\centering
\includegraphics{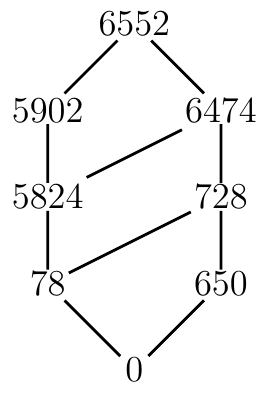}
\caption{$q = 5$}
\end{subfigure}
\caption{The submodule structure of $L^3V$, where $V$ is a minimal $\mathbb{F}_q[\hat G]$-module, with $\tilde G = E_6(q)$.}
\label{fig:e6l3vsub}
\end{figure}

\begin{figure}[t]
\centering
\begin{subfigure}[b]{0.475\textwidth}
\centering
\includegraphics{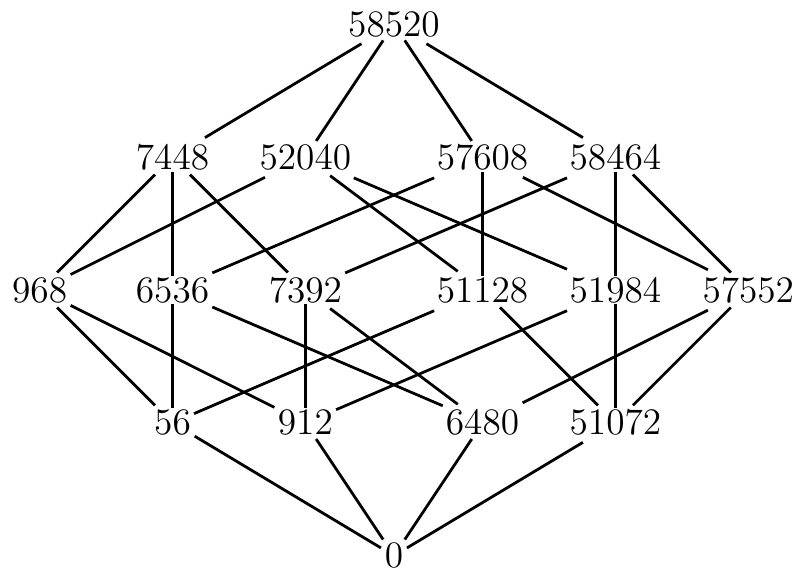}
\caption{$q$ a power of a prime $p \notin \{2,3,7,11,19\}$}
\end{subfigure}
\begin{subfigure}[b]{0.475\textwidth}
\centering
\includegraphics{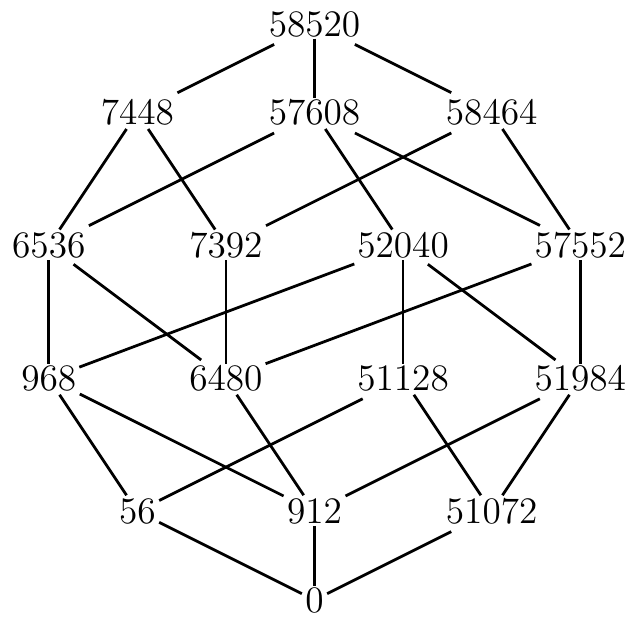}
\caption{$q = 7$}
\end{subfigure}
\vskip\baselineskip
\begin{subfigure}[b]{0.475\textwidth}
\centering
\includegraphics{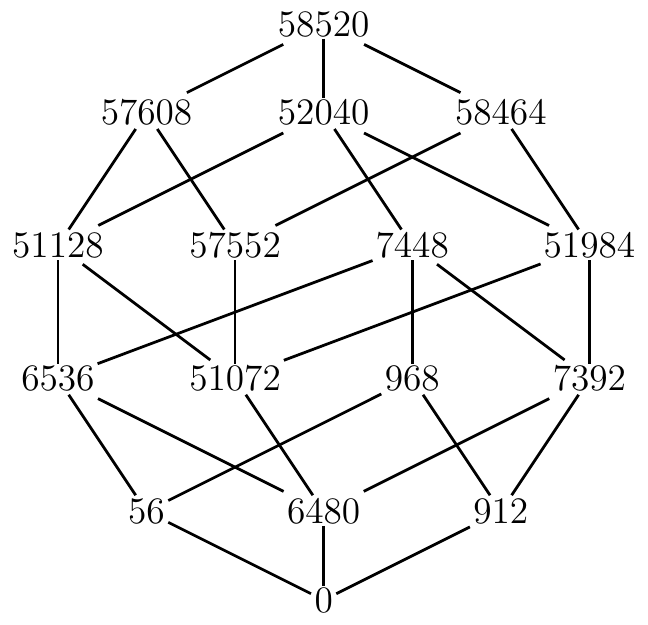}
\caption{$q = 11$}
\end{subfigure}
\begin{subfigure}[b]{0.475\textwidth}
\centering
\includegraphics{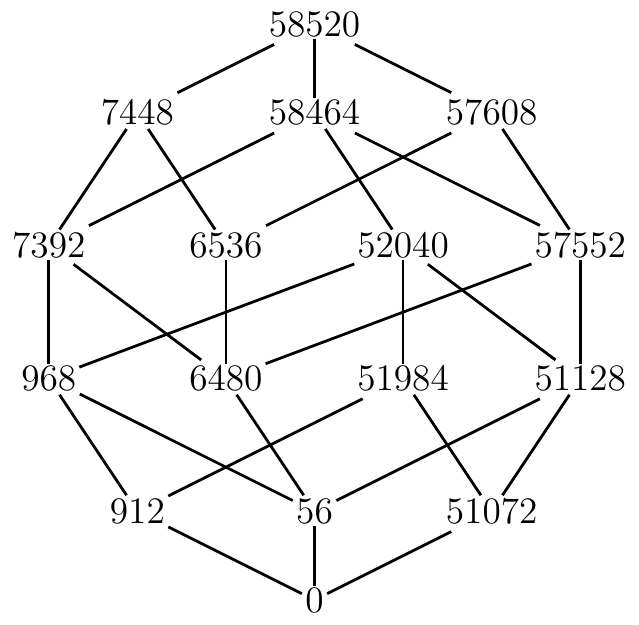}
\caption{$q = 19$}
\end{subfigure}
\caption{The submodule structure of $L^3V$, where $V$ is a minimal $\mathbb{F}_q[\hat G]$-module, with $\tilde G = E_7(q)$.}
\label{fig:e7l3vsub}
\end{figure}

As mentioned in \S\ref{sec:intro}, the submodule structure of $A^2V$ in the case $\tilde G = G_2(q)$ has been explored, in less detail, in \cite{sdfg2} and \cite[Ch.~9.3.2]{schroder}. It is also known that there exists an irreducible $7$-dimensional module over $\mathbb{R}$ for the group $G_2(\mathbb{R})$ whose exterior square is the direct sum of two irreducible submodules, of dimension $7$ and $14$, respectively \cite[p.~9]{gow}.

Observe that the dimensions of the submodules of each $\hat G$-module described in Theorem \ref{thm:finiteextsubmods} do not depend on $p$, excluding the case $(\tilde G,p) = (F_4(q),3)$ where $d$ takes an exceptional value. Furthermore, containments between these submodules in the exceptional prime cases are exactly those that are allowed by the dimensions of the module's composition factors. This observation, in the case of the smaller modules, inspired the above proof of the submodule structure of $L^3V$ with $\tilde G = E_7(p)$ and $p$ an exceptional prime. Our observation also leads to the following conjecture.

\begin{conj}
\label{conj:allq}
Theorem \ref{thm:finiteextsubmods}\ref{finitestructl3v} holds if $q \ne p$, and $L^3V$ is $\mathcal{S}$-equivalent to each of $L^3W$ and $(L^3W)|_{\hat G}$, even if $p$ is an exceptional prime for $\tilde G$.
\end{conj}

Suppose now that $p$ is an exceptional prime for $\tilde G \in \{E_6(q),E_7(q)\}$, with $\tilde X$ the corresponding group of Lie type defined over $\mathbb{F}_p$, $\hat X$ the simply connected version of $\tilde X$, and $U$ a minimal $\mathbb{F}_p[\hat X]$-module. We can adapt the proof of the uniserial case of Theorem \ref{thm:finiteextsubmods}\ref{finitestructa2v} to show that if $N$ is a uniserial submodule of $L^3U$, then $N^K$ is uniserial. As $L^3U$ is a direct sum of uniserial submodules with no common composition factors (even when $\tilde X = E_6(5)$), it follows that $L^3U$ and $(L^3U)^K$ are $\mathcal{S}$-equivalent. Hence the submodule lattice of each of $L^3W$, $(L^3V)^K \cong (L^3W)|_{\hat G}$ and $L^3V$ is equivalent to a sublattice of the submodule lattice of $L^3U$. In fact, as all composition factors of $L^3W$ lie in $\mathcal{L}$, the length of a composition series is the same for all of these modules by Lemma \ref{lem:finitemodscompfactors}. In particular, if $L^3W$ stabilises a submodule of the same dimension as each of the uniserial direct summands of $L^3U$, then so do $(L^3V)^K$ and $L^3V$. In this case, Conjecture \ref{conj:allq} holds.

\section{Overgroups of exceptional Chevalley groups}
\label{sec:overgroupschev}

We now determine some of the overgroups of the simply connected versions of the exceptional Chevalley groups in the general linear groups corresponding to their minimal modules. In the next section, we will use these results to deduce which overgroups are distinguishable from each exceptional Chevalley group in the context of Theorems \ref{thm:univap} and \ref{thm:expclass2stab}. Throughout the remainder of this paper, we use the following notation:
\begin{itemize}
\item $q$ is a power of an odd prime $p$;
\item $\tilde G$ is an exceptional Chevalley group defined over $\mathbb{F}_q$;
\item $\hat G$ is the simply connected version of $\tilde G$;
\item $V$ is a minimal $\mathbb{F}_q[\hat G]$-module;
\item $d:=\dim(V)$; and
\item $\zgl := Z(\mathrm{GL}(d,q))$.
\end{itemize}
By Lemma \ref{lem:lieconjimages}, $d$ is given in Table \ref{table:minmodweights}. Although this lemma applies when $p = 2$, we will not consider this case, as some of the important results in this section hold only when $p$ is odd. Lemma \ref{lem:lieconjimages} also shows that all minimal $\mathbb{F}_q[\hat G]$-modules are absolutely irreducible and faithful, and that there is a unique $\mathrm{GL}(d,q)$-conjugacy class of images of representations afforded by these modules. Hence no result in this section depends on the choice of $V$, and we can identify $\hat G$ with the image in $\mathrm{GL}(d,q)$ of the corresponding representation. In fact, Proposition \ref{prop:faithfultildegmod} shows that we can choose $V$ to be any faithful $d$-dimensional $\mathbb{F}_q[\hat G]$-module. Furthermore, we have from Lemma \ref{lem:lieconjimages} that no conjugate of $\hat G$ in $\mathrm{GL}(d,q)$ can be written over a proper subfield of $\mathbb{F}_q$. Recall also that $\hat G$ is quasisimple, and thus $\hat G$ lies in $(\mathrm{GL}(d,q))' = \mathrm{SL}(d,q)$. The following is another important fact about $\hat G$.

\begin{prop}
\label{prop:excepcentre}
The centraliser of $\hat G$ in $\mathrm{GL}(d,q)$ is equal to $\zgl$. In particular, $Z(\hat G)$ is a subgroup of $\zgl$.
\end{prop}

\begin{proof}
The centre of $\hat G$ is a subgroup of the centraliser $C_{\mathrm{GL}(d,q)}(\hat G)$. As $\hat G$ is an absolutely irreducible subgroup of $\mathrm{GL}(d,q)$, it follows from Schur's Lemma (see \cite[p.~38]{BHRD}) that this centraliser is equal to $\zgl$.
\end{proof}

We will use the following lemma many times in this section, often without reference. Here, if $\beta$ is a bilinear or unitary form on a vector space $U$, then (using the notation of Definition \ref{def:formisom}) we say that a subgroup $H$ of $\mathrm{GL}(U)$ \emph{preserves} $\beta$ if $H \le I(\beta)$, and that $H$ \emph{preserves $\beta$ up to scalars} if $H \le \Delta(\beta)$.

\begin{lem}[{\cite[p.~40--41]{BHRD}}]
\label{lem:absirredpresform}
Let $U$ be a vector space over $\mathbb{F}_q$, and let $H$ be a perfect, absolutely irreducible subgroup of $\mathrm{GL}(U)$. Suppose also that no conjugate of $H$ in $\mathrm{GL}(U)$ can be written over a proper subfield of $\mathbb{F}_q$, and that $\beta$ is a nonzero bilinear or unitary form on $U$ with $H \le \Delta(\beta)$. Then:
\begin{enumerate}[label={(\roman*)}]
\item $\beta$ is non-degenerate;
\item $H \le I(\beta)$;
\item $H$ preserves no nonzero bilinear or unitary forms on $U$ other than scalar multiples of $\beta$; and
\item $N_{\mathrm{GL}(U)}(H) \le \Delta(\beta)$.
\end{enumerate}
\end{lem}

In particular, if $\hat G$ preserves a nonzero bilinear form on $V$ up to scalars, then the above lemma holds with $H = \hat G$ and $U = V$.

\begin{prop}
\label{prop:hatgform}
If $\tilde G = E_7(q)$, then $\hat G$ preserves a non-degenerate alternating form on $V$. If $\tilde G = E_6(q)$, then $\hat G$ preserves no nonzero unitary or reflexive bilinear form on $V$. Otherwise, $\hat G$ preserves a non-degenerate orthogonal form on $V$.
\end{prop}

\begin{proof}
The result follows by \cite[p.~200]{kleidman} for $\tilde G \in \{G_2(q),F_4(q),E_8(q)\}$ and by \cite[Proposition 5.4.18]{kleidman} for $\tilde G = E_7(q)$. If $\tilde G = E_6(q)$, then Lemma \ref{lem:lieconjimages} implies that $V^* \not\cong V^\alpha$ for any (trivial or not) field automorphism $\alpha$ of $\hat G$, and \cite[Lemma 2.10.15]{kleidman} yields the~result.
\end{proof}

Through the rest of this section, let $\beta$ denote the zero form if $\tilde G = E_6(q)$, and otherwise, let $\beta$ denote the non-degenerate bilinear form preserved by $\hat G$. If $\beta$ is the zero form, then $I(\beta) = \mathrm{GL}(d,q)$. If $\beta$ is alternating, then we write $\mathrm{Sp}(d,q):=I(\beta)$ and $\mathrm{CSp}(d,q):=\Delta(\beta)$. Finally, if $\beta$ is orthogonal of type $\varepsilon \in \{\circ,+,-\}$, then we write $\mathrm{GO}^\varepsilon(d,q):=I(\beta)$ and $\mathrm{SO}^\varepsilon(d,q):=S(\beta)$, while $\Omega^\varepsilon(d,q)$ denotes $I(\beta)^\infty$, i.e., the last subgroup in the derived series of $I(\beta)$. If $\varepsilon = \circ$, i.e., if $d$ is odd, then we omit the superscript $\circ$. Since $\hat G$ is quasisimple, we have $\hat G \le \Omega^\varepsilon(d,q)$ if $\tilde G \in \{G_2(q),F_4(q),E_8(q)\}$.

Note that if $\tilde G = F_4(q)$ with $p > 3$, then $\hat G \cong \tilde G$ is the automorphism group of the $27$-dimensional \emph{Albert algebra}, i.e., the algebra of $3 \times 3$ octonion Hermitian matrices, over $\mathbb{F}_q$ \cite[Ch.~4.8]{wilson}. Here, we can choose $V$ to be the $26$-dimensional subspace of the algebra consisting of trace $0$ matrices. In this case, if $A$ and $B$ are matrices in $V$, then $(A,B)\beta$ is the trace of $\frac{1}{2}(AB+BA)$. By considering the matrix of this form, it can be shown that $\beta$ is of plus type if $q \equiv {1,7 \pmod {12}}$, and of minus type otherwise. However, we will not use this fact in this paper.

Throughout the remainder of this section, we use the following notation:
\begin{itemize}
\item $T$ is equal either to $\mathrm{GL}(d,q)$, or to $I(\beta)$;
\item $S:=T^\infty$; and
\item $R$ is an arbitrary subgroup of $T$ that contains $S$.
\end{itemize}
In particular (as $d \ge 7$), if $T = \mathrm{GL}(d,q)$, then $S = \mathrm{SL}(d,q)$; if $T = \mathrm{Sp}(d,q)$, then $S = T$; and as above, if $T = \mathrm{GO}^\varepsilon(d,q)$, then $S = \Omega^\varepsilon(d,q)$. In each case, $S \le \mathrm{SL}(d,q)$.

Our next goal is to determine the maximal subgroups of $S$ that contain $\hat G$. In order to do so, we will consider the Aschbacher classes of subgroups of $T$, and the related Aschbacher's Theorem \cite{aschbacherthm}, which we describe below. Our descriptions here are based on those of Bray, Holt and Roney-Dougal \cite[Ch. 2]{BHRD}. Note that some adjustments can be made here to consider classical groups of dimension less than $7$ or of even characteristic, or to include the case where $\beta$ is a non-degenerate unitary form on $V$.

\begin{defn}
\label{def:geomsub}
The \emph{geometric subgroups} of $R$ are the subgroups that belong to (at least) one of the following classes.
\begin{enumerate}[label={$\mathcal{C}_\arabic*:$}]
\item Stabilisers of certain nonzero proper subspaces of $V$.
\item Stabilisers of direct sum decompositions of $V$ into proper equidimensional subspaces.
\item Stabilisers of extension fields $\mathbb{F}_{q^r}$ of $\mathbb{F}_q$, for primes $r$ dividing $d$.
\item Stabilisers of decompositions $V = V_1 \otimes V_2$, where each $V_i$ is a vector space over $\mathbb{F}$ equipped with a zero or non-degenerate reflexive bilinear form $\beta_i$, and $(V_1,\beta_1)$ is not similar to $(V_2,\beta_2)$.
\item Stabilisers of subfields $\mathbb{F}_{q^{1/r}}$ of $\mathbb{F}_q$, for primes $r$.
\item Normalisers of symplectic-type or extraspecial $r$-subgroups of $R$, for primes $r \ne p$ such that $d$ is a power of $r$.
\item Stabilisers of decompositions $V = V_1 \otimes \cdots \otimes V_n$ with $n > 1$, where each $V_i$ is a vector space over $\mathbb{F}$ equipped with a zero or non-degenerate reflexive bilinear form $\beta_i$, with $(V_i,\beta_i)$ similar (and in particular, equidimensional) to $(V_j,\beta_j)$ for all $i,j$.
\item Full groups of similarities of non-degenerate unitary or reflexive bilinear forms (when $T = \mathrm{GL}(d,q)$).
\end{enumerate}
\end{defn}

Note that each geometric subgroup of $R$ is the full stabiliser in $R$ of some object; we do not consider any proper subgroup of this stabiliser as a geometric subgroup associated with the same object.

Recall that a group $H$ is \emph{almost simple} if there exists a non-abelian simple group $X$ satisfying $X \trianglelefteq H \le \mathrm{Aut}(X)$.

\begin{defn}
\label{def:classc9}
A subgroup $H$ of $R$ lies in \emph{class $\mathcal{C}_9$} if all of the following hold:
\begin{enumerate}[label={(\roman*)}]
\item $H/(H \cap \zgl)$ is almost simple;
\item $H$ does not contain $S$;
\item $H^\infty$ is absolutely irreducible;
\item no conjugate of $H^\infty$ in $\mathrm{GL}(d,q)$ can be written over a proper subfield of $\mathbb{F}_q$;
\item $H^\infty$ preserves no nonzero unitary or reflexive bilinear form if $T = \mathrm{GL}(d,q)$;
\item $H^\infty$ preserves no nonzero unitary or orthogonal form if $T$ preserves a non-degenerate alternating form; and
\item $H^\infty$ preserves no nonzero unitary or alternating form if $T$ preserves a non-degenerate orthogonal form.
\end{enumerate}
\end{defn}

Note that Bray, Holt and Roney-Dougal write $\mathscr{S}$ to denote the Aschbacher class $\mathcal{C}_9$, but we use notation consistent with Bamberg et al.~\cite{bamberg}. The subgroup classes $\mathcal{C}_1, \ldots, \mathcal{C}_9$ of $R$ are the \emph{Aschbacher classes}\index{Aschbacher class} of $R$.

\begin{thm}[Aschbacher's Theorem]
\label{thm:aschthm}
Suppose that a maximal subgroup $H$ of $R$ does not contain $S$ and is not a geometric subgroup of $R$. Then $H$ is a $\mathcal{C}_9$-subgroup of $R$.
\end{thm}

In this paper, when we say ``a maximal $\mathcal{C}_i$-subgroup of $R$'', we mean a maximal subgroup of $R$ that is also a $\mathcal{C}_i$-subgroup of $R$. All maximal subgroups of $S$ have been classified by Bray, Holt and Roney-Dougal \cite{BHRD} for $d \le 12$. Note that Kleidman \cite{kleidmanphd} previously presented a classification of the maximal geometric subgroups of $S$ for $d \le 12$, but without proof. Additionally, Kleidman and Liebeck \cite{kleidman} classified the maximal geometric subgroups for all $d > 12$, while Schr\"oder \cite{schroder} classified the maximal $\mathcal{C}_9$-subgroups for $d \in \{13,14,15\}$. However, there is no known method of classifying the maximal $\mathcal{C}_9$-subgroups uniformly for all $d$ \cite[p.~2]{BHRD}.

We now determine exactly when a geometric subgroup or a maximal $\mathcal{C}_9$-subgroup of $S$ can contain $\hat G$ (although some of our results apply more generally to subgroups of $R$). The next result follows from Definitions \ref{def:geomsub} and \ref{def:classc9}, and from the fact that if $X$ is a $\mathcal{C}_3$-subgroup of $T$, then $X^\infty$ is not absolutely irreducible \cite[p.~56]{BHRD}.

\begin{prop}
\label{prop:cinclusion}
Let $H$ be a subgroup of $R$ such that $H^\infty$ is absolutely irreducible and such that no conjugate of $H^\infty$ can be written over a proper subfield of $\mathbb{F}_q$. Then $H$ does not lie in any $\mathcal{C}_1$-, $\mathcal{C}_3$- or $\mathcal{C}_5$-subgroup of $R$. In particular, this holds if $H = \hat G$, or if $H$ is a $\mathcal{C}_9$-subgroup of $R$.
\end{prop}

\begin{lem}
\label{lem:hatginc2}
No $\mathcal{C}_2$-subgroup of $R$ contains $\hat G$.
\end{lem}

\begin{proof}
Suppose that $H$ is a $\mathcal{C}_2$-subgroup of $R$ that contains $\hat G$. Then $H$ stabilises a decomposition of $V$ as $V = V_1 \oplus \cdots \oplus V_m$, where each $V_i$ is a subspace of $V$ of dimension $d/m < d$. In particular, $H$ permutes the components of this decomposition, and hence there exists a permutation representation $\rho$ from $H$ to the symmetric group $S_m$. Since $|\tilde G|$ does not divide $|S_m|$ for any $m \le d$, it follows that $|\tilde G|$ does not divide the order of the quotient $(\hat G)\rho$ of $\hat G$. As $\tilde G \cong \hat G/Z(\hat G)$, Proposition \ref{prop:hatgquotients} implies that $(\hat G)\rho = 1$.

Now, we can identify $\ker \rho$ with a subgroup of $B:=\mathrm{GL}(V_1) \times \cdots \times \mathrm{GL}(V_m)$ \cite[Ch.~2.2.2]{BHRD}. Let $\pi_i$ be the projection map from $B$ to $\mathrm{GL}(V_i)$ for each $i \in \{1,\ldots,m\}$. As $\hat G \le \ker \rho$, there exists $j \in \{1,\ldots,m\}$ such that the quotient $(\hat G)\pi_j$ of $\hat G$ is nontrivial. Moreover, $(\hat G)\pi_j$ is a subgroup of $\mathrm{GL}(V_j) \cong \mathrm{GL}(d/m,q)$. However, Proposition \ref{prop:faithfultildegmod} shows that no nontrivial quotient of $\hat G$ is a subgroup of $\mathrm{GL}(n,q)$ for any $n < d$. This is a contradiction, and thus $\hat G$ lies in no $\mathcal{C}_2$-subgroup of $R$.
\end{proof}

\begin{lem}
\label{lem:hatginc4}
No $\mathcal{C}_4$-subgroup of $R$ contains $\hat G$.
\end{lem}

\begin{proof}
Suppose that $H$ is a $\mathcal{C}_4$-subgroup of $R$ that contains $\hat G$. Then $H$ lies in the central product $X:=\mathrm{GL}(m,q) \, \circ \, \mathrm{GL}(n,q)$, where\footnote{The inequality $m < \sqrt{d}$ holds because, in each of our cases where $d$ is even, $\sqrt{d}$ is not an integer.} $1 < m < \sqrt{d}$ and $mn = d$ \cite[Ch.~2.2.4]{BHRD}. Bamberg et al.~\cite[Lemma 5.5]{bamberg} show that $X$ stabilises a subspace of $A^2V$ of dimension $\binom{m}{2}\binom{n+1}{2}$, and so the subgroup $\hat G$ of $X$ also stabilises this subspace. However, Theorem \ref{thm:finiteextsubmods} shows that $A^2V$ contains no $\hat G$-submodule of dimension $\binom{m}{2}\binom{n+1}{2}$, for any permitted values of $m$ and $n$. This is a contradiction, and hence $\hat G$ lies in no $\mathcal{C}_4$-subgroup of $R$.
\end{proof}

We now consider the remaining geometric subgroups in the case $R = S$.

\begin{lem}
\label{lem:hatgingeom}
If $\tilde G \ne E_6(q)$, then $\Sigma(\beta)$ is a maximal $\mathcal{C}_8$-subgroup of $\mathrm{SL}(d,q)$, and the only geometric subgroup of $\mathrm{SL}(d,q)$ that contains $\hat G$. If $\tilde G = E_6(q)$, or if $S \ne \mathrm{SL}(d,q)$, then $\hat G$ does not lie in any geometric subgroup of $S$.
\end{lem}

\begin{proof}
By Proposition \ref{prop:cinclusion} and Lemmas \ref{lem:hatginc2} and \ref{lem:hatginc4}, we only need to consider subgroups of $S$ that lie in $\mathcal{C}_6 \cup \mathcal{C}_7 \cup \mathcal{C}_8$. Definition \ref{def:geomsub} shows that $\mathcal{C}_6$-subgroups of $S$ are only defined when $d$ is a power of a prime, i.e., $d \in \{7,25,27\}$, and that $\mathcal{C}_7$-subgroups of $S$ are only defined when $d$ is a power of a positive integer less than $d$, i.e., $d \in \{25,27\}$. If $H_6$ is a $\mathcal{C}_6$-subgroup of $S$, then $S = \mathrm{SL}(d,q)$ and $H_6$ has shape $(A \circ r_{+}^{1+m}).\mathrm{Sp}(m,r)$, where $|A| \le 27$ and $(d,r,m) \in \{(7,7,2),(25,5,4),(27,3,6)\}$, and where $r_{+}^{1+m}$ is the extraspecial group of order $r^{1+m}$ and exponent $r$ \cite[Table 2.9]{BHRD}. Additionally, if $H_7$ is a $\mathcal{C}_7$-subgroup of $S$, then the shape of $H_7$ is either $B.\mathrm{PSL}(n,q)^t.C.S_t$ or $\Omega(n,q)^t.D.S_t$, where $|B| \le 5$, $|C| \le 125$, $|D| \le 4$ and $(d,n,t) \in \{(25,5,2),(27,3,3)\}$ \cite[Table 2.10]{BHRD}. By considering group orders in each case, we see that $\hat G$ does not lie in $H_6$ or in $H_7$.

We now see from Definition \ref{def:geomsub} that if $\hat G$ lies in a geometric subgroup $H$ of $S$, then $S = \mathrm{SL}(d,q)$ and $H$ is a $\mathcal{C}_8$-subgroup of $S$, i.e., $H=\Sigma(\gamma)$ for some non-degenerate unitary or reflexive bilinear form $\gamma$. Lemma \ref{lem:absirredpresform} then implies that $\hat G \le S(\gamma)$. In particular, $\tilde G \ne E_6(q)$ by Proposition \ref{prop:hatgform}. Moreover, $\gamma$ must be a scalar multiple of $\beta$, i.e., $H = \Sigma(\beta)$.  Finally, it follows from \cite[Table 8.35]{BHRD} (when $d = 7$) and \cite[Proposition 7.8.1, Lemma 8.1.6]{kleidman} (in the remaining cases) that $\Sigma(\beta)$ is indeed maximal in $\mathrm{SL}(d,q)$.
\end{proof}

In the proof of the following theorem, if $H$ is a subgroup of $\mathrm{GL}(d,q)$, then we write $\overline H$ to denote the subgroup $\zgl H/\zgl$ of $\mathrm{PGL}(d,q)$.

\begin{thm}
\label{thm:excepmax}
Suppose that $S = I(\beta)^\infty$. Then the $\mathcal{C}_9$-subgroup $N_S(\hat G)$ of $S$ is the unique maximal subgroup of $S$ that contains $\hat G$. Furthermore, if $\tilde G \ne E_6(q)$, then $\Sigma(\beta)$ is the unique maximal subgroup of $\mathrm{SL}(d,q)$ that contains $\hat G$.
\end{thm}

\begin{proof}
By Theorem \ref{thm:aschthm} and Lemma \ref{lem:hatgingeom}, it suffices to show that if a maximal $\mathcal{C}_9$-subgroup of $S$ contains $\hat G$, then that maximal subgroup is equal to $N_S(\hat G)$, and that if $\tilde G \ne E_6(q)$, then no maximal $\mathcal{C}_9$-subgroup of $\mathrm{SL}(d,q)$ contains $\hat G$. Let $U \in \{S,\mathrm{SL}(d,q)\}$, and suppose that $H$ is a maximal $\mathcal{C}_9$-subgroup of $U$ that contains $\hat G$. The perfect group $\hat G$ then lies in $H^\infty$, which is easily shown to be quasisimple, with $\hinfbar \cong H^\infty/Z(H^\infty)$. 
Since $\tilde G \cong \hat G/Z(\hat G)$, which is equal to $\hat G/(\hat G \cap \zgl) \cong \overline{\hat G}$ by Proposition \ref{prop:excepcentre}, $\tilde G$ is isomorphic to a subgroup of $\hinfbar$. Note also that $H^\infty$ is a covering group of $\hinfbar$, and is therefore a quotient of the universal cover $Q$ of $\hinfbar$.

Definition \ref{def:classc9} implies that $H^\infty$ is absolutely irreducible and cannot be written over a proper subfield of $\mathbb{F}_q$. The tables in \cite{hiss} can be used to determine the finite, quasisimple, absolutely irreducible subgroups $A$ of $\mathrm{GL}(d,q)$ such that $A/Z(A)$ is not a simple group of Lie type defined over a field of characteristic $p$. In each case, there is no such group $A$ with $|A|$ divisible by $|\hat G|$. Hence $\hinfbar$ is a simple group of Lie type ${}^t Y_\ell(r)$, where $r$ is some power of $p$, with $(t,Y) \notin \{(2,B),(2,F)\}$ as $p > 2$. We will show that in fact $\hinfbar \cong \tilde G$.

Recall that the simply connected version $\hat X$ of $\hinfbar$ is the quotient of $Q$ by the Sylow $p$-subgroup of $Z(Q)$. Using \cite[Theorem 5.1.4]{kleidman} and the fact that $\hinfbar$ contains a copy of $\tilde G$, we see that if $\hat X \not\cong Q$, then $(\tilde G,d) = (G_2(3),7)$ and either $\hinfbar \cong \tilde G$, as claimed, or $\hinfbar \cong \Omega(7,3)$. However, this final case cannot occur, as here $\hinfbar \cong S$, and $\mathrm{SL}(7,3)$ contains no maximal $\mathcal{C}_9$-subgroup \cite[Table 8.36]{BHRD}. Thus we may assume that $H^\infty$ is a quotient of $\hat X \cong Q$. In particular, there exists an absolutely irreducible $d$-dimensional $\mathbb{F}_q[\hat X]$-module that cannot be written over a proper subfield of $\mathbb{F}_q$. It follows from Lemma \ref{lem:untwistedfinitemods} that there exists an irreducible $d$-dimensional $K[X]$-module, where $K:=\overline{\mathbb{F}_p}$ and $X$ is the linear algebraic group associated with $\hinfbar$. We also have from Proposition \ref{prop:minfieldmod} that either $\hinfbar = {}^2 G_2(r)$ or $r \in \{q,q^{1/t},q^2,q^{2/t},q^3,q^{3/t}\}$. Here, we have used the fact that if $d = x^y$ for integers $x > 0$ and $y > 1$, then the tuple $(d,x,y)$ lies in the set $\{(25,5,2),(27,3,3)\}$. In particular, if $\hinfbar \ne {}^2 G_2(r)$, then $r \in \{q^2,q^{2/t}\}$ is only possible if $(\tilde G,p) = (F_4(q),3)$, and $r \in \{q^3,q^{3/t}\}$ is only possible if $\tilde G = E_6(q)$.

Suppose first that $\hinfbar$ is a classical group of Lie type of dimension $b$. Using \cite[Proposition 5.4.11]{kleidman} and L\"ubeck's \cite{lubeck} lists of irreducible modules for linear algebraic groups, we determine that in order for the aforementioned $\mathbb{F}_q[\hat X]$-module and $K[X]$-module to exist, and in order for $|\tilde G|$ to divide $\hinfbar$, we require $b = d$. Moreover, the aforementioned $\mathbb{F}_q[\hat X]$-module is the unique minimal $\mathbb{F}_q[\hat X]$-module up to quasi-equivalence. If $\hinfbar = \mathrm{PSU}(d,r)$, then this module corresponds to the irreducible $\overline{\mathbb{F}_p}[X]$-module $L(a\lambda_i)$, where $a$ is a power of $p$ and $i \in \{1,d-1\}$; and if $\hinfbar = \mathrm{P}\Omega^-(d,r)$, then the module corresponds to $L(a\lambda_1)$ \cite[\S1]{liebeck85} (see also \cite[Appendix A.3]{lubeck}). By considering the permutation that the involutory graph automorphism of $X$ induces on the associated Dynkin diagram, we see from Proposition \ref{prop:minfieldmod} that $r = q^{1/2}$ in the former case and $r = q$ in the latter case. Since there is a unique minimal $\mathbb{F}_q[\hat X]$-module up to quasi-equivalence, there is a unique conjugacy class of subgroups $B$ of $\mathrm{GL}(d,q)$ such that $\zgl \le B$ and $\overline B \cong \hinfbar$ \cite[p.~39--40]{BHRD}. It follows that $\zgl H^\infty$ is conjugate in $\mathrm{GL}(d,q)$ to $\zgl M$, where $$M \in \{\mathrm{SL}(d,q),\mathrm{SU}(d,q^{1/2}),\mathrm{Sp}(d,q),\Omega(d,q),\Omega^\pm(d,q)\}.$$ As $M$ is perfect \cite[Proposition 1.10.3]{BHRD}, the perfect group $H^\infty$ is in fact conjugate in $\mathrm{GL}(d,q)$ to $M$. By Lemma \ref{lem:absirredpresform}, $H^\infty$ preserves at most one nonzero unitary or bilinear form. It follows that if $\tilde G \ne E_6(q)$ and $U$ preserves $\beta$, then $H^\infty = U$. If instead $U = \mathrm{SL}(d,q)$, then Definition \ref{def:classc9} implies that $H^\infty$ preserves no nonzero unitary or reflexive bilinear form, and we again have $H^\infty = U$. In each case, this contradicts the maximality of $H$ in $U$. Thus $\hinfbar$ is not a classical group, i.e., it is an exceptional group of Lie type.

L\"ubeck's \cite{lubeck} lists of irreducible modules imply that if $\tilde G = Y'_{\ell'}(q)$, then $Y'_{\ell'} = Y_\ell$ or $$(Y'_{\ell'},Y_\ell) \in \{(F_4,G_2),(E_6,G_2),(E_7,D_4),(E_8,G_2)\}.$$ In particular, if $\hinfbar = {}^2 G_2(r)$, then $p = 3$, and hence $d \in \{7,27\}$ \cite[Appendix A.49]{lubeck}. It follows from \cite[Remark 5.4.7(b)]{kleidman} that $r = q$ in this case. In each other case, there is no irreducible $K[X]$-module of dimension $2$ or $3$ \cite{lubeck}, and hence $r \in \{q,q^{1/t}\}$ by Proposition \ref{prop:minfieldmod}. In particular, since the graph automorphism of $E_6$ does not fix any minimal $K[E_6]$-module (see Table \ref{table:minmodweights}), Proposition \ref{prop:minfieldmod} implies that if $\tilde G = E_6(q)$ and $\hinfbar ={}^2 E_6(r)$, then $r = q^{1/2}$. Considering group orders in each case shows that we must have $(t,Y,\ell,r) = (1,Y',\ell',q)$, and hence $\hinfbar \cong \tilde G$, as claimed.

We have shown that $\zgl \hat G/\zgl = \overline{\hat G} \cong \tilde G \cong \hinfbar = \zgl H^\infty/\zgl$. As $\hat G \le H^\infty$, we have $\zgl \hat G = \zgl H^\infty$. Since $\hat G$ and $H^\infty$ are perfect, it follows that $\hat G = H^\infty$, and hence $H \le N_U(\hat G)$. Furthermore, since $\hat G < U$, and since $\overline {U}$ is simple, $\zgl U$ does not normalise $\zgl \hat G$. This implies that $U$ does not normalise $\hat G$, and thus $H = N_U(\hat G)$ by the maximality of $H$ in $U$. If $U = S$, then we are done. If instead $U = \mathrm{SL}(d,q)$, then $N_{U}(\hat G)$ lies in $\Sigma(\beta)$ by Lemma \ref{lem:absirredpresform}, and so is not a $\mathcal{C}_9$-subgroup of $U$ by Definition \ref{def:classc9}. Hence in this case, no maximal $\mathcal{C}_9$-subgroup of $U$ contains $\hat G$, as required.
\end{proof}

Note that Bray, Holt and Roney-Dougal \cite[Table 8.40]{BHRD} previously showed that $G_2(q)$ is a maximal $\mathcal{C}_9$-subgroup of $\Omega(7,q)$.

We now determine the normaliser of $\hat G$ in $\mathrm{GL}(d,q)$.

\begin{lem}
\label{lem:excepnormaliser}
Let $N:=N_{\mathrm{GL}(d,q)}(\hat G)$. If $\tilde G = E_6(q)$, then $N$ is equal to $(\zgl \hat G).(q-1,3)$, a $\mathcal{C}_9$-subgroup of $\mathrm{GL}(d,q)$. If $\tilde G = E_7(q)$, then $N$ is equal to $(\zgl \hat G).2$, a proper subgroup of $\mathrm{CSp}(56,q)$. Otherwise, $N$ is equal to $\zgl \hat G$, a proper, non-maximal subgroup of $\mathrm{GL}(d,q)$.
\end{lem}

\begin{proof}
First note that $\zgl \hat G \trianglelefteq N$. Let $\theta: N \to \mathrm{Aut}(\hat G)$ be the action of $N$ on $\hat G$ induced by conjugation. Then $(\hat G)\theta = \mathrm{Inn}(\hat G)$, and $\ker \theta = C_N(\hat G)$, which is equal to $\zgl$ by Proposition \ref{prop:excepcentre}. Hence $\theta$ induces a monomorphism from $N/(\zgl \hat G)$ to $\mathrm{Out}(\hat G)$. We therefore obtain the structure of $N$ by determining when an ``outer automorphism'' $\alpha \in \mathrm{Aut}(G)$ lies in $(N)\theta$. Since $N = N_{\Delta(\beta)}(\hat G)$ (using Lemma \ref{lem:absirredpresform} in the case $\beta \ne 0$), and since $V$ is a faithful, absolutely irreducible module for the quasisimple group $\hat G$, this occurs exactly when $V^\alpha \cong V$. By Lemma \ref{lem:lieconjimages} and \cite[Proposition 5.1.1, Proposition 5.1.9(i)]{BHRD}, this is equivalent to $\alpha$ being a diagonal automorphism of $\hat G$. The group $\hat G$ has two nontrivial diagonal automorphisms if $\tilde G = E_6(q)$ and $q \equiv 1 \pmod 3$; one if $\tilde G = E_7(q)$; and zero otherwise (see \cite[Ch.~4]{wilson}). Thus the structure of $N$ is as required.

Suppose now that $\tilde G = E_6(q)$. Since $N/(\zgl \hat G)$ is soluble, we have $N^\infty = (\zgl \hat G)^\infty$, which is equal to $\hat G^\infty = \hat G$. Lemma \ref{lem:lieconjimages} and Proposition \ref{prop:hatgform} show that $N^\infty$ satisfies all properties required by Definition \ref{def:classc9} for $N$ to be a $\mathcal{C}_9$-subgroup of $\mathrm{GL}(d,q)$. We also see, by considering group orders, that $N$ does not contain $\mathrm{SL}(d,q)$. Moreover, it is easy to show that $N/(N \cap \zgl) = N/\zgl$ is isomorphic to $(\hat G/Z(\hat G)).(q-1,3)$, which is the almost simple group $\tilde G.(q-1,3)$. Thus $N$ is a $\mathcal{C}_9$-subgroup of $\mathrm{GL}(d,q)$. Finally, if $\tilde G \ne E_6(q)$, then by considering group orders, we see that $N < \Delta(\beta) < \mathrm{GL}(d,q)$.
\end{proof}

Dedekind's Identity now implies that $N_S(\hat G)$ lies in the set $\{Z(S) \hat G,(Z(S) \hat G).3\}$ if $\tilde G = E_6(q)$ with $q \equiv {1 \pmod 3}$; $N_S(\hat G)$ lies in the set $\{Z(S) \hat G,(Z(S) \hat G).2\}$ if $\tilde G = E_7(q)$; and $N_S(\hat G) = Z(S)\hat G$ otherwise.

\begin{prop}
\label{prop:hatgspomegaover}
Suppose that $\tilde G \ne E_6(q)$, and that $S = I(\beta)^\infty$. Additionally, let $H$ be a subgroup of $\mathrm{SL}(d,q)$ that contains $\hat G$. Then $S \not\le H$ if and only if $H \le N_{\mathrm{SL}(d,q)}(\hat G)$.
\end{prop}

\begin{proof}
Suppose that $S \not\le H$. Then $H \cap S$ is a proper subgroup of $S$ containing $\hat G$, and so Theorem \ref{thm:excepmax} implies that $H \cap S$ lies in the maximal subgroup $N_S(\hat G)$ of $S$. Since $\hat G$ is perfect, we have $\hat G \le (H \cap S)^\infty \le N^\infty$, where $N:=N_{\mathrm{GL}(d,q)}(\hat G)$. Lemma \ref{lem:excepnormaliser} shows that $N/(\zgl \hat G)$ is soluble, and hence $N^\infty = (\zgl \hat G)^\infty$, which is equal to $\hat G^\infty = \hat G$. Therefore, $\hat G$ is the characteristic subgroup $(H \cap S)^\infty$ of $H \cap S$. Furthermore, Theorem \ref{thm:excepmax} implies that the proper subgroup $H$ of $\mathrm{SL}(d,q)$ lies in $\Sigma(\beta)$, which normalises $S$ by \cite[p.~14]{kleidman}. Hence $H \cap S \trianglelefteq H$, and so $H$ normalises the characteristic subgroup $\hat G$ of $H \cap S$, as required. Conversely, as $N_S(\hat G) < S$, no subgroup of $N_{\mathrm{SL}(d,q)}(\hat G)$ contains $S$.
\end{proof}

\section{Stabilisers of submodules of Lie powers}
\label{sec:liepowersubmodstab}

In this section, we determine the stabilisers in $\mathrm{GL}(d,q)$ of relevant subspaces of the Lie powers of $\mathbb{F}_q[\hat G]$-modules given in Theorem \ref{thm:finiteextsubmods}. We retain the notation outlined at the start of the previous section. In particular, $V$ is a minimal $\mathbb{F}_q[\hat G]$-module.

Recall that if $\tilde G = E_7(q)$, then $\hat G$ lies in the symplectic group $\mathrm{Sp}(56,q)$. In \S\ref{sec:highestweights}, we claimed that applying Theorem \ref{thm:univap} to the $\hat G$-submodules of $L^3V$ (with $q = p$) is more useful than applying this theorem to the $\hat G$-submodules of $L^2V$. The following lemma yields the reason for this.

\begin{lem}
\label{lem:sp56lrv}
Let $\tilde G = E_7(q)$.
\begin{enumerate}[label={(\roman*)}]
\item The group $\mathrm{CSp}(56,q)$ stabilises each $\hat G$-submodule of $L^2V$.
\item Suppose that $p > 3$. Then each of $\mathrm{Sp}(56,q)$ and $\mathrm{CSp}(56,q)$ stabilises exactly two nonzero proper subspaces of $L^3V$, of dimension $56$ and $58464$, respectively. If $p = 19$, then the latter subspace contains the former, and otherwise, $L^3V$ splits as the direct sum of these subspaces.
\end{enumerate}
\end{lem}

\begin{proof}
Let $\tilde X$ be the simple group of Lie type $\mathrm{PSp}(56,q) = C_{28}(q)$. The associated (simple, simply connected) linear algebraic group over $K:=\overline{\mathbb{F}_p}$ is $X = C_{28}$, and the (finite) simply connected version of $\tilde X$ is $\hat X = \mathrm{Sp}(56,q)$ \cite[p.~193]{malle}. The $56$-dimensional $K[X]$-module $U$ with highest weight\footnote{The weight associated with $C_{28}$ that we write as $\lambda_i$ is written as $\lambda_{29-i}$ by L\"ubeck \cite{lubeck}.} $\lambda_1$ is self-dual \cite[p.~132--133]{malle}, and it follows from \cite[\S1]{liebeck85} that $U$ is the unique minimal $K[X]$-module up to isomorphism and twisting by a field automorphism. Lemma \ref{lem:untwistedfinitemods} then implies that there is a unique (in the same way) minimal $\mathbb{F}_q[\hat X]$-module $W$, of dimension $56$. Hence the irreducible $56$-dimensional $\hat G$-module $V$ is isomorphic to $W|_{\hat G}$. Similarly, $L^rV \cong (L^rW)|_{\hat G}$ for each $r \in \{2,3\}$.

Suppose now that $p \notin \{2,7\}$. We see from \cite[\S1]{liebeck85} that $L^2U \cong A^2U$ has two composition factors, namely, $L(0)$ of dimension $1$ and $L(\lambda_2)$ of dimension $1539$, each lying in the set $\mathcal{L}$ defined in \S\ref{sec:liepowerschev}. Theorem \ref{thm:semisimplekgmod} then implies that $L^2U$ is multiplicity free, and so $L^2W$ is also multiplicity free by Proposition \ref{prop:extrestrictiso} and Lemma \ref{lem:finitemodscompfactors}. In particular, $L^2W$ is the direct sum of a $1$-dimensional submodule and a $1539$-dimensional submodule. By Theorem \ref{thm:finiteextsubmods}, $\hat G$ stabilises the same subspaces of $L^2V$ as $\mathrm{Sp}(56,q)$. Since $\mathrm{Sp}(56,q) \trianglelefteq \mathrm{CSp}(56,q)$, and since the irreducible $\hat G$-submodules of the multiplicity free module $L^2V$ are not equidimensional, $\mathrm{CSp}(56,q)$ stabilises the same subspaces of $L^2V$ as $\mathrm{Sp}(56,q)$.

If instead $p = 7$, then the $\hat G$-module $L^2V$ is uniserial, and it has three composition factors, of dimension $1$, $1$ and $1538$, respectively, by Theorem \ref{thm:finiteextsubmods}. In fact, these are exactly the dimensions of the $\mathrm{Sp}(56,q)$-composition factors of $L^2W$ \cite[\S1]{liebeck85}. Hence $\mathrm{Sp}(56,q)$ is uniserial, and in particular, $\mathrm{Sp}(56,q)$ stabilises the same subspaces of $L^2V$ as $\hat G$. Since $\mathrm{Sp}(56,q) \trianglelefteq \mathrm{CSp}(56,q)$, and since $L^2W$ has a unique composition series, it follows from Clifford's Theorem that $\mathrm{CSp}(56,q)$ stabilises the same subspaces of $L^2V$ as $\mathrm{Sp}(56,q)$.

We now apply the methods used to derive Theorem \ref{thm:l3vkgmod} in order to determine the composition factors of the $K[X]$-module $L^3U$. Magma calculations show that the Weyl orbit of $\lambda_1$ has size $56$. Since the weight multiset $\Lambda(U)$ has size $\dim(U)=56$, and since $\Lambda(U)$ contains the Weyl orbit of $\lambda_1$, it follows that $\Lambda(U)$ is precisely this Weyl orbit. We also calculate that the highest weight of $L^3U$ is $\lambda_1 + \lambda_2$. The computations described in \cite[\S3]{lubeck} can be used to show that if $p \notin \{3,19\}$, then the irreducible module $L(\lambda_1 + \lambda_2)$ has dimension $58464$, and the weight multiset for this module consists of one, two and $54$ copies of the Weyl orbits of the weights $\lambda_1 + \lambda_2$, $\lambda_3$ and $\lambda_1$, respectively \cite{lubeckpriv}. When these weights are excluded from the weight multiset for $L^3U$, $56$ weights remain, the highest of which is $\lambda_1$. It follows that if $p \notin \{3,19\}$, then $L^3U$ has two composition factors, of dimension $56$ and $58464$, respectively, each lying in $\mathcal{L}$. In this case, $L^3U$ is multiplicity free by Theorem \ref{thm:semisimplekgmod}, as is $L^3W$ by Proposition \ref{prop:extrestrictiso} and Lemma \ref{lem:finitemodscompfactors}. Specifically, $L^3W$ is the direct sum of an irreducible submodule of dimension $56$ and an irreducible submodule of dimension $58464$. As above, $\mathrm{CSp}(56,q)$ stabilises the same subspaces of $L^3V$ as $\mathrm{Sp}(56,q)$.

Finally, if $p = 19$, then the irreducible module $L(\lambda_1 + \lambda_2)$ has dimension $58408$, and the weight multiset for this module consists of one, two and $53$ copies of the Weyl orbits of the weights $\lambda_1 + \lambda_2$, $\lambda_3$ and $\lambda_1$, respectively \cite{lubeckpriv}. When these weights are excluded from the weight multiset for $L^3U$, $112$ weights remain, the highest of which is $\lambda_1$, with multiplicity $2$. Hence $L^3W$ has two $56$-dimensional composition factors and one $58408$-dimensional composition factor. The $\hat G$-submodule structure of $L^3V$ given in Figure \ref{fig:e7l3vsub} for the case $q = 19$ implies that $L^3W$ is uniserial, with the dimensions of its submodules as required. An easy adaptation of the uniserial case proof of Theorem \ref{thm:finiteextsubmods}\ref{finitestructa2v} shows that the submodule structure of $L^3W$ is as required even if $q > p = 19$. Again, $\mathrm{CSp}(56,q)$ stabilises the same subspaces of $L^3V$ as $\mathrm{Sp}(56,q)$.
\end{proof}

We are therefore not able to distinguish between $\mathrm{CSp}(56,q)$ and the simply connected version $\hat G$ of $E_7(q)$ by considering how these groups act on $L^2V$. Thus applying Theorem \ref{thm:univap} to the nonzero proper $\hat G$-submodules of $L^2V$ yields the same $p$-groups as applying the theorem to the nonzero proper $\mathrm{CSp}(56,q)$-submodules of $L^2V$.

\begin{prop}
\label{prop:normhatgsubs}
Let $r \in \{2,3\}$, and suppose that $p > r$. Then $\zgl \hat G$ stabilises each $\hat G$-submodule of $L^rV$. Furthermore, if $N_{\mathrm{GL}(d,q)}(\hat G)$ does not stabilise every $\hat G$-submodule of $L^rV$, then $r = 3$, and:
\begin{enumerate}[label={(\roman*)}]
\item $\tilde G = E_6(q)$, $p = 5$, and $q \equiv 1 \pmod 3$; or \label{e63norm}
\item $\tilde G = E_7(q)$, $p \in \{7,11,19\}$, and $q > p$. \label{e73norm}
\end{enumerate}
\end{prop}

\begin{proof}
Since $\zgl$ acts on $V$ by scalar multiplication, the definition of the action of $\mathrm{GL}(d,q)$ on $L^rV$ implies that $\zgl$ also acts on $L^rV$ by scalar multiplication. Hence $\zgl$ stabilises every subspace of $L^rV$, and so $\zgl \hat G$ stabilises each $\hat G$-submodule of $L^rV$. Let $N:=N_{\mathrm{GL}(d,q)}(\hat G)$. Lemma \ref{lem:excepnormaliser} implies that if $\tilde G \in \{G_2(q),F_4(q),E_8(q)\}$, or if $\tilde G = E_6(q)$ with $q \not\equiv 1 \pmod 3$, then $N = \zgl \hat G$. Otherwise, if $L^rV$ is multiplicity free, then we see from Theorem \ref{thm:finiteextsubmods} that no two irreducible submodules of $L^rV$ are equidimensional. We also have $\zgl \hat G \trianglelefteq N$, and thus $N$ stabilises the same subspaces of $L^rV$ as $\zgl \hat G$.

We now consider the remaining cases where $L^rV$ is not multiplicity free, and either \begin{NoHyper}\ref{e63norm} and \ref{e73norm}\end{NoHyper} do not apply, or $r = 2$. Theorem \ref{thm:finiteextsubmods} implies that $\tilde G = E_7(q)$, with $p = 7$ if $r = 2$, and with $q = p \in \{7,11,19\}$ if $r = 3$. Recall from Lemma \ref{lem:excepnormaliser} that ${N < \mathrm{CSp}(56,q)}$. Hence Lemma \ref{lem:sp56lrv} implies that $N$ stabilises each $\hat G$-submodule of $L^2V$, and that $N$ stabilises the $\hat G$-submodules of $L^3V$ of dimension $56$ and $58464$. Let $\mathcal{N}$ denote the set of $\hat G$-submodules of $L^3V$ stabilised by $N$. We can show that $\mathcal{N}$ contains all $\hat G$-submodules of $L^3V$ using the following:
\begin{enumerate}[label={(\alph*)}]
\item if $X,U \in \mathcal{N}$, and if $X = U \oplus W$ for some $\hat G$-submodule $W$ of $L^3V$, then $W \in \mathcal{N}$;
\item if $U,W \in \mathcal{N}$ and $W \subseteq U$, then $\mathcal{N}$ contains each $\hat G$-submodule in at least one $\hat G$-composition series for $U$ containing $W$; and
\item if $U \in \mathcal{N}$ and $U$ is multiplicity free, then $\mathcal{N}$ contains each $\hat G$-submodule of $U$.
\end{enumerate}
These properties hold because the normal subgroup $\zgl \hat G$ of $N$ has index coprime to $p$ by Lemma \ref{lem:excepnormaliser}; because $\zgl \hat G$ stabilises each $\hat G$-submodule of $L^3V$; and because no two submodules of $L^3V$ are equidimensional.

We detail our argument in the case $p = 7$; the other cases are similar. Recall the submodule structure of $L^3V$ from Figure \ref{fig:e7l3vsub}. Let $U_k$ denote the unique $\hat G$-submodule of $L^3V$ of dimension $k$, when such a submodule exists. As $U_{56} \in \mathcal{N}$, (b) implies that $\mathcal{N}$ contains each module in some composition series for $L^3V$ containing $U_{56}$. Hence either $U_{52040} \in \mathcal{N}$, $U_{7448} \in \mathcal{N}$, or $\{U_{6536}, U_{57608}\} \subset \mathcal{N}$. Since $L^3V = U_{7448} \oplus U_{51072}$ and $U_{57608} = U_{6536} \oplus U_{51072}$, and since $U_{52040} = U_{56} \oplus U_{912} \oplus U_{51072}$ is multiplicity free, it follows from (a) and (c) that $U_{51072} \in \mathcal{N}$. Thus $U_{51128} = U_{56} \oplus U_{51072} \in \mathcal{N}$. As $L^3V = U_{51128} \oplus U_{7392}$, we also have $U_{7392} \in \mathcal{N}$ by (a). The submodule $U_{7392}$ is uniserial, and hence (b) implies that $\mathcal{N}$ contains each of its submodules. We have shown that $N$ stabilises each submodule of each direct summand in the decomposition $L^3V = U_{56} \oplus U_{7392} \oplus U_{51072}$. Therefore, $N$ stabilises all direct sums of submodules of these summands, which accounts for all $\hat G$-submodules of $L^3V$.
\end{proof}

In particular, $N_{\mathrm{GL}(d,q)}(\hat G)$ stabilises every $\hat G$-submodule of $L^rV$ whenever $q = p$. In fact, if Conjecture \ref{conj:allq} holds, i.e., if the $\hat G$-submodule structure of $L^3V$ depends only on $p$ and not $q$, then the arguments in the above proof for the case $\tilde G = E_7(q)$ and $p \in \{7,11,19\}$ hold even if $q \ne p$. Furthermore, if this conjecture holds, then a similar argument using Figure \ref{fig:e6l3vsub} shows that $N_{\mathrm{GL}(d,p)}(\hat G)$ stabilises all submodules of $L^3V$ whenever $\tilde G = E_6(q)$ and $p = 5$.

\begin{thm}
\label{thm:lrvhatgstab}
Suppose that $p > r$, where $r = 3$ if $\tilde G \in \{E_6(q),E_7(q)\}$, and $r = 2$ otherwise. Suppose also that $q = p$ if $\tilde G = E_6(q)$ with $p = 5$, or if $\tilde G = E_7(q)$ with $p \in \{7,11,19\}$. Additionally, let $X$ be a nonzero proper submodule of $L^rV$. Then either $\mathrm{GL}(d,q)_X = N_{\mathrm{GL}(d,q)}(\hat G)$, or $\tilde G = E_7(q)$ and $\mathrm{CSp}(56,q)$ stabilises $X$. 
\end{thm}

\begin{proof}
Let $N:=N_{\mathrm{GL}(d,q)}(\hat G)$, $M:=N_{\mathrm{SL}(d,q)}(\hat G)$, and $Z_{\mathrm{SL}}:=Z(\mathrm{SL}(d,q))$. Lemma \ref{lem:excepnormaliser} and Dedekind's Identity imply that the normal subgroup $Z_{\mathrm{SL}} \hat G$ of $M$ has index at most $3$. Thus $M^\infty = (Z_{\mathrm{SL}} \hat G)^\infty = \hat G^\infty$, which is equal to the perfect group $\hat G$. Hence $\hat G$ is a characteristic subgroup of $M$, and so $N_{\mathrm{GL}(d,q)}(M) \le N$. Moreover, ${\mathrm{SL}(d,q) \trianglelefteq \mathrm{GL}(d,q)}$, and so $M = {N \cap \mathrm{SL}(d,q) \trianglelefteq N}$. It follows that $N_{\mathrm{GL}(d,q)}(M) = N$. We also have $\mathrm{GL}(d,q)_X \le N_{\mathrm{GL}(d,q)}(\mathrm{SL}(d,q)_X)$. Since $N$ stabilises $X$ by Proposition \ref{prop:normhatgsubs}, we have $M \le \mathrm{SL}(d,q)_X$, and it suffices to show that $\mathrm{SL}(d,q)_X = M$.

Suppose that $\tilde G \in \{G_2(q),F_4(q),E_8(q)\}$. Recall that $\hat G \le S(\beta)$ for some non-degenerate orthogonal form $\beta$ on $V$. Let $\varepsilon \in \{\circ,+,-\}$ be the type of this form. Then by Proposition \ref{prop:hatgspomegaover}, $M$ is the largest subgroup of $\mathrm{SL}(d,q)$ that contains $\hat G$ but does not contain $\Omega^\varepsilon(d,q)$. The group $\Omega^\varepsilon(d,q)$ acts irreducibly on $L^2V$ \cite[\S1]{liebeck85}, and hence $\mathrm{SL}(d,q)_X = M$.

Next, suppose that $\tilde G = E_7(q)$, so that $d = 56$. Then Proposition \ref{prop:hatgspomegaover} implies that $M$ is the largest subgroup of $\mathrm{SL}(56,q)$ that contains $\hat G$ but does not contain $\mathrm{Sp}(56,q)$. Furthermore, if $\mathrm{Sp}(56,q)$ stabilises $X$, then so does $\mathrm{CSp}(56,q)$ by Lemma \ref{lem:sp56lrv}. Thus if $\mathrm{CSp}(56,q)$ does not stabilise $X$, then $\mathrm{SL}(56,q)_X = M$.

Finally, suppose that $\tilde G = E_6(q)$. Then $M$ is a maximal subgroup of $\mathrm{SL}(d,q)$ by Theorem \ref{thm:excepmax}. Bamberg et al.~\cite[Lemma 3.1]{bamberg} show that $\mathrm{GL}(d,q)$ acts irreducibly on $L^3V$. As no two $\hat G$-submodules of $L^3V$ are equidimensional by Theorem \ref{thm:finiteextsubmods}, it follows that no two $\mathrm{SL}(d,q)$-submodules of $L^3V$ are equidimensional. Clifford's Theorem then implies that $\mathrm{SL}(d,q)$ acts irreducibly on $L^3V$, and hence $\mathrm{SL}(d,q)_X = M$.
\end{proof}

If Conjecture \ref{conj:allq} holds, then we do not need to assume that $q = p$ in any case in the above lemma.

Recall that if $\tilde G \in \{(G_2(q),E_8(q)\}$, then $\hat G \cong \tilde G$. The following lemma will allow us to apply Theorem \ref{thm:expclass2stab} to these groups (with $q = p$) to yield $p$-groups that Theorem \ref{thm:univap} does not yield.

\begin{lem}
\label{lem:directsuml2vstab}
Suppose that $\tilde G \in \{(G_2(q),E_8(q)\}$. Then $\tilde G \cong \hat G$ is the stabiliser in $\mathrm{GL}(d,q)$ of a proper subspace $M$ of $V \oplus L^2V$ of maximal dimension. Moreover, $M \cong L^2V$, and $M$ does not contain $V$ or $L^2V$.
\end{lem}

\begin{proof}
Theorem \ref{thm:finiteextsubmods} shows that there exists an $\mathbb{F}_q[\tilde G]$-epimorphism $\theta:L^2V \to V$ whose kernel is a nonzero maximal submodule of $L^2V$. It follows that $V \oplus L^2V$ contains the submodule $M:= {\{((u)\theta,u) \mid u \in L^2V\}}$. This submodule does not contain the direct summand $V$, and is isomorphic, but not equal, to $L^2V$. Since the composition factors of $V \oplus L^2V$ are $V$ and the composition factors of $L^2V$, it follows from Theorem \ref{thm:finiteextsubmods} that $\dim(L^2V)$ is the largest possible dimension of a proper submodule of $V \oplus L^2V$.

Let $H:=\mathrm{GL}(d,q)_M$, so that $\theta$ extends to an $\mathbb{F}_q[H]$-epimorphism. Then $H$ stabilises $\ker \theta$, and Lemma \ref{lem:excepnormaliser} and Theorem \ref{thm:lrvhatgstab} imply that $\tilde G \le H \le \zgl \tilde G$. Let $z \in \zgl$, so that the action of $z$ on $V$ is equivalent to multiplication by some scalar $\mu \in \mathbb{F}_q \setminus \{0\}$. Then for each $u \in L^2V$, we have $((u)\theta,u)^z = (\mu(u)\theta,\mu^2 u)$, which is equal to $((\mu u)\theta,\mu^2 u)$ by the linearity of $\theta$. Thus $z \in \mathrm{GL}(d,q)_M$ if and only if $\mu = 1$, i.e., $z = 1$. Hence $H = \tilde G$.
\end{proof}

\section{Inducing exceptional Chevalley groups on $P/\Phi(P)$}
\label{sec:mainproof}

We can now induce on the Frattini quotient of a $p$-group the simply connected version of a given exceptional Chevalley group (defined over an appropriate field of prime order), or the normaliser of this simply connected group in $\mathrm{GL}(d,p)$. We again retain the notation outlined at the start of \S\ref{sec:overgroupschev}. In particular, $p$ is odd. Recall that $\hat G = \nonsplit{3}{\tilde G}$ if $\tilde G = E_6(q)$ with $q \equiv {1 \pmod 3}$; that $\hat G = \nonsplit{2}{\tilde G}$ if $\tilde G = E_7(q)$; and that $\hat G \cong \tilde G$ otherwise.

As in \S\ref{sec:univpgps}, we write $P_U$ (respectively, $Q_W$) to denote the quotient of the universal $p$-group $\Gamma(d,p,2)$ (respectively, $\Gamma(d,p,3)$) by a proper subgroup $U$ of $L^2V$ (respectively, a proper subgroup $W$ of $L^3V$). Recall also that if $P$ is a $p$-group of rank $d$, then $A(P) \le \mathrm{GL}(d,p)$ is the group induced by $\mathrm{Aut}(P)$ on $P/\Phi(P)$, and that if $A(P) < \mathrm{GL}(d,p)$, then the exponent-$p$ class of $P$ is at least $2$. We begin by highlighting some proper subgroups of $\mathrm{GL}(d,p)$ that cannot be induced on the Frattini quotient of a $p$-group of low exponent-$p$ class (in some cases, with low exponent or low nilpotency class also assumed).

\begin{thm}
\label{thm:nopgpthm}
Suppose that $q = p$ (which is odd), with $p > 3$ if $\tilde G \in \{E_6(p),E_7(p)\}$.
\begin{enumerate}[label={(\roman*)}]
\item Assume that $(\tilde G,H) \in \{(F_4(p),\zgl F_4(p)),(E_6(p),\mathrm{GL}(27,p)),(E_7(p),\mathrm{CSp}(56,p))\},$ and let $K$ be a proper subgroup of $H$ that contains $\hat G$. Then there is no $p$-group $P$ of exponent-$p$ class $2$ such that $A(P) = K$. \label{nopgpthm1}
\item Assume that $\tilde G \in \{G_2(p),E_8(p)\}$, and let $K$ be a proper subgroup of $\zgl \hat G$ that contains $\hat G$. If $P$ is a $p$-group of exponent-$p$ class $2$ satisfying $A(P) = K$, then $P$ has nilpotency class $2$ and exponent $p^2$. \label{nopgpthm2}
\item Assume that $\tilde G \in \{E_6(p),E_7(p)\}$, and let $K$ be a proper subgroup of $N_{\mathrm{GL}(d,p)}(\hat G)$ that contains $\hat G$. Then there is no $p$-group $P$ such that $A(P) = K$ and $P \cong Q_W$ for some proper subspace $W$ of $L^3V$. \label{nopgpthm3}
\end{enumerate}
\end{thm}

\begin{proof}
In case \ref{nopgpthm1}, Theorem \ref{thm:finiteextsubmods} shows that no $\hat G$-composition factor of $L^2V$ is isomorphic to the irreducible $\hat G$-module $V$. Hence the submodules of $V \oplus L^2V$ in this case are the direct sums of the submodules of the direct summands. Note also that a $p$-group of exponent-$p$ class $2$ has nilpotency class at most $2$ and exponent at most $p^2$. As $V$ is irreducible, it suffices by Theorems \ref{thm:expclass2stab} and \ref{thm:univap} to show that each $\hat G$-submodule of $L^2V$ is stabilised by $H$ or $\zgl \hat G$ in case \ref{nopgpthm1} or \ref{nopgpthm2}, respectively, and that $N_{\mathrm{GL}(d,p)}(\hat G)$ stabilises each $\hat G$-submodule of $L^3V$ in case \ref{nopgpthm3}. Indeed, this is satisfied in the $\tilde G = E_6(p)$ case of \ref{nopgpthm1}, as here $L^2V$ is irreducible by Theorem \ref{thm:finiteextsubmods}. The result follows from Lemma \ref{lem:sp56lrv} in the $\tilde G = E_7(p)$ case of \ref{nopgpthm1}, and from Proposition \ref{prop:normhatgsubs} in all remaining cases.
%
%
%
\end{proof}

Note that the condition in Theorem \ref{thm:nopgpthm}\ref{nopgpthm1} that $K$ is a proper subgroup of $H$ is necessary. This fact is clear when $H = \mathrm{GL}(27,p)$; it follows from \cite[Lemma 5.8, Table 6.1]{bamberg} when $H = \mathrm{CSp}(56,q)$; and we will see in Theorem \ref{thm:mainthm} that it holds when $H = \zgl F_4(p)$. Note also that the groups $Q_W$ mentioned in Theorem \ref{thm:nopgpthm}\ref{nopgpthm3} have exponent $p$ and nilpotency class $3$ (and hence exponent-$p$ class $3$) by Lemma \ref{lem:gamma23quotients}. However, there may exist a $p$-group $P$ of exponent $p$ and nilpotency class $3$, with $P$ not isomorphic to $Q_W$ for any proper subspace $W$ of $L^3V$, such that $A(P)$ is a group $K$ defined in Theorem \ref{thm:nopgpthm}\ref{nopgpthm3}.

We are now able to state and prove the main theorem of this paper. Recall the definitions of optimal and quasi-optimal $p$-groups, from Definitions \ref{def:optimalpgp} and \ref{def:quasioptimalpgp}, respectively.

\begin{thm}\leavevmode
\label{thm:mainthm}
Suppose that $q = p$ (which is odd).
\begin{enumerate}[label={(\roman*)}]
\item \label{mainthm1} If $\tilde G \in \{G_2(p),F_4(p),E_8(p)\}$, then each $p$-group that is optimal with respect to $N_{\mathrm{GL}(d,p)}(\hat G)$ has exponent-$p$ class $2$, nilpotency class $2$ and exponent $p$.
\item \label{mainthm2} If $\tilde G \in \{G_2(p),E_8(p)\}$, then each $p$-group that is optimal with respect to $\hat G$ has exponent-$p$ class $2$, nilpotency class $2$ and exponent $p^2$.
\item \label{mainthm3} If $\tilde G \in \{E_6(p),E_7(p)\}$ and $p > 3$, then each $p$-group that is optimal (or quasi-optimal) with respect to $N_{\mathrm{GL}(d,p)}(\hat G)$ has exponent-$p$ class $3$, nilpotency class $3$ and exponent $p$.
\end{enumerate}
Table \ref{table:mainthmtable} specifies the properties of each optimal $p$-group in case (i) or (ii), and each quasi-optimal $p$-group in case (iii). Finally, each $p$-group in case (i) or (ii) has a unique nontrivial proper characteristic subgroup.
\end{thm}

\begin{table}[ht]
\centering
\renewcommand{\arraystretch}{1.2}
\caption{The properties of each optimal or quasi-optimal $p$-group $P$ from Theorem \ref{thm:mainthm}. Here, $t:=(3,p-1)$, and $c$ denotes both the exponent-$p$ class and nilpotency class of $P$.}
\label{table:mainthmtable}
\begin{tabular}{ |c|c|c|c|c|c| }
\hline
$\tilde G$ & $d$ & $c$ & {\parbox{1.8cm}{\centering \vspace{.15cm}Exponent\\of $P$\vspace{.15cm}}} & $|P|$ & $A(P)$\\
\hline
\hline
$G_2(p)$ & $7$ & $2$ & $p$ & $p^{14}$ & $\zgl G_2(p)$\\
\hline
$G_2(p)$ & $7$ & $2$ & $p^2$ & $p^{14}$ & $G_2(p)$\\
\hline
$F_4(3)$ & $25$ & $2$ & $3$ & $3^{77}$ & $\zgl F_4(3)$\\
\hline
$F_4(p), p > 3$ & $26$ & $2$ & $p$ & $p^{78}$ & $\zgl F_4(p)$\\
\hline
$E_6(p), p > 3$ & $27$ & $3$ & $p$ & $p^{456}$ & $(\zgl (\nonsplit{t}{E_6(p)})).t$\\
\hline
$E_7(p), p > 3$ & $56$ & $3$ & $p$ & $p^{2508}$ & $(\zgl (\nonsplit{2}{E_7(p)})).2$\\
\hline
$E_8(p)$ & $248$ & $2$ & $p$ & $p^{496}$ & $\zgl E_8(p)$\\
\hline
$E_8(p)$ & $248$ & $2$ & $p^2$ & $p^{496}$ & $E_8(p)$\\
\hline
\end{tabular}
\end{table}

\begin{proof}
Theorem \ref{thm:finiteextsubmods} gives the dimensions of the proper $\hat G$-submodules of $L^2V$ (respectively, $L^3V$) in case \ref{mainthm1} (respectively, in case \ref{mainthm3}). By Lemma \ref{lem:sp56lrv} and Theorem \ref{thm:lrvhatgstab}, $N:=N_{\mathrm{GL}(d,p)}(\hat G)$ is the stabiliser in $\mathrm{GL}(d,p)$ of the largest of these submodules, except when $\tilde G = E_7(p)$, in which case $N$ is the stabiliser in $\mathrm{GL}(d,p)$ of the second largest proper submodule. Similarly, Lemma \ref{lem:directsuml2vstab} implies that in case \ref{mainthm2}, $\hat G$ is the stabiliser in $\mathrm{GL}(d,p)$ of a proper $\hat G$-submodule of $V \oplus L^2V$ of maximal dimension, with this submodule isomorphic to $L^2V$ but not containing $V$ or $L^2V$. Let $X$ be the specified proper submodule whose stabiliser in $\mathrm{GL}(d,p)$ is $N$ or $\hat G$. It follows from Theorem \ref{thm:univap} that $N$ is equal to $A(P_X)$ or $A(Q_X)$ in case \ref{mainthm1} or \ref{mainthm3}, respectively. In case \ref{mainthm2}, Theorem \ref{thm:expclass2stab} shows that $\hat G$ is equal to $A(E^*/X)$, where $E^*$ is the $p$-covering group of the elementary abelian $p$-group of rank $d$. Since $V$ is an irreducible $\hat G$-module, and since $(L^2V)/X \cong (V \oplus L^2V)/(V \oplus X)$ in case \ref{mainthm1}, Theorems \ref{thm:expclass2stab} and \ref{thm:ucsgroups} imply that each $p$-group $P_X$ or $E^*/X$ has a unique nontrivial proper characteristic subgroup.

Recall that if a $p$-group has exponent $p$, then its exponent-$p$ class is equal to its nilpotency class, and that a $p$-group of exponent-$p$ class $2$ has nilpotency class at most $2$ and exponent at most $p^2$. Hence Lemma \ref{lem:gamma23quotients} and Theorem \ref{thm:expclass2stab} show that, in each case, the $p$-group $P_X$, $E^*/X$ or $Q_X$ has the exponent-$p$ class, nilpotency class and exponent of the $p$-group given in Table \ref{table:mainthmtable}. Now, the order of $\Gamma(d,p,r)$ is given in Theorem \ref{thm:gamma23}, while we have from Proposition \ref{prop:vstarext} that $E^*$ is an extension of $V \times L^2V$ by $V$. Since $L^2V \cong A^2V$, we have $|P_X| = p^{d(d+1)/2-\dim(X)}$, $|E^*/X| = p^{(d^2+d)/2-\dim(X)}$ and $|Q_X| = p^{d(d+1)(2d+1)/6-\dim(X)}$, with $\dim(X) = (d^2-d)/2$ in case \ref{mainthm2}. In each case, the order of this $p$-group is given in Table \ref{table:mainthmtable}. The structure of $A(P) = N$ in cases \ref{mainthm1} and \ref{mainthm3} is given by Lemma \ref{lem:excepnormaliser}. Note that this lemma also implies that $N$ is a proper subgroup of $\mathrm{GL}(d,p)$ in each case, and also a proper subgroup of $\mathrm{CSp}(56,p)$ if $\tilde G = E_7(p)$.

We now show that the specified $p$-groups are optimal or quasi-optimal as required. In each case, the optimal $p$-group has exponent-$p$ class at least $2$. Hence the $p$-group $P_X$ in case \ref{mainthm1} is indeed optimal with respect to $N$. It also follows from Theorem \ref{thm:nopgpthm}\ref{nopgpthm2} that $E^*/X$ is optimal with respect to $\hat G$ in case \ref{mainthm2}, and from Theorem \ref{thm:nopgpthm}\ref{nopgpthm1} that $Q_X$ is quasi-optimal with respect to $N$ in case \ref{mainthm3}. By definition, each $p$-group that is optimal with respect to $N$ in case \ref{mainthm3} has the same exponent-$p$ class, exponent and nilpotency class as $Q_X$.
\end{proof}

When calculating the order of $A(P)$ in cases \ref{mainthm1} and \ref{mainthm3} of the above theorem, it is useful to know the order of $\zgl \cap \hat G$. Proposition \ref{prop:excepcentre} implies that this order is $3$ if $\tilde G = E_6(p)$ with $p \equiv {1 \pmod 3}$, and $2$ if $\tilde G = E_7(p)$. Otherwise, $\hat G \cong \tilde G$ is non-abelian and simple, and hence $\zgl \cap \hat G  = 1$.

The $p$-groups in cases \ref{mainthm1}, \ref{mainthm2} and \ref{mainthm3} of Theorem \ref{thm:mainthm} can be constructed as specific quotients of the universal $p$-group $\Gamma(d,p,2)$, the $p$-covering group $E^*$ of the elementary abelian $p$-group of rank $d$, and the universal $p$-group $\Gamma(d,p,3)$, respectively, as detailed in the proof of the theorem. The $p$-groups in case \ref{mainthm1} can also be constructed as corresponding quotients of $E^*$, by Theorem \ref{thm:expclass2stab}. Note also that the $\tilde G = G_2(p)$ case of \ref{mainthm1} was previously proved by Bamberg, Freedman and Morgan \cite{sdfg2}. However, Theorem \ref{thm:mainthm} is completely disjoint from the work of Bamberg et al.~\cite{bamberg} mentioned in \S\ref{sec:intro}, which covers maximal subgroups of $\mathrm{GL}(d,p)$ that do not lie in the Aschbacher classes $\mathcal{C}_6$ or $\mathcal{C}_9$. Indeed, Lemma \ref{lem:excepnormaliser} shows that $A(P)$ is a $\mathcal{C}_9$-subgroup of $\mathrm{GL}(d,p)$ if $\tilde G = E_6(p)$, and otherwise $A(P)$ is not maximal in $\mathrm{GL}(d,p)$.

Now, if $Q$ is a $p$-group isomorphic to $P$, then $A(P)$ and $A(Q)$ are the images of representations afforded by isomorphic modules, with $V$ the restriction of the former module to $\hat G$. However, our method of constructing an optimal or quasi-optimal $p$-group in each case of Theorem \ref{thm:mainthm} does not depend on the choice of the minimal $\mathbb{F}_p[\hat G]$-module $V$. Therefore, if there are two isomorphism classes of minimal $\mathbb{F}_p[\hat G]$-modules, then there are at least two isomorphism classes of optimal or quasi-optimal $p$-groups. By Lemma \ref{lem:lieconjimages}, this is the case when $\tilde G$ is equal to $G_2(3)$ or to $E_6(p)$ for some $p > 3$. If $M$ is a minimal $\mathbb{F}_p[\hat G]$-module isomorphic to $V$, then $M = V^x$ for some $x \in \mathrm{GL}(V)$, and it is easy to see that $L^rM = (L^rV)^x$ for each $r \in \{2,3\}$. In particular, since no two submodules of $L^rV$ are equidimensional by Theorem \ref{thm:finiteextsubmods}, $x$ maps the submodule $X$ of $L^rV$ from the proof of Theorem \ref{thm:mainthm} to the unique submodule of $L^rM$ of dimension $\dim(X)$, which we will denote by $Y$. It follows from the definition of the action of $\mathrm{GL}(d,p)$ on $\Gamma(d,p,r)$ and the definitions of $P_X$ and $Q_X$ that $x$ induces an isomorphism from $P_X$ to $P_Y$ or from $Q_X$ to $Q_Y$, as appropriate. Therefore, in cases \ref{mainthm1} and \ref{mainthm3} of Theorem \ref{thm:mainthm}, there are exactly two isomorphism classes of optimal or quasi-optimal $p$-groups, respectively, if $\tilde G$ is equal to $G_2(3)$ or to $E_6(p)$, and exactly one isomorphism class otherwise.

We conclude with a note on a variation of Theorem \ref{thm:mainthm}\ref{mainthm2}.

\begin{rem}
\label{rem:liealg}
Let $\tilde H := {}^t Y_\ell(q)$ be an exceptional group of Lie type that is not a Suzuki or Ree group, with $\hat H$ the simply connected version of $\tilde H$, and $H$ the associated (simple, simply connected) linear algebraic group. In addition, let $U$ be the irreducible $\overline{\mathbb{F}_p}[H]$-module whose highest weight is the highest weight of the Lie algebra of $H$. Then:
\begin{enumerate}[label={(\roman*)}]
\item \label{liealg1} $A^2U$ is multiplicity free if $p>5$, or if $p = 5$ and $\tilde H \ne E_8(q)$;
\item \label{liealg1a} $U$ is isomorphic to a composition factor of $A^2U$;
\item \label{liealg2} the absolutely irreducible $\hat H$-modules corresponding (per Lemma \ref{lem:untwistedfinitemods}) to $U$ and the composition factors of $A^2U$ can all be written over $\mathbb{F}_q$; and
\item \label{liealg3} these $\hat H$-modules are actually faithful modules for $\tilde H$.
\end{enumerate}
Hence in the case $q = p$, we can adapt the proof of Theorem \ref{thm:mainthm}\ref{mainthm2} to construct a $p$-group $Q$ of rank $\dim(U)$, exponent $p$-class $2$, nilpotency class $2$ and exponent $p^2$ such that $A(Q)$ contains $\tilde H$, but does not contain any nontrivial scalar matrix of $\mathrm{GL}(\dim(U),p)$.
\end{rem}

Note that $U$ is equal to the Lie algebra of $H$, unless $Y_\ell \in \{G_2,E_6\}$ and $p=3$. Moreover, if $Y_\ell = E_8$, then $U$ is a minimal $\overline{\mathbb{F}_p}[H]$-module, and so $Q$ is a $p$-group mentioned in Theorem \ref{thm:mainthm}\ref{mainthm2}. In general, \ref{liealg1} and \ref{liealg1a} follow from the methods used to derive Theorem \ref{thm:a2vkgmod}. If $t \ne 1$, then the graph automorphism of $H$ of order $t$ fixes $U$ and each composition factor of $A^2U$, and the dimension of each of these irreducible modules is not an integral power of any integer other than itself. Hence \ref{liealg2} follows from Proposition \ref{prop:minfieldmod}. Additionally, $Z(H)$ acts trivially on $U$ in each case (see \cite[Appendix A.2]{lubeck}). Since $Z(\hat H) \le Z(H)$ \cite[Corollary 24.13]{malle}, and since $\tilde H$ is simple and isomorphic to $\hat H/Z(\hat H)$, we obtain \ref{liealg3}. Finally, when $q = p$, the argument used in the proof of Lemma \ref{lem:directsuml2vstab} implies that $A(Q)$ contains no nontrivial scalars. However, whether $A(Q)$ is equal to or properly contains $\tilde H$ is an open question, which may benefit from the approach of \S\ref{sec:overgroupschev}--\ref{sec:liepowersubmodstab}. The relationship here between $Q$ and the Lie algebra of $H$ is an example of the more general duality between certain $p$-groups and algebras discussed in \cite{glasbyduality}.

\subsection*{Acknowledgements}
This work was supported by a Hackett Postgraduate Research Scholarship and an Australian Government Research Training Program at the University of Western Australia. The author is also grateful to Martin Liebeck, Donna Testerman and Gunter Malle for helpful discussions about highest weight theory; to John Bamberg and Luke Morgan for general helpful discussions and detailed feedback; and to the referee for useful recommendations.

\bibliographystyle{plain}
\bibliography{Paperrefs}

\providecommand\noopsort[1]{}\def\cprime{$'$} \def\cprime{$'$} \def\cprime{$'$}
  \def\cprime{$'$}
\begin{thebibliography}{10}

\bibitem{aschbacherthm}
M.~Aschbacher.
\newblock On the maximal subgroups of the finite classical groups.
\newblock {\em Invent. Math.}, 76(3):469--514, 1984.

\bibitem{bahturin}
Yuri Bahturin.
\newblock {\em Basic structures of modern algebra}, volume 265 of {\em
  Mathematics and its Applications}.
\newblock Kluwer Academic Publishers Group, Dordrecht, 1993.

\bibitem{sdfg2}
John Bamberg, Saul~D. Freedman, and Luke Morgan.
\newblock On $p$-groups with automorphism groups related to the {C}hevalley
  group ${G}_2(p)$.
\newblock {\em J. Aust. Math. Soc.}, 108(3):321--331, 2020.

\bibitem{bamberg}
John Bamberg, S.~P. Glasby, Luke Morgan, and Alice~C. Niemeyer.
\newblock Maximal linear groups induced on the {F}rattini quotient of a
  {$p$}-group.
\newblock {\em J. Pure Appl. Algebra}, 222(10):2931--2951, 2018.

\bibitem{magma}
Wieb Bosma, John Cannon, and Catherine Playoust.
\newblock The {M}agma algebra system. {I}. {T}he user language.
\newblock {\em J. Symbolic Comput.}, 24(3-4):235--265, 1997.

\bibitem{bourbaki}
Nicolas Bourbaki.
\newblock {\em Elements of mathematics. {A}lgebra, {P}art {I}: {C}hapters 1-3}.
\newblock Hermann, Paris; Addison-Wesley Publishing Co., Reading Mass., 1974.
\newblock Translated from the French.

\bibitem{BHRD}
John~N. Bray, Derek~F. Holt, and Colva~M. Roney-Dougal.
\newblock {\em The maximal subgroups of the low-dimensional finite classical
  groups}, volume 407 of {\em London Mathematical Society Lecture Note Series}.
\newblock Cambridge University Press, Cambridge, 2013.

\bibitem{bryant}
R.~M. Bryant and L.~G. Kov{\'a}cs.
\newblock Lie representations and groups of prime power order.
\newblock {\em J. London Math. Soc. (2)}, 17(3):415--421, 1978.

\bibitem{chevalley}
Claude Chevalley.
\newblock {\em Classification des groupes alg\'ebriques semi-simples}.
\newblock Springer-Verlag, Berlin, 2005.

\bibitem{ATLAS}
J.~H. Conway, R.~T. Curtis, S.~P. Norton, R.~A. Parker, and R.~A. Wilson.
\newblock {\em Atlas of finite groups}.
\newblock Oxford University Press, Eynsham, 1985.

\bibitem{eick}
Bettina Eick, C.~R. Leedham-Green, and E.~A. O'Brien.
\newblock Constructing automorphism groups of {$p$}-groups.
\newblock {\em Comm. Algebra}, 30(5):2271--2295, 2002.

\bibitem{glasbyucs}
S.~P. Glasby, P.~P. P\'alfy, and Csaba Schneider.
\newblock {$p$}-groups with a unique proper non-trivial characteristic
  subgroup.
\newblock {\em J. Algebra}, 348:85--109, 2011.

\bibitem{glasbyduality}
S.~P. Glasby, Frederico A.~M. Ribeiro, and Csaba Schneider.
\newblock Duality between $p$-groups with three characteristic subgroups and
  semisimple anti-commutative algebras.
\newblock {\em Proc. Roy. Soc. Edinburgh Sect. A}, pages 1--26, 2019.
\newblock FirstView article.

\bibitem{gorenstein}
Daniel Gorenstein.
\newblock {\em Finite groups}.
\newblock Chelsea Publishing Co., New York, second edition, 1980.

\bibitem{gorensteincfsg3}
Daniel Gorenstein, Richard Lyons, and Ronald Solomon.
\newblock {\em The classification of the finite simple groups. {N}umber 3.
  {P}art {I}. {C}hapter {A}}, volume~40 of {\em Mathematical Surveys and
  Monographs}.
\newblock American Mathematical Society, Providence, RI, 1998.

\bibitem{gow}
Rod Gow.
\newblock {Subspaces of $7 \times 7$ skew-symmetric matrices related to the
  group $G_2$. 2008. \href{https://arxiv.org/abs/0811.1298}{arXiv:0811.1298}}.

\bibitem{bhall}
Brian Hall.
\newblock {\em Lie groups, {L}ie algebras, and representations}, volume 222 of
  {\em Graduate Texts in Mathematics}.
\newblock Springer, Cham, second edition, 2015.

\bibitem{hiss}
Gerhard Hiss and Gunter Malle.
\newblock Low-dimensional representations of quasi-simple groups.
\newblock {\em LMS J. Comput. Math.}, 4:22--63, 2001.
\newblock Corrigenda: \emph{LMS J. Comput. Math.}, 5:95--126, 2002.

\bibitem{humphreysmod}
James~E. Humphreys.
\newblock {\em Modular representations of finite groups of {L}ie type}, volume
  326 of {\em London Mathematical Society Lecture Note Series}.
\newblock Cambridge University Press, Cambridge, 2006.

\bibitem{ichar}
I.~Martin Isaacs.
\newblock {\em Character theory of finite groups}.
\newblock AMS Chelsea Publishing, Providence, RI, 2006.

\bibitem{johnson}
Marianne Johnson.
\newblock L. {G}. {K}ov\'acs' work on {L}ie powers.
\newblock {\em J. Aust. Math. Soc.}, 102(1):9--19, 2017.

\bibitem{kleidmanphd}
P.~B. Kleidman.
\newblock {\em The subgroup structure of some finite simple groups}.
\newblock PhD thesis, University of Cambridge, October 1987.

\bibitem{kleidman}
Peter Kleidman and Martin Liebeck.
\newblock {\em The subgroup structure of the finite classical groups}, volume
  129 of {\em London Mathematical Society Lecture Note Series}.
\newblock Cambridge University Press, Cambridge, 1990.

\bibitem{liebeck85}
Martin~W. Liebeck.
\newblock On the orders of maximal subgroups of the finite classical groups.
\newblock {\em Proc. London Math. Soc. (3)}, 50(3):426--446, 1985.

\bibitem{lubeck}
Frank L{\"u}beck.
\newblock Small degree representations of finite {C}hevalley groups in defining
  characteristic.
\newblock {\em LMS J. Comput. Math.}, 4:135--169, 2001.

\bibitem{lubeckpriv}
Frank L{\"u}beck.
\newblock Private communication, February 10, 2018.

\bibitem{lubeckws}
Frank L{\"u}beck.~\emph{Tables of Weight Multiplicities}.
\newblock
  \\\url{http://www.math.rwth-aachen.de/~Frank.Luebeck/chev/WMSmall/index.html}.
\newblock Accessed: June 20, 2017.

\bibitem{malle}
Gunter Malle and Donna Testerman.
\newblock {\em Linear algebraic groups and finite groups of {L}ie type}, volume
  133 of {\em Cambridge Studies in Advanced Mathematics}.
\newblock Cambridge University Press, Cambridge, 2011.

\bibitem{mcninch}
George~J. McNinch.
\newblock Dimensional criteria for semisimplicity of representations.
\newblock {\em Proc. London Math. Soc. (3)}, 76(1):95--149, 1998.

\bibitem{obrien}
E.~A. O'Brien.
\newblock The {$p$}-group generation algorithm.
\newblock {\em J. Symbolic Comput.}, 9(5-6):677--698, 1990.

\bibitem{samelson}
Hans Samelson.
\newblock {\em Notes on {L}ie algebras}.
\newblock Universitext. Springer-Verlag, New York, second edition, 1990.

\bibitem{schroder}
Anna~Katharina Schr\"oder.
\newblock {\em The maximal subgroups of the classical groups in dimension 13,
  14 and 15}.
\newblock PhD thesis, University of St Andrews, July 2015.

\bibitem{steinbergend}
Robert Steinberg.
\newblock {\em Endomorphisms of linear algebraic groups}.
\newblock Memoirs of the American Mathematical Society, No. 80. American
  Mathematical Society, Providence, R.I., 1968.

\bibitem{steinbergyale}
Robert Steinberg.
\newblock {\em Lectures on {C}hevalley groups}.
\newblock Yale University, New Haven, Conn., 1968.

\bibitem{wilson}
Robert~A. Wilson.
\newblock {\em The finite simple groups}, volume 251 of {\em Graduate Texts in
  Mathematics}.
\newblock Springer-Verlag London, Ltd., London, 2009.

\end{thebibliography}

\begin{appendices}

\section{Composition factors of exterior squares of minimal $K[G]$-modules}
\label{sec:compkga2v}

Tables \ref{table:g2a2v1}--\ref{table:e8a2v} list the composition factors of the exterior square $A^2V$ of each irreducible $K[G]$-module $V$ from Theorem \ref{thm:a2vkgmod}. Here, $V$ is a minimal $K[G]$-module, unless $(G,V) = (G_2,L(\lambda_2))$ with $p \ne 3$, or $(G,V) = (F_4,L(\lambda_1))$ with $p > 2$. In each case, the highest weights of all composition factors of $A^2V$ are comparable with respect to the partial order $\le$ defined in \S\ref{sec:highestweights}, and we list these composition factors by descending highest weight.

\begin{table}[H]
\centering
\renewcommand{\arraystretch}{1.1}
\caption{The composition factors of $A^2V$, with $G = G_2$ and $V = L(\lambda_1)$.}
\label{table:g2a2v1}
\begin{tabular}{ |c|c|c|c| }
\hline
Condition on $p$ & Composition factor & Dimension & Multiplicity \\
\hline
\hline
\multirow{2}{*}{$p = 2$} & $L(\lambda_2)$ & $14$ & $1$\\
 & $L(0)$ & $1$ & $1$\\
\hline
\multirow{2}{*}{$p = 3$} & $L(\lambda_2)$ & $7$ & $1$\\
 & $L(\lambda_1)$ & $7$ & $2$\\
\hline
\multirow{2}{*}{$p > 3$} & $L(\lambda_2)$ & $14$ & $1$\\
 & $L(\lambda_1)$ & $7$ & $1$\\
\hline
\end{tabular}
\end{table}

\begin{table}[H]
\centering
\renewcommand{\arraystretch}{1.1}
\caption{The composition factors of $A^2V$, with $G = G_2$ and $V = L(\lambda_2)$.}
\label{table:g2a2v2}
\begin{tabular}{ |c|c|c|c| }
\hline
Condition on $p$ & Composition factor & Dimension & Multiplicity \\
\hline
\hline
\multirow{5}{*}{$p = 2$} & $L(3 \lambda_1) \cong L(\lambda_1) \otimes L(\lambda_1)^{\phi}$ & $36$ & $1$\\
 & $L(2\lambda_1) \cong L(\lambda_1)^\phi$ & $6$ & $2$\\
 & $L(\lambda_2)$ & $14$ & $2$\\
 & $L(\lambda_1)$ & $6$ & $2$\\
 & $L(0)$ & $1$ & $3$\\
\hline
\multirow{2}{*}{$p = 3$} & $L(3 \lambda_1) \cong L(\lambda_1)^\phi$ & $7$ & $1$\\
 & $L(\lambda_2)$ & $7$ & $2$\\
\hline
\multirow{2}{*}{$p > 3$} & $L(3 \lambda_1)$ & $77$ & $1$\\
 & $L(\lambda_2)$ & $14$ & $1$\\
\hline
\end{tabular}
\end{table}

\begin{table}[H]
\centering
\renewcommand{\arraystretch}{1.1}
\caption{The composition factors of $A^2V$, with $G = F_4$ and $V = L(\lambda_1)$.}
\label{table:f4a2v1}
\begin{tabular}{ |c|c|c|c| }
\hline
Condition on $p$ & Composition factor & Dimension & Multiplicity \\
\hline
\hline
\multirow{4}{*}{$p = 2$} & $L(\lambda_2)$ & $246$ & $1$\\
 & $L(2\lambda_4) \cong L(\lambda_4)^\phi$ & $26$ & $1$\\
 & $L(\lambda_1)$ & $26$ & $2$\\
 & $L(0)$ & $1$ & $1$\\
\hline
\multirow{2}{*}{$p = 3$} & $L(\lambda_2)$ & $1222$ & $1$\\
 & $L(\lambda_1)$ & $52$ & $2$\\
\hline
\multirow{2}{*}{$p > 3$} & $L(\lambda_2)$ & $1274$ & $1$\\
 & $L(\lambda_1)$ & $52$ & $1$\\
\hline
\end{tabular}
\end{table}

\begin{table}[H]
\centering
\renewcommand{\arraystretch}{1.1}
\caption{The composition factors of $A^2V$, with $G = F_4$ and $V = L(\lambda_4)$.}
\label{table:f4a2v2}
\begin{tabular}{ |c|c|c|c| }
\hline
Condition on $p$ & Composition factor & Dimension & Multiplicity \\
\hline
\hline
\multirow{4}{*}{$p = 2$} & $L(\lambda_3)$ & $246$ & $1$\\
 & $L(\lambda_1)$ & $26$ & $1$\\
 & $L(\lambda_4)$ & $26$ & $2$\\
 & $L(0)$ & $1$ & $1$\\
\hline
\multirow{2}{*}{$p = 3$} & $L(\lambda_3)$ & $196$ & $1$\\
 & $L(\lambda_1)$ & $52$ & $2$\\
\hline
\multirow{2}{*}{$p > 3$} & $L(\lambda_3)$ & $273$ & $1$\\
 & $L(\lambda_1)$ & $52$ & $1$\\
\hline
\end{tabular}
\end{table}

\begin{table}[H]
\centering
\renewcommand{\arraystretch}{1.1}
\caption{The composition factors of $A^2(V_1)$ and $A^2(V_2)$, with $G = E_6$, $V_1 = L(\lambda_1)$, and $V_2 = L(\lambda_6)$.}
\label{table:e6a2v}
\begin{tabular}{ |c|c|c|c|c| }
\hline
Condition on $p$ & {\parbox{3.0cm}{\centering \vspace{.15cm}Composition factor of $A^2(V_1)$\vspace{.15cm}}} & {\parbox{3.0cm}{\centering \vspace{.15cm}Composition factor of $A^2(V_2)$\vspace{.15cm}}} & Dimension & Multiplicity \\
\hline
\hline
\multirow{2}{*}{$p = 2$} & $L(\lambda_3)$ & $L(\lambda_5)$ & $324$ & $1$\\
 & $L(\lambda_6)$ & $L(\lambda_1)$ & $27$ & $1$\\
\hline
$p > 2$ & $L(\lambda_3)$ & $L(\lambda_5)$ & $351$ & $1$\\
\hline
\end{tabular}
\end{table}

\begin{table}[H]
\centering
\renewcommand{\arraystretch}{1.1}
\caption{The composition factors of $A^2V$, with $G = E_7$ and $V = L(\lambda_7)$.}
\label{table:e7a2v}
\begin{tabular}{ |c|c|c|c| }
\hline
Condition on $p$ & Composition factor & Dimension & Multiplicity \\
\hline
\hline
\multirow{3}{*}{$p = 2$} & $L(\lambda_6)$ & $1274$ & $1$\\
 & $L(\lambda_1)$ & $132$ & $2$\\
 & $L(0)$ & $1$ & $2$\\
\hline
\multirow{2}{*}{$p = 7$} & $L(\lambda_6)$ & $1538$ & $1$\\
 & $L(0)$ & $1$ & $2$\\
\hline
\multirow{2}{*}{$p \notin \{2,7\}$} & $L(\lambda_6)$ & $1539$ & $1$\\
 & $L(0)$ & $1$ & $1$\\
\hline
\end{tabular}
\end{table}

\begin{table}[H]
\centering
\renewcommand{\arraystretch}{1.1}
\caption{The composition factors of $A^2V$, with $G = E_8$ and $V = L(\lambda_8)$.}
\label{table:e8a2v}
\begin{tabular}{ |c|c|c|c| }
\hline
Condition on $p$ & Composition factor & Dimension & Multiplicity \\
\hline
\hline
\multirow{4}{*}{$p = 2$} & $L(\lambda_7)$ & $26504$ & $1$\\
 & $L(\lambda_1)$ & $3626$ & $1$\\
 & $L(\lambda_8)$ & $248$ & $2$\\
 & $L(0)$ & $1$ & $2$\\
\hline
\multirow{2}{*}{$p \in \{3,5\}$} & $L(\lambda_7)$ & $30132$ & $1$\\
 & $L(\lambda_8)$ & $248$ & $2$\\
\hline
\multirow{2}{*}{$p>5$} & $L(\lambda_7)$ & $30380$ & $1$\\
 & $L(\lambda_8)$ & $248$ & $1$\\
\hline
\end{tabular}
\end{table}

\section{Composition factors of third Lie powers of minimal $K[G]$-modules}
\label{sec:compkgl3v}

We now list, in Tables \ref{table:e6l3v} and \ref{table:e7l3v}, the composition factors of the module $(A^2V \otimes V)/A^3V$ for each minimal $K[G]$-module $V$ from Theorem \ref{thm:l3vkgmod}. The module $(A^2V \otimes V)/A^3V$ is isomorphic to the third Lie power $L^3V$ of $V$ when $p > 3$. In each case, the highest weights of all composition factors of $(A^2V \otimes V)/A^3V$ are comparable with respect to the partial order $\le$ defined in \S\ref{sec:highestweights}, and we list these composition factors by descending highest weight.

\begin{table}[H]
\centering
\caption{The composition factors of $U:=(A^2(V_1) \otimes V_1)/A^3(V_1)$ and $W:=(A^2(V_2) \otimes V_2)/A^3(V_2)$, with $G = E_6$, $V_1 = L(\lambda_1)$, and $V_2 = L(\lambda_6)$.}
\label{table:e6l3v}
\begin{tabular}{ |c|c|c|c|c| }
\hline
Condition on $p$ & {\parbox{2.5cm}{\centering \vspace{.15cm}Composition factor of $U$ \vspace{.15cm}}} & {\parbox{2.5cm}{\centering \vspace{.15cm}Composition factor of $W$\vspace{.15cm}}} & Dimension & Multiplicity \\
\hline
\hline
\multirow{3}{*}{$p = 2$} & $L(\lambda_1+\lambda_3)$ & $L(\lambda_5+\lambda_6)$ & $5824$ & $1$\\
 & $L(\lambda_1+\lambda_6)$ & $L(\lambda_1+\lambda_6)$ & $572$ & $1$\\
 & $L(\lambda_2)$ & $L(\lambda_2)$ & $78$ & $2$\\
\hline
\multirow{5}{*}{$p = 3$} & $L(\lambda_1+\lambda_3)$ & $L(\lambda_5+\lambda_6)$ & $2404$ & $1$\\
 & $L(\lambda_4)$ & $L(\lambda_4)$ & $2771$ & $1$\\
 & $L(\lambda_1+\lambda_6)$ & $L(\lambda_1+\lambda_6)$ & $572$ & $2$\\
 & $L(\lambda_2)$ & $L(\lambda_2)$ & $77$ & $3$\\
 & $L(0)$ & $L(0)$ & $1$ & $2$\\
\hline
\multirow{3}{*}{$p = 5$} & $L(\lambda_1+\lambda_3)$ & $L(\lambda_5+\lambda_6)$ & $5746$ & $1$\\
 & $L(\lambda_1+\lambda_6)$ & $L(\lambda_1+\lambda_6)$ & $650$ & $1$\\
 & $L(\lambda_2)$ & $L(\lambda_2)$ & $78$ & $2$\\
\hline
\multirow{3}{*}{$p > 5$} & $L(\lambda_1+\lambda_3)$ & $L(\lambda_5+\lambda_6)$ & $5824$ & $1$\\
 & $L(\lambda_1+\lambda_6)$ & $L(\lambda_1+\lambda_6)$ & $650$ & $1$\\
 & $L(\lambda_2)$ & $L(\lambda_2)$ & $78$ & $1$\\
\hline
\end{tabular}
\end{table}

\begin{table}[H]
\centering
\caption{The composition factors of $(A^2V \otimes V)/A^3V$, with $G = E_7$ and $V = L(\lambda_7)$.}
\label{table:e7l3v}
\begin{tabular}{ |c|c|c|c| }
\hline
Condition on $p$ & Composition factor & Dimension & Multiplicity \\
\hline
\hline
\multirow{4}{*}{$p = 2$} & $L(\lambda_6+\lambda_7)$ & $50160$ & $1$\\
 & $L(\lambda_1+\lambda_7)$ & $6480$ & $1$\\
 & $L(\lambda_2)$ & $912$ & $2$\\
 & $L(\lambda_7)$ & $56$ & $1$\\
\hline
\multirow{5}{*}{$p = 3$} & $L(\lambda_6+\lambda_7)$ & $24264$ & $1$\\
 & $L(\lambda_5)$ & $25896$ & $1$\\
 & $L(\lambda_1+\lambda_7)$ & $6480$ & $1$\\
 & $L(\lambda_2)$ & $856$ & $2$\\
 & $L(\lambda_7)$ & $56$ & $3$\\
\hline
\multirow{4}{*}{$p = 7$} & $L(\lambda_6+\lambda_7)$ & $51072$ & $1$\\
 & $L(\lambda_1+\lambda_7)$ & $5568$ & $1$\\
 & $L(\lambda_2)$ & $912$ & $2$\\
 & $L(\lambda_7)$ & $56$ & $1$\\
\hline
\multirow{4}{*}{$p = 11$} & $L(\lambda_6+\lambda_7)$ & $44592$ & $1$\\
 & $L(\lambda_1+\lambda_7)$ & $6480$ & $2$\\
 & $L(\lambda_2)$ & $912$ & $1$\\
 & $L(\lambda_7)$ & $56$ & $1$\\
\hline
\multirow{4}{*}{$p = 19$} & $L(\lambda_6+\lambda_7)$ & $51072$ & $1$\\
 & $L(\lambda_1+\lambda_7)$ & $6424$ & $1$\\
 & $L(\lambda_2)$ & $912$ & $1$\\
 & $L(\lambda_7)$ & $56$ & $2$\\
\hline
\multirow{4}{*}{$p \notin \{2,3,7,11,19\}$} & $L(\lambda_6+\lambda_7)$ & $51072$ & $1$\\
 & $L(\lambda_1+\lambda_7)$ & $6480$ & $1$\\
 & $L(\lambda_2)$ & $912$ & $1$\\
 & $L(\lambda_7)$ & $56$ & $1$\\
\hline
\end{tabular}
\end{table}

\section{Magma computations}
\label{sec:magmacalc}

In this appendix, we discuss important details about the computations in the Magma \cite{magma} computer algebra system involved in the proof of Theorem \ref{thm:finiteextsubmods}. Let $q$ be a power of an odd prime $p$, let $\tilde G = Y_m(q)$ be an exceptional Chevalley group, and let $\hat G$ be the simply connected version of $\tilde G$. We can construct a minimal $\mathbb{F}_q[\hat G]$-module $V$ in Magma using the command \texttt{GModule(ChevalleyGroup("Y",m,q))}, unless $(\tilde G,p) = (F_4(q),3)$, in which case $V$ is a proper submodule of the original module.

When $q = p$ is an ``exceptional prime'' for $\tilde G$, we can determine the submodule structure of $L^2V \cong A^2V$ using the \texttt{SubmoduleLattice} command. However, when $\tilde G \in \{E_7(p), E_8(p)\}$, it is significantly faster to determine this structure by iterating the \texttt{MaximalSubmodules} command. Computations using the latter method in the case $\tilde G = E_8(5)$ required a CPU time of 4 hours and 7.9 GB of RAM. Note that all computations mentioned in this appendix were performed using a 2.6 GHz CPU.

Suppose now that $p > 3$, and let $d:=\dim(V)$. Recall that $L^3V$ is isomorphic to $(A^2V \otimes V)/A^3V$. If $\{e_1,\ldots,e_d\}$ is a basis for $V$, then $$\{(e_i \wedge e_j) \otimes e_k + (e_j \wedge e_k) \otimes e_i - (e_i \wedge e_k) \otimes e_j \mid 1 \le i < j < k \le d\}$$ is a basis for the submodule $A^3V$ of $A^2V \otimes V$ \cite[p.~2938]{bamberg}. We can therefore construct $L^3V$ in Magma as the quotient of $A^2V \otimes V$ by the submodule with this basis. Note that Magma orders its exterior square and tensor product basis vectors in descending canonical order. For example, if $U$ is a $3$-dimensional module with basis $\{e_1,e_2,e_3\}$, then the ordered basis of $A^2U \otimes U$ in Magma is
\begin{align*}
\{&(e_2 \wedge e_3) \otimes e_3, (e_2 \wedge e_3) \otimes e_2, (e_2 \wedge e_3) \otimes e_1,\\
&(e_1 \wedge e_3) \otimes e_3, (e_1 \wedge e_3) \otimes e_2, (e_1 \wedge e_3) \otimes e_1,\\
&(e_1 \wedge e_2) \otimes e_3, (e_1 \wedge e_2) \otimes e_2, (e_1 \wedge e_2) \otimes e_1\}.
\end{align*}
In the case of $\tilde G = E_6(5)$, we can determine the submodule structure of $L^3V$ by iterating the \texttt{MaximalSubmodules} command. Here, the \texttt{subset} operator can be used to check that if a submodule $U$ of $L^3V$ contains a submodule $X$, and if another submodule $W$ of $L^3V$ contains a submodule $Y$ with the same dimension as $X$, then $Y=X$.

Assume now that $\tilde G = E_7(p)$ with $p \in \{7,11,19\}$, and let $R$ be a minimal $\mathbb{F}_5[\hat X]$-module, with $\hat X$ the simply connected version of $E_7(5)$. Figure \ref{fig:e7l3vsub} shows that if $n$ is the dimension of a direct summand of $L^3V$ mentioned in the proof of Theorem \ref{thm:finiteextsubmods}, then there is a unique $n$-dimensional submodule $R_n$ of $L^3R$. In order to construct the $n$-dimensional direct summand of $L^3V$, we first construct $R_n$ and its maximal submodules by iterating the \texttt{MaximalSubmodules} command. Next, we find a vector $r_n \in R_n$ such that:
\begin{enumerate}[label={(\roman*)}]
\item $r_n$ is an element of the basis for $R_n$ used by Magma;
\item $R_n$ is the smallest submodule of $L^3R$ that contains $r_n$; and
\item if $r_n = \sum_{i=1}^k \alpha_i s_i$, where $\{s_1,\ldots,s_k\}$ is the basis for $L^3R$ used by Magma, then $\alpha_i \in \{-1,0,1\}$ for each $i$.
\end{enumerate}
We can then construct the $n$-dimensional direct summand of $L^3V$ as the submodule generated by $\sum_{i=1}^k \alpha_i t_i$, where $\{t_1,\ldots,t_k\}$ is the basis for $L^3V$ used by Magma. Our computations to obtain all necessary vectors $r_n$ required a CPU time of 49 hours and 31.8 GB of RAM. Note that after constructing $L^3R$, we deleted variables that were no longer necessary in order to reduce the amount of RAM used.

To complete the proof of Theorem \ref{thm:finiteextsubmods}, we must determine the submodule structure of the reducible direct summand $U$ of $L^3V$. In the cases $p \in \{7,19\}$, this was achieved by iterating the \texttt{MaximalSubmodules} command. However, $\dim(U)$ is very large when $p = 11$. For a faster computation in this case, we computed first the irreducible submodules of $L^3V$, and then the irreducible submodules of the quotient of $U$ by the irreducible submodule of dimension $6480$. Together with the dimensions of the composition factors of $L^3V$, our computational results imply that the submodule structure of $U$ is as required.

Our computations (including the construction of $L^3V$ but excluding the construction of the vectors $r_n \in L^3R$) required a CPU time of 0.6 hours in the $E_7(7)$ case; 223 hours in the $E_7(11)$ case; and 4 hours in the $E_7(19)$ case. The maximum RAM usage in each case was 14.5 GB, 67.0 GB and 38.7 GB, respectively. As above, we deleted variables when they were no longer necessary in order to minimise the RAM usage.

\end{appendices}

\end{document}